\documentclass[12pt,british,reqno,x11names]{amsart}

\usepackage{babel}

\usepackage{amsthm,amssymb,amsmath} 
\usepackage[latin1]{inputenc}
\usepackage{tikz}
\usepackage{tikz-cd}
\usetikzlibrary{arrows,decorations.markings}
\tikzset{commutative diagrams/.cd,arrow style=tikz,diagrams={>=stealth'}}

\usepackage{enumerate}
\makeatletter
\def\@map#1#2[#3]{\mbox{$#1 \colon\thinspace #2 \to #3$}}
\def\map#1#2{\@ifnextchar [{\@map{#1}{#2}}{\@map{#1}{#2}[#2]}}
\makeatother

\newcommand{\altmap}[4][\lhra]{\mbox{$#2 \colon\thinspace #3 #1 #4$}}

\renewcommand{\to}{\longrightarrow}
\DeclareRobustCommand*{\up}[1]{\textsu{#1}}
\newcommand{\id}{\ensuremath{\operatorname{\text{Id}}}}
\renewcommand{\ker}[1]{\ensuremath{\operatorname{\text{Ker}}({#1})}}
\newcommand{\im}[1]{\ensuremath{\operatorname{\text{Im}}\left({#1}\right)}}
\newcommand{\ab}{\textsuperscript{Ab}}
\newcommand{\brak}[1]{\ensuremath{\left\{ #1 \right\}}}
\newcommand{\ang}[1]{\ensuremath{\left\langle #1\right\rangle}}
\newcommand{\set}[2]{\ensuremath{\left\{#1 \,\mid\, #2\right\}}}
\newcommand{\setang}[2]{\ensuremath{\ang{#1 \,\mid\, #2}}}
\newcommand{\setangr}[2]{\ensuremath{\ang{#1 \,\left\lvert \, #2 \right.}}}
\newcommand{\setangl}[2]{\ensuremath{\ang{\left. #1 \,\right\rvert \, #2}}}
\newcommand{\lhra}{\lhook\joinrel\longrightarrow}
\newcommand{\setr}[2]{\ensuremath{\brak{#1 \,\left\lvert \, #2 \right.}}}
\newcommand{\setl}[2]{\ensuremath{\brak{\left. #1 \,\right\rvert \, #2}}}
\newcommand{\abs}[1]{\lvert#1\rvert}

\newcommand{\req}[1]{equation~(\protect\ref{eq:#1})}
\newcommand{\reqref}[1]{(\protect\ref{eq:#1})}
\newcommand{\reth}[1]{Theorem~\protect\ref{th:#1}}
\newcommand{\relem}[1]{Lemma~\protect\ref{lem:#1}}
\newcommand{\reco}[1]{Corollary~\protect\ref{cor:#1}}
\newcommand{\repr}[1]{Proposition~\protect\ref{prop:#1}}
\newcommand{\rerem}[1]{Remark~\protect\ref{rem:#1}}
\newcommand{\rerems}[1]{Remarks~\protect\ref{rems:#1}}
\newcommand{\reexs}[1]{Examples~\protect\ref{exs:#1}}
\newcommand{\resec}[1]{Section~\protect\ref{sec:#1}}
\newtheorem{teo}{Theorem}[section]

\newtheorem{prop}[teo]{Proposition}
\newtheorem{lem}[teo]{Lemma}

\theoremstyle{remark}
\newtheorem{remark}[teo]{Remark}
\newtheorem{remarks}[teo]{Remarks}

\numberwithin{equation}{section}
\newcommand{\blankbox}[2]{%
  \parbox{\columnwidth}{\centering
    \setlength{\fboxsep}{0pt}%
    \fbox{\raisebox{0pt}[#2]{\hspace{#1}}}%
  }%
}

\usepackage[%pdftex,%
bookmarks=true,
breaklinks=true,
bookmarksnumbered = true,%     % Signets numœ≤otœ≥
colorlinks= true,%     % Liens en couleur
urlcolor= green,
anchorcolor = yellow,
citecolor=blue,%  % Couleur des liens externes
]{hyperref}                   % Utilisation de HyperTeX

\makeatletter
\renewcommand{\labelenumi}{(\theenumi)}
\renewcommand{\labelenumii}{(\theenumii)}
\renewcommand{\labelenumiii}{(\theenumiii)}
\renewcommand{\p@enumi}{}
\renewcommand{\p@enumii}{}
\renewcommand{\p@enumiii}{}
\def\@enum@{\list{\csname label\@enumctr\endcsname}%
          {\usecounter{\@enumctr}\def\makelabel##1{
\normalfont\ignorespaces\emph{{##1}~}}
\setlength{\labelsep}{3pt}
\setlength{\parsep}{2pt}
\setlength{\leftmargin}{0pt}
\setlength{\labelwidth}{0pt}
\setlength{\listparindent}{\parindent}
\setlength{\itemsep}{0pt}
\setlength{\itemindent}{0pt}
\topsep=3pt plus 1pt minus 1 pt}}
\makeatother

\newcommand{\comj}[1]{\noindent\textcolor{DarkOrchid3}{\textbf{!!!J!!!~#1}}}
\newcommand{\comd}[1]{\noindent\textcolor{blue}{\textbf{!!!D!!!~#1}}}
\newcommand{\comc}[1]{\noindent\textcolor{red}{\textbf{!!!C!!!~#1}}}

\allowdisplaybreaks

\begin{document}

\title[Splitting of Generalisations of the Fadell-Neuwirth Short exact sequence]{The splitting of Generalisations of the Fadell-Neuwirth Short exact sequence}

%\title{Crystallographic groups and flat manifolds from surface braid groups} 
\author[D.~L.~Gon\c{c}alves]{Daciberg Lima Gon\c{c}alves}
\address{Departamento de Matem\'atica - IME-USP, Rua~do~Mat\~ao~1010,
CEP:~05508-090 - S\~ao Paulo - SP - Brazil}
\email{dlgoncal@ime.usp.br}

\author[J.~Guaschi]{John Guaschi}
\address{Normandie Univ., UNICAEN, CNRS, LMNO, 14000 Caen, France}
\email{john.guaschi@unicaen.fr}

%\author[O.~Ocampo]{Oscar Ocampo}
%\address{Universidade Federal da Bahia, Departamento de Matem\'atica - IME,
%Av.~Adhemar de Barros~S/N, CEP:~40170-110 - Salvador - BA - Brazil}
%\email{oscaro@ufba.br}

\author[C.~M.~Pereiro]{Carolina de Miranda e Pereiro}
\address{Universidade Federal do Esp\'{i}rito Santo, UFES, Departamento de Matem\'{a}tica, 29075-910, Vit\'{o}ria, Esp\'{i}rito Santo, Brazil}
\email{carolina.pereiro@ufes.br}

%\author[]{NOME}
%\address{}
%\email{}

%\author[C.~M.~Pereiro]{Carolina de Miranda e Pereiro}
%\address{Universidade Federal do Esp\'{i}rito Santo, UFES, Departamento de Matem\'{a}tica, 29075-910, Vit\'{o}ria, Esp\'{i}rito Santo, Brazil}
%\email{carolinapereiro@gmail.com}

%
\date{\today}

%\keywords{Differential geometry, algebraic geometry}

\begin{abstract} 
\noindent We study some generalisations to mixed braid groups of the Fadell-Neuwirth short exact sequence and the possible splitting of this sequence. In certain cases, we determine conditions under which the projection from the mixed braid group $B_{n_{1},\ldots,n_{k}}(M)$ to $B_{n_{1},\ldots, n_{k-q}}(M)$ admits a section, where $M$ is either the torus or the Klein bottle, $n_{1}, \ldots, n_{k},q \in \mathbb{N}$, and $1\leq q \leq k-1$. For $k\geq 2$ and $q=k-1$, we show that this projection admits a section if and only if $n_{1}$ divides $n_{i}$ for all $i=2,\ldots, k$. We present some partial conclusions in the case $k\geq 3$ and $q=1$. To obtain our results, we compute and make use of suitable mixed braid groups of $M$, as well as certain key quotients that play a central r\^{o}le in our analysis.
%\noindent  {\textcolor{blue}{ NOVA VERSAO 7junho -  We study some generalisations of Fadell-Neuwirth short exact sequence. Specifically, we examine in certain cases  conditions under which the projection from the mixed braid $B_{n_{1},\ldots,n_{k}}(M)$ to $B_{n_{1},\ldots, n_{k-q}}(M)$ admits a section, when $M$ is either the torus or the Klein bottle. For $k\geq 2$ and $q=k-1$ we obtain a complete characterization that the projection splits if and only if $n_{1}$ divides $n_{i}$, for all $i=2,\ldots, k$. In the case of $k\geq 3$ and $q=1$ we present some partial results. To obtain these results, we compute and make use of the suitable mixed braid groups of $M$, along with certain key quotients that play a central role in our analysis.}}
%\comj{Describe a little more precisely the results?}
\end{abstract}

\maketitle

\section{Introduction}

The braid groups of the disc, also known as Artin braid groups, were introduced by E.~Artin~\cite{Ar0}. If $n\in \mathbb{N}$, the $n$-string Artin braid group, denoted by $B_{n}$, is generated by the elements $\sigma_{1},\ldots,\sigma_{n-1}$, illustrated in~Figure~\ref{fig:artin}, and known as the Artin generators of $B_n$, that are subject to the Artin relations:
\begin{equation*}
\begin{cases}
\sigma_{i}\sigma_{i+1}\sigma_{i}=\sigma_{i+1}\sigma_{i}\sigma_{i+1}& \text{for all $1\leq i \leq n-2$}\\ 
\sigma_{i}\sigma_{j}=\sigma_{j}\sigma_{i}&\text{if $1\leq i,j\leq n-1$ and $\lvert i-j\rvert \geq2$.}
\end{cases}
\end{equation*}
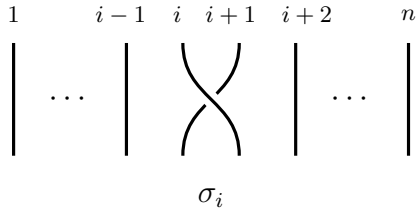
\begin{figure}[h]%[!h]
\hfill
\begin{tikzpicture}[scale=0.75, very thick]

\foreach \k in {5}
{\draw (\k,3) .. controls (\k,2) and (\k-1,2) .. (\k-1,1);}

\foreach \k in {4}
{\draw[white,line width=6pt] (\k,3) .. controls (\k,2) and (\k+1,2) .. (\k+1,1);
\draw (\k,3) .. controls (\k,2) and (\k+1,2) .. (\k+1,1);}

%%\foreach \k in {15}
%{\draw (\k,3) .. controls (\k,2) and (\k+1,2) .. (\k+1,1);};
%
%%\foreach \k in {16}
%{\draw[white,line width=6pt] (\k,3) .. controls (\k,2) and (\k-1,2) .. (\k-1,1);
%\draw (\k,3) .. controls (\k,2) and (\k-1,2) .. (\k-1,1);};

\foreach \k in {1,3,6,8}
{\draw (\k,1)--(\k,3);}

\foreach \k in {2,7}
{\node at (\k,2) {$\cdots$};}

\foreach \k in {1}
{\node at (\k,3.5) {{\scriptsize$1$}};
\node at (\k+1.9,3.5) {{\scriptsize$i-1$}};
\node at (\k+2.9,3.52) {{\scriptsize$i$}};
\node at (\k+3.85,3.5) {{\scriptsize$i+1$}};
\node at (\k+5.2,3.5) {{\scriptsize$i+2$}};
\node at (\k+7,3.5) {{\scriptsize$n$}};}

\node at (4.5,0.25) {$\sigma_{i}$};
%\node at (15.5,0.25) {$\sigma_{i}^{-1}$};

\end{tikzpicture}
\hspace*{\fill}
\caption{The Artin generator $\sigma_{i}$}\label{fig:artin}
\end{figure}

These groups were later generalised by Fox and Neuwirth using configuration spaces as follows~\cite{FoN}. Let $M$ be a connected surface, and let $n\in \mathbb{N}$. The \textit{$n$\textsuperscript{th} configuration space of $M$}, denoted by $F_{n}(M)$, is defined by:
\begin{equation*}
F_{n}(M)=\left\{(x_{1},\ldots,x_{n})\,:\; \text{$x_{i}\in M$, and $x_{i}\neq x_{j}$ if $i\neq j$, $i,j=1,\ldots,n$}\right\}.
\end{equation*}
The \textit{$n$-string pure braid group $P_{n}(M)$} of $M$ is defined by $P_{n}(M)=\pi_{1}(F_{n}(M))$. The symmetric group $S_{n}$ on $n$ letters acts freely on $F_{n}(M)$ by permuting coordinates, and the \textit{$n$-string (full) braid group $B_{n}(M)$} of $M$ is defined by $B_{n}(M)=\pi_{1}(D_{n}(M))$, where $D_{n}(M)=F_{n}(M)/S_{n}$. This gives rise to the following short exact sequence:
\begin{equation}\label{eq:seq1} 
1\longrightarrow P_{n}(M)\longrightarrow B_{n}(M)\longrightarrow S_{n}\longrightarrow 1.
\end{equation}
If $M$ is the $2$-disc (or the plane $\mathbb{R}^2$), then $B_n(M)$ (resp.\ $P_n(M)$) is isomorphic to $B_n$ (resp.\ to the Artin pure braid group $P_n$). If $M$ is a compact surface without boundary, by the work of Fadell and Neuwirth~\cite{FN}, the projection $p\colon\thinspace F_{n+m}(M)\longrightarrow F_{n}(M)$ defined by $p(x_{1},\ldots,x_{n+m})=(x_{1},\ldots,x_{n})$ for all $(x_{1},\ldots,x_{n+m})\in F_{n+m}(M)$ is a locally-trivial fibration whose fibre may be identified with $F_{m}(M\setminus\left\{x_{1}, \ldots,x_{n}\right\})$. Taking the associated long exact sequence in homotopy of this fibration, we obtain the \emph{Fadell-Neuwirth short exact sequence}:\begin{equation}\label{eq:seqFN} 
1\longrightarrow P_{m}(M\setminus\left\{x_{1},\ldots,x_{n}\right\})\longrightarrow P_{n+m}(M)\stackrel{p_{\ast}}\longrightarrow P_{n}(M) \longrightarrow 1,
\end{equation}
where $p_{\ast}$ is the homomorphism induced by $p$, $m\geq 1$, and $n\geq 3$ if $M=\mathbb{S}^{2}$~\cite{Fa,FVB}, $n\geq 2$ if $M$ is the real projective plane $\mathbb RP^{2}$~\cite{VB}, and $n\geq 1$ otherwise~\cite{FN}. Geometrically, the homomorphism $p_{\ast}$ may be interpreted as the map that forgets the last $m$ strings of a pure braid with $m+n$ strings. If $M$ has boundary,  $p$ is not a fibration, but the short exact sequence~(\ref{eq:seqFN}) nevertheless exists~(see for example the proof of~\cite[Theorem~2(a)]{GG4}). The sequence~(\ref{eq:seqFN}) is an important tool in the study of surface braid groups. Its use leads to presentations of the corresponding pure braid groups, and it allows us to compute their centre and their possible torsion elements, and to analyse their residual properties. In the case of the Artin pure braid groups,~(\ref{eq:seqFN}) splits, and gives rise to a decomposition of $P_n$ as a repeated semi-direct product of free groups, known as the Artin `combing' operation~\cite{Ar1}. This is the principal result of Artin's classical theory of braid groups, from which one may obtain normal forms and a solution to the word problem in $B_n$. One of the principal problems regarding~(\ref{eq:seqFN}), known as the \emph{splitting problem}, is to determine for which surfaces and which values of $n$ and $m$ the sequence splits~\cite{B1,FN,FVB,GG,GG1,GG4,VB}. If~(\ref{eq:seqFN}) splits for all $n, m \in\mathbb N$, the group $P_{n+m}(M)$ may be decomposed as an iterated semi-direct product, which aids in the study of its properties. In contrast with the case of compact surfaces without boundary of higher genus, in the cases where $M$ is the $2$-torus $\mathbb{T}$ or the Klein bottle $\mathbb{K}$, the fibration $p$ admits a cross-section arising from the existence of a non-vanishing vector field on $M$ for all $n,m\in \mathbb{N}$~\cite{FN}, which gives rise to an algebraic section for $p_{\ast}$. Recall that if the fibre of the fibration is an Eilenberg-MacLane space, which is the case here, then the existence of a section for $p_{\ast}$ is equivalent to that of a cross-section for $p$~\cite{Ba,GGagt,Wh}.

With respect to the splitting problem, it is natural to study the corresponding full braid groups. Although the short exact sequence~(\ref{eq:seqFN}) does not generalise directly to $B_{n+m}(M)$ directly, the projection $p_{\ast}$ extends to an intermediate subgroup of $B_{n+m}(M)$, namely the \emph{mixed braid group} $B_{n,m}(M)$ that is defined by $B_{n,m}(M)= \pi_{1}(D_{n,m}(M))$, where $D_{n,m}(M)=F_{n+m}(M)/(S_{n}\times S_{m})$. In this case, if $M$ is a compact surface without boundary, the map $p \colon\thinspace F_{n+m}(M)/(S_{n}\times S_{m}) \to F_{n}(M)/S_{n}$ given by forgetting the last $m$ coordinates is a fibration whose fibre may be identified with $F_{m}(M\setminus\left\{x_{1},\ldots ,x_{n}\right\})/S_{n}$. As in the case of the pure braid groups, applying the associated long exact sequence in homotopy to this fibration, we obtain the \emph{generalised Fadell-Neuwirth short exact sequence:}
\begin{equation}\label{eq:gen.seqFN} 
1\longrightarrow B_{m}(M\setminus\left\{x_{1},\ldots,x_{n}\right\})\longrightarrow B_{n,m}(M) \stackrel{p_{\ast}}{\longrightarrow} B_{n}(M)\longrightarrow 1,
\end{equation}
where $p_{\ast}$ is the homomorphism induced by $p$, $m\geq 1$, and $n\geq 3$ if $M=\mathbb{S}^{2}$, $n\geq 2$ if $M=\mathbb RP^{2}$, and $n\geq 1$ otherwise. Once more, the short exact sequence~(\ref{eq:gen.seqFN}) exists even if $M$ has boundary. We are interested in deciding whether this sequence splits. Once more, the existence of a section for $p_{\ast}$ is equivalent to that of a cross-section for $p$. In the case of the Artin braid groups, it is easy to see that~(\ref{eq:gen.seqFN}) splits for all $n$ and $m$. The case where $M=\mathbb{S}^{2}$ was originally studied in~\cite{GG0}, with further results being obtained in~\cite{CS}, and the case where $M=\mathbb{R}P^{2}$ has been analysed in~\cite{Ma}. The case of orientable surfaces has been studied recently in~\cite{Ma2}. In this paper, we solve the splitting problem with respect to~(\ref{eq:gen.seqFN}) for the cases where $M=\mathbb{T}$ or $M=\mathbb{K}$, the precise statement being as follows.

\begin{teo}\label{th:FNsplits}
Let $M$ be the $2$-torus or the Klein bottle. Then the generalised Fadell-Neuwirth short exact sequence~(\ref{eq:gen.seqFN}) splits if and only if $n$ divides $m$.
\end{teo}

Observe that \reth{FNsplits} implies the result of~\cite[Theorem~1]{Ma2} in the case where $M=\mathbb{T}$. 

To prove that the condition of the statement of Theorem~\ref{th:FNsplits} is sufficient, we make use of the existence of a non-vanishing vector field on $\mathbb{T}$ and $\mathbb{K}$ to construct a geometric section on the level of the associated configuration spaces, which implies the existence of an algebraic section for~(\ref{eq:gen.seqFN}). The proof of the converse is algebraic in nature, and for this we determine presentations of the groups appearing in~(\ref{eq:gen.seqFN}) as well as some of their quotients.

The mixed braid groups defined above may be generalised to any number of factors. To do so, for $k,n_{1},\ldots,n_{k}\in \mathbb{N}$, let: 
\begin{equation*}
B_{n_1,\ldots,n_k}(M)=\pi_{1}(D_{n_1,\ldots,n_{k}}(M)),
\end{equation*}
where $D_{n_1,\ldots,n_{k}}(M)=F_{n_1+\cdots +n_{k}}(M)/(S_{n_1}\times\cdots\times S_{n_{k}})$. One may obtain short exact sequences similar to that of~(\ref{eq:gen.seqFN}) by forgetting one, or several blocks, of strings. More precisely, if $q=1,\ldots, k-1$, then there exists a short exact sequence:
\begin{equation}\label{eq:p} 
1 \to \ker{p_{\ast}} \to B_{n_1,\ldots,n_{k}}(M) \stackrel{p_{\ast}}{\longrightarrow} B_{n_1,\ldots, n_{k-q}}(M) \to 1,
\end{equation}
where $p_{\ast}$ is the homomorphism induced by the map $p\colon\thinspace D_{n_1,\ldots,n_{k}}(M) \to D_{n_1,\ldots,n_{k-q}}(M)$ that forgets the last $n_{k-q+1}+\cdots +n_{k}$ points, and where $\ker{p_{\ast}}$ may be identified with the group $B_{n_{k-q+1},\ldots,n_{k}}(M\setminus \{x_{1},\ldots, x_{n_{1}+\cdots+n_{k-q}}\})$. Once more, our aim is to decide when~(\ref{eq:p}) splits. As in the situation of Theorem~\ref{th:FNsplits}, in this paper we restrict our attention to the cases where $M=\mathbb{T}$ or $M=\mathbb{K}$, in which case the existence of a splitting for~(\ref{eq:p}) is equivalent to that of a geometric section on the level of the corresponding configuration spaces. % We are especially interested in extreme values of $q$, namely $q=1$ and $q=k-1$. %\comj{Mention some consequences of knowing the result in these two cases.} \comj{A question (which might also justify the previous sentence): if the sequence splits (generically) for $k-1$ then it should split for all $q$?}\textbf{!!C!! N\~{a}o sei responder genericamente, apenas para o toro e a garrafa e klein a resposta \'{e} sim, e usa a mesma ideia da demonstracao. Se $B_{n_{1},\ldots,n_{k}}\rightarrow B_{n_{1}}$ cinde, entao o que ocorre no nosso caso \'{e} que os outros $n_{i}$ sao multiplos de $n_{1}$, entao conseguimos construir os pontos distintos no espaco de configuracao, e podemos usar os mesmos pontos em cada bloco no caso geral. Mas se projecao $B_{n_{1},\ldots,n_{k}}\rightarrow B_{n_{1}}$ nao cinde, pode ocorrer qualquer coisa} 
As a first step in the resolution of this splitting problem, we analyse the extreme values of $q$, namely $q=1$ and $q=k-1$. In the latter case, we solve the problem completely (and the answer is similar to that of the case $k=2$ of Theorem~\ref{th:FNsplits}), while in the former case, we give a partial answer.

% the We will show that for $q=k-1$ the problem  is completed solved, and for  $q=1$ a partial answer is given.  The remain cases are left  for a future work.   }{\color{red} !!C!!oct02 Concordo com a frase acima }

\begin{teo}\label{th:split1}
Let $M$ be the $2$-torus or the Klein bottle. If $q=k-1$, with the above notation, the short exact sequence~(\ref{eq:p}) splits if and only if $n_1$ divides $n_i$ for all $i=2,\ldots,k$.
\end{teo}

The case $q=1$ is more subtle. We currently have the following partial result.

\begin{teo}\label{th:split2}
Let $M$ be the $2$-torus or the Klein bottle, let $k\geq 2$, and let $n_1,\ldots,n_k\in \mathbb{N}$.
%\comj{add `let $k\geq 2$, and let $n_1,\ldots,n_k\in \mathbb{N}$'?} {\color{red}!!C!!21JUL Ok, concordo}
\begin{enumerate}[(a)]
\item\label{it:split2a} If there exist $l_{1},\ldots,l_{k-1}\in\mathbb{N}$ such that $n_{k}=l_{1}n_{1}+\cdots+ l_{k-1}n_{k-1}$, then the homomorphism $p_{\ast}\colon\thinspace B_{n_{1},\ldots,n_{k}}(M) \longrightarrow B_{n_{1},\ldots,n_{k-1}}(M)$ admits a section.
\item\label{it:split2b} If the homomorphism $p_{\ast}\colon\thinspace B_{n_{1},\ldots,n_{k}}(M) \longrightarrow B_{n_{1},\ldots,n_{k-1}}(M)$ admits a section, then there exist $l_{1},\ldots,l_{k-1}\in\mathbb{Z}$ such that $n_{k}=l_{1}n_{1}+\cdots+l_{k-1}n_{k-1}$.
\end{enumerate}
\end{teo}

The obstruction that occurs in part~(\ref{it:split2b}) of Theorem~\ref{th:split2} to proving the converse of part~(\ref{it:split2a}) is that our (algebraic) methods do not allow us to decide whether the integers $l_{1},\ldots,l_{k-1}$ are positive. In theory, some of these integers could be negative, but in that case, the section does not arise as the induced homomorphism of a cross-section for the map $p$. 
% difficulty in the converse of the above theorem is that our algebraic methods do not guarantee if the numbers are positive or negative. 
However, the following result shows that in one of the simplest such situations, where $k=3$, $q=1$, $n_1,n_2\geq 2$, $n_{3}=1$, and $M=\mathbb T$ or $\mathbb K$, the converse of Theorem~\ref{th:split2}(\ref{it:split2a}) holds, and if $n_1$ and $n_2$ are coprime then the conclusion of Theorem~\ref{th:split2}(\ref{it:split2b}) is also satisfied.

% Consider the projection $p_{\ast}$~(\ref{eq:p}), with $k=3$, $q=1$ and $n_{3}=1$. In the specific case that $M$ is the torus, we have the equivalence, which make us conjecture that the converse in Theorem~\ref{th:split2} will be true. 

\begin{teo}\label{th:split3}
Let $M$ be either the $2$-torus or the Klein bottle, and let $t,s\geq 2$. Then the projection $p_{\ast}\colon\thinspace B_{t,s,1}(M)\longrightarrow B_{t,s}(M)$ does not admit a section.
% !!!C!!!oct02 nova versao do enunciado, incluindo o caso $M=\mathbb K$.}\textcolor{purple}{!!!D!!!Oct01 Concordo}
\end{teo}

This gives some evidence to support the conjecture that the converse to Theorem~\ref{th:split2}(\ref{it:split2a}) holds in general. Note that if either $t=1$ or $s=1$ then $p_{\ast}\colon\thinspace B_{t,s,1}(M)\longrightarrow B_{t,s}(M)$ admits a section by Theorem~\ref{th:split2}(\ref{it:split2a}). 

\bigskip

\noindent
\textbf{Acknowledgements:}
the first author was  partially supported by the FAPESP Projeto Tem\'atico-FAPESP Topologia Alg\'ebrica, Geom\'etrica e Diferencial no.~2022/16455-6 (S\~ao Paulo-Brazil). His mission to the Laboratoire de Math\'ematiques Nicolas Oresme, Universit\'e de Caen Normandie during the period 14th--27th of May 2023 was also supported by the French-Brazilian network (R\'eseau Franco-Br\'esilien de Math\'ematiques).
% The second  author..... ??????? 
%\comj{Is the following still relevant?} The third  author was supportedby project grant no 2010/18930-6 and 2012/01740-5 from FAPESP.
%}

\section{Group presentations}
 
In this section, we give presentations of some of the groups that will be studied in this paper. If $M=\mathbb{T}$ or $\mathbb{K}$, we will make use of the presentations of $P_{n}(M)$ and $B_{n}(M)$ that appeared in~\cite{CMP} and~\cite[Theorems~2.1 and~2.2]{GP}. Geometric representatives of the generators of $P_n(M)$ are illustrated in Figure~\ref{fig:geradores}, where the figures represent the projection of the braids onto $M$, so that the constant paths in each figure correspond to vertical strings of the braid.
\begin{figure}[ht!]%[!h]
\hfill
\begin{tikzpicture}[scale=1, very thick]

{\draw[line width=1pt][->>] (0,1)-- (-1,0.5);}
{\draw[line width=1pt][-] (0,1)-- (-2,0);}
{\draw[line width=1pt][-<<] (0,-1) -- (1,-0.5);}
{\draw[line width=1pt][-] (2,0) -- (0,-1);}
{\draw[line width=1pt][-<] (0,1) -- (1,0.5);}
{\draw[line width=1pt][-] (2,0) -- (0,1);}
{\draw[line width=1pt][>-] (-1,-0.5) -- (-2,0);}
{\draw[line width=1pt][-] (0,-1) -- (-2,0);}

{\draw (0,0) circle (0.2mm);}
{\draw (0.3,0) circle (0.2mm);}
{\draw (0.6,0) circle (0.2mm);}
{\draw (-0.3,0) circle (0.2mm);}
{\draw (-0.6,0) circle (0.2mm);}

\draw (0,0) -- (1,0.5);
\draw[->>] (-1,-0.5) -- (-0.2,-0.1);

\node at (0,0.25) {$i$};
\node at (0,-1.5) {$a_{i}$};
{\draw[line width=1pt][->>] (5,1) -- (4,0.5);}
{\draw[line width=1pt][-] (5,1) -- (3,0);}
{\draw[line width=1pt][-<<] (5,-1) -- (6,-0.5);}
{\draw[line width=1pt][-] (7,0) -- (5,-1);}

{\draw[line width=1pt][-<] (5,1) -- (6,0.5);}
{\draw[line width=1pt][-] (7,0) -- (5,1);}
{\draw[line width=1pt][>-] (4,-0.5) -- (3,0);}
{\draw[line width=1pt][-] (5,-1) -- (3,0);}

%{\draw[line width=1pt][->] (7,0) -- (5,1);}
%{\draw[line width=1pt][->] (5,-1) -- (3,0);}
{\draw (5,0) circle (0.2mm);}
{\draw (5.3,0) circle (0.2mm);}
{\draw (5.6,0) circle (0.2mm);}
{\draw (4.7,0) circle (0.2mm);}
{\draw (4.4,0) circle (0.2mm);}

\draw (5,0) -- (6,-0.5);
\draw[->>] (4,0.5) -- (4.8,0.1);

\node at (5,0.25) {$i$};
\node at (5,-1.5) {$b_{i}$};

{\draw[line width=1pt][->>] (10,1) -- (9,0.5);}
{\draw[line width=1pt][-] (10,1) -- (8,0);}
{\draw[ line width=1pt][-<<] (10,-1)--(11,-0.5);}
{\draw[ line width=1pt][-] (12,0) -- (10,-1);}

{\draw[line width=1pt][-<] (10,1) -- (11,0.5);}
{\draw[line width=1pt][-] (12,0) -- (10,1);}
{\draw[line width=1pt][>-] (9,-0.5) -- (8,0);}
{\draw[line width=1pt][-] (10,-1) -- (8,0);}

%{\draw[ line width=1pt][->] (12,0) -- (10,1);}
%{\draw[ line width=1pt][->] (10,-1) -- (8,0);}
{\draw (10,0) circle (0.2mm);}
{\draw (10.5,0) circle (0.2mm);}
{\draw (11,0) circle (0.2mm);}
{\draw (9.5,0) circle (0.2mm);}
{\draw (9,0) circle (0.2mm);}

%(10,-3.5)\draw[->>][cyan] (0.5,0) .. controls (-1.5,1.35) and (-1.5,-1.32) .. (0.45,-0.09);
\draw[->>] (10.5,0) .. controls (9,1.35) and (9,-1.32) .. (10.45,-0.09);
%\draw[->>] (11,-3.5) .. controls (10.09,-3.24) .. (9.27,-3.37) .. controls (9.15,-3.5) .. (9.45,-3.76) .. controls (10,-3.78) .. (10.88,-3.59);
\node at (9.65,0.15) {$i$};
\node at (10.65,0.15) {$j$};
\node at (10,-1.5) {$C_{i,j}$};

{\draw[ line width=1pt][->>] (0,-2.5) -- (-1,-3);} %subtrair 3 e mais .5
{\draw[ line width=1pt][-] (0,-2.5) -- (-2,-3.5);}
{\draw[ line width=1pt][-] (2,-3.5) -- (0,-4.5);} 
{\draw[ line width=1pt][->>] (0,-4.5) -- (1,-4);} 

{\draw[line width=1pt][-<] (0,-2.5) -- (1,-3);}
{\draw[line width=1pt][-] (2,-3.5) -- (0,-2.5);}
{\draw[line width=1pt][>-] (-1,-4) -- (-2,-3.5);}
{\draw[line width=1pt][-] (0,-4.5) -- (-2,-3.5);}

%{\draw[ line width=1pt][->] (2,-3.5) -- (0,-2.5);}
%{\draw[ line width=1pt][->] (0,-4.5) -- (-2,-3.5);}
{\draw (0,-3.5) circle (0.2mm);}
{\draw (0.3,-3.5) circle (0.2mm);}
{\draw (0.6,-3.5) circle (0.2mm);}
{\draw (-0.3,-3.5) circle (0.2mm);}
{\draw (-0.6,-3.5) circle (0.2mm);}

\draw (0,-3.5) -- (1,-3);
\draw[->>] (-1,-4) -- (-0.2,-3.6);

\node at (0,-3.25) {$i$};
\node at (0,-5) {$a_{i}$};
{\draw[ line width=1pt][-] (5,-2.5) -- (3,-3.5);}
{\draw[ line width=1pt][->>] (5,-2.5) -- (4,-3);}
{\draw[ line width=1pt][->>] (5,-4.5) --(6,-4);}

{\draw[ line width=1pt][-] (7,-3.5) -- (5,-4.5);}

{\draw[line width=1pt][-<] (5,-2.5) -- (6,-3);}
{\draw[line width=1pt][-] (7,-3.5) -- (5,-2.5);}
{\draw[line width=1pt][>-] (4,-4) -- (3,-3.5);}
{\draw[line width=1pt][-] (5,-4.5) -- (3,-3.5);}

%
%{\draw[ line width=1pt][->] (7,-3.5) -- (5,-2.5);}
%{\draw[ line width=1pt][->] (5,-4.5) -- (3,-3.5);}
{\draw (5,-3.5) circle (0.2mm);}
{\draw (5.3,-3.5) circle (0.2mm);}
{\draw (5.6,-3.5) circle (0.2mm);}
{\draw (4.7,-3.5) circle (0.2mm);}
{\draw (4.4,-3.5) circle (0.2mm);}

\draw (5,-3.5) -- (6,-4);
\draw[->>] (4,-3) -- (4.8,-3.4);

\node at (5,-3.25) {$i$};
\node at (5,-5) {$b_{i}$};

{\draw[ line width=1pt][-] (10,-2.5) -- (8,-3.5);}
{\draw[ line width=1pt][->>] (10,-2.5) -- (9,-3);}
{\draw[ line width=1pt][->>] (10,-4.5) -- (11,-4);}
{\draw[ line width=1pt][-] (12,-3.5) -- (10,-4.5);}

{\draw[line width=1pt][-<] (10,-2.5) -- (11,-3);}
{\draw[line width=1pt][-] (12,-3.5) -- (10,-2.5);}
{\draw[line width=1pt][>-] (9,-4) -- (8,-3.5);}
{\draw[line width=1pt][-] (10,-4.5) -- (8,-3.5);}

%{\draw[ line width=1pt][->] (12,-3.5) -- (10,-2.5);}
%{\draw[ line width=1pt][->] (10,-4.5) -- (8,-3.5);}
{\draw (10,-3.5) circle (0.2mm);}
{\draw (10.5,-3.5) circle (0.2mm);}
{\draw (11,-3.5) circle (0.2mm);}
{\draw (9.5,-3.5) circle (0.2mm);}
{\draw (9,-3.5) circle (0.2mm);}

%(10,-3.5)\draw[->>][cyan] (0.5,0) .. controls (-1.5,1.35) and (-1.5,-1.32) .. (0.45,-0.09);
\draw[->>] (10.5,-3.5) .. controls (9,-2.15) and (9,-4.82) .. (10.45,-3.59);
%\draw[->>] (11,-3.5) .. controls (10.09,-3.24) .. (9.27,-3.37) .. controls (9.15,-3.5) .. (9.45,-3.76) .. controls (10,-3.78) .. (10.88,-3.59);
\node at (9.65,-3.35) {$i$};
\node at (10.65,-3.35) {$j$};
\node at (10,-5) {$C_{i,j}$};
%\node at (10,-1.5) {$C_{i,j}$};

\node at (0,1.5) {$M=\mathbb{T}$};
\node at (0,-2) {$M=\mathbb{K}$};

\foreach \j  in {1,6,11}
{\node at (\j,-0.85) {$\alpha$};}

\foreach \j  in {-1,4,9}
{\node at (\j,0.85) {$\alpha$};}

\foreach \j  in {1,6,11}
{\node at (\j,0.85) {$\beta$};}

\foreach \j  in {-1,4,9}
{\node at (\j,-0.85) {$\beta$};}

\foreach \j  in {1,6,11}
{\node at (\j,-4.35) {$\alpha$};}

\foreach \j  in {-1,4,9}
{\node at (\j,-2.65) {$\alpha$};}

\foreach \j  in {1,6,11}
{\node at (\j,-2.65) {$\beta$};}

\foreach \j  in {-1,4,9}
{\node at (\j,-4.35) {$\beta$};}
\end{tikzpicture}
\hspace*{\fill}
\caption{The generators of $P_{n}(\mathbb{T})$ and $P_{n}(\mathbb{K})$}\label{fig:geradores}
\end{figure}
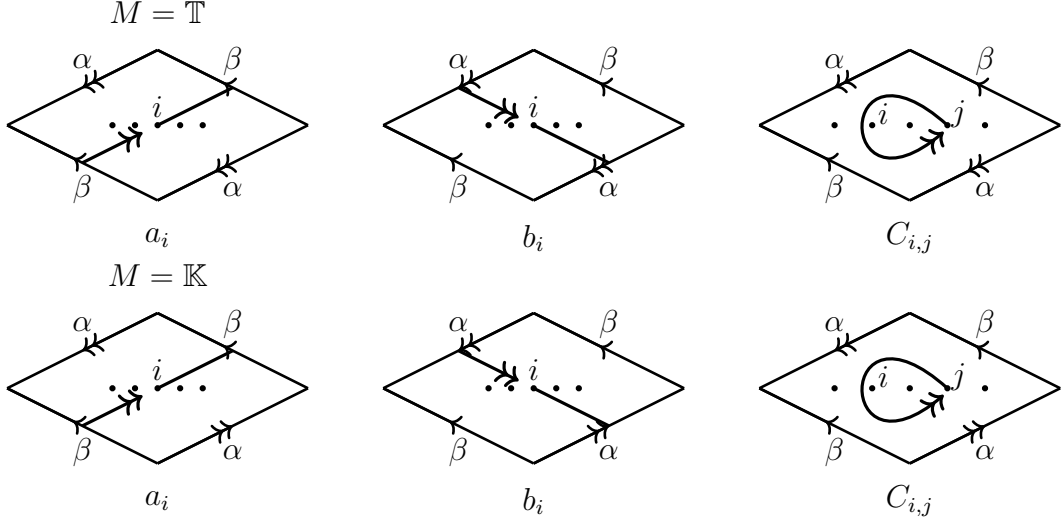
 
\begin{teo}[\cite{GP}]\label{th:puras}
Let $n\geq1$, and let $M$ be the torus $\mathbb{T}$ or the Klein bottle $\mathbb{K}$. The following constitutes a presentation of the pure braid group $P_n(M)$ of $M$:

\noindent
generating set: $\left\{a_i,\, b_i,\,i=1,\ldots,n\right\}\cup\left\{C_{i,j},\,1\leq i < j\leq n\right\}$.

\noindent
relations:
\begin{enumerate}[(1)]
%\footnotesize
\item\label{it:puras1} $a_ia_j=a_ja_i$, where $1\leq i<j\leq n$.
	
\item\label{it:puras2} $a^{-1}_{i}b_ja_i=b_ja_jC^{-1}_{i,j}C_{i+1,j}a^{-1}_{j}$, 	 where $1\leq i<j\leq n$. 
	
\item\label{it:puras3} $a^{-1}_{i}C_{j,k}a_i=\begin{cases} \text{$C_{j,k}$, where $1\leq i<j<k\leq n$ or $1\leq j<k<i\leq n$}\\
\text{$a_kC^{-1}_{i+1,k}C_{i,k}a^{-1}_{k}C_{j,k}C^{-1}_{i,k}C_{i+1,k}$, where $1\leq j\leq i<k\leq n$.}\end{cases}$
	
\item\label{it:puras4} $C^{-1}_{i,l}C_{j,k}C_{i,l}=\begin{cases} \text{$C_{j,k}$, where $1\leq i<l<j<k\leq n$ or $1\leq j\leq i<l<k\leq n$}\\
\text{$C_{i,k}C^{-1}_{l+1,k}C_{l,k}C^{-1}_{i,k}C_{j,k}C^{-1}_{l,k}C_{l+1,k}$, where $1\leq i< j\leq l<k\leq n$.}\end{cases}$
	
\item\label{it:puras5} $\begin{cases}
\text{$\prod^{n}_{j=i+1}C^{-1}_{i,j}C_{i+1,j}=a_{i}b_{i}C_{1,i}a^{-1}_{i}b^{-1}_{i}$, where $1\leq i\leq n$ and $M=\mathbb{T}$}\\
\text{$\prod^{n}_{j=i+1}C_{i,j}C^{-1}_{i+1,j}=b_{i}C_{1,i}a^{-1}_{i}b^{-1}_{i}a^{-1}_{i}$, where $1\leq i\leq n$ and $M=\mathbb{K}$.}\end{cases}$
	
\item\label{it:puras6} $b_{j}b_{i}=\begin{cases}\text{$b_{i}b_{j}$, where $1\leq i<j\leq n$ and $M=\mathbb{T}$}\\
\text{$b_{i}b_{j}C_{i,j}C^{-1}_{i+1,j}$, where $1\leq i<j \leq n$ and $M=\mathbb{K}$.}\end{cases}$
	
\item\label{it:puras7} $b^{-1}_{i}a_jb_i=\begin{cases}
\text{$a_jb_jC_{i,j}C^{-1}_{i+1,j}b^{-1}_{j}$, where $1\leq i<j\leq n$ and $M=\mathbb{T}$}\\
\text{$a_jb_j(C_{i,j}C^{-1}_{i+1,j})^{-1}b^{-1}_{j}$, where $1\leq i<j\leq n$ and $M=\mathbb{K}$.}\end{cases}$
	
\item\label{it:puras8}
$b^{-1}_{i}C_{j,k}b_i \!=\!
\begin{cases}\!
\begin{cases}\text{$C_{j,k}$, where $1\leq i<j<k\leq n$ or $1\leq j<k<i\leq n$}\\
\text{$C_{i+1,k}C^{-1}_{i,k}C_{j,k}b_kC_{i,k}C^{-1}_{i+1,k}b^{-1}_{k}$, where $1\leq j\leq i<k\leq n$}
\end{cases}  \text{and $M\!=\!\mathbb{T}$}\\[15pt]
\!\begin{cases} \text{$C_{j,k}$, where $1\leq i<j<k\leq n$ or $1\leq j<k<i\leq n$}\\
\text{$C_{i+1,k}C^{-1}_{i,k}C_{j,k}b_k(C_{i,k}C^{-1}_{i+1,k})^{-1}b^{-1}_{k}$, where $1\leq j\leq i<k\leq n$}\end{cases} \hspace*{-4mm} \text{and $M\!=\!\mathbb{K}$.}
\end{cases}$
%$\left\{\begin{array}{ll}{b^{-1}_{i}C_{j,k}b_i=\left\{\begin{array}{l} C_{j,k},\,\,  (1\leq i<j<k\leq n)\,\, \mbox{or}\,\,(1\leq j<k<i\leq n)\\
%C_{i+1,k}C^{-1}_{i,k}C_{j,k}b_kC_{i,k}C^{-1}_{i+1,k}b^{-1}_{k},\,\,  (1\leq j\leq i<k\leq n)\end{array}\right.} & \text{if}\,\, M=\mathbb{T}\\
%{b^{-1}_{i}C_{j,k}b_i=\left\{\begin{array}{l} C_{j,k},\,\,  (1\leq i<j<k\leq n)\,\, \mbox{or}\,\,(1\leq j<k<i\leq n)\\
%C_{i+1,k}C^{-1}_{i,k}C_{j,k}b_k(C_{i,k}C^{-1}_{i+1,k})^{-1}b^{-1}_{k},\,\,  (1\leq j\leq i<k\leq n)\end{array}\right.} &\text{if}\,\, M=\mathbb{K}. \end{array}\right.$
\end{enumerate}
\end{teo}

\begin{remark}\label{rem:Cij}
For $1\leq i<j\leq n$, the elements $C_{i,j}$ may be described in terms of the Artin generators of $B_n(M)$, where  $C_{i,i+1}=\sigma^2_{i}$, for $1\leq i\leq n-1$, and $C_{i,j}=\sigma_{j-1} \sigma_{j-2}\cdots\sigma_{i+1}\sigma^{2}_{i}\sigma_{i+1}\cdots \sigma_{j-2} \sigma_{j-1}$
 %$C_{i,j}=\sigma_{i}\sigma_{i+1}\cdots\sigma_{j-2}\sigma^2_{j-1} \sigma_{j-2} \cdots\sigma_{i+1} \sigma_{i}$
%  {{\bf\color{blue} !!!C!!!MAY22 a formula do $C_{i,j}$ esta errada, o correto \'{e}: $C_{i,j}=\sigma_{j-1}\sigma_{j-2}\cdots\sigma_{i+1}\sigma^{2}_{i}
%\sigma_{i+1}\cdots\sigma_{j-2}\sigma_{j-1}$}} \textcolor{red}{DMAY23 Fiz a modificacao},
for $1\leq i$ and $i+1<j\leq n$. If $i = j$, then by convention we take $C_{i,i}$ to be the trivial braid. 
%\comj{21 June: The following added. The result is used later, but it seems like a convenient place to do the calculation (Carolina did the computation before).} {\color{red}!!C!!14JUL concordo, adicionei labels} 
So if $3\leq j\leq n$ then using the Artin relations, we obtain:
\begin{align*}
C_{1,j}C^{-1}_{2,j}&=(\sigma_{j-1}\cdots\sigma_{2}\sigma_{1}^{2}\sigma_{2}\cdots\sigma_{j-1}) (\sigma^{-1}_{j-1}\cdots\sigma^{-1}_{3}\sigma^{-2}_{2}\sigma^{-1}_{3} \cdots \sigma^{-1}_{j-1})\notag\\
&=\sigma_{j-1}\cdots\sigma_{2}\sigma^{2}_{1}\sigma^{-1}_{2}\cdots\sigma^{-1}_{j-1}=\sigma^{-1}_{1}\cdots\sigma^{-1}_{j-2} \sigma^{2}_{j-1}\sigma_{j-2} \cdots\sigma_{1},\\%\label{c1}\\
\intertext{and}
C_{1,j}^{-1}C_{2,j}&=(\sigma_{j-1}^{-1} \cdots \sigma_{2}^{-1} \sigma_{1}^{-2}\sigma_{2}^{-1}\cdots\sigma_{j-1}^{-1}) (\sigma_{j-1}\cdots\sigma_{3}\sigma^{2}_{2}\sigma_{3} \cdots \sigma_{j-1})\notag\\
&=\sigma_{j-1}^{-1}\cdots\sigma_{2}^{-1}\sigma^{-2}_{1}\sigma_{2}\cdots\sigma_{j-1}=\sigma_{1}\cdots\sigma_{j-2} \sigma^{-2}_{j-1}\sigma_{j-2}^{-1} \cdots\sigma_{1}^{-1}.%\label{c2}
\end{align*}
Thus for all $2\leq k\leq n$:
\begin{align}
\prod^{k}_{j=2} C_{1,j}C^{-1}_{2,j}&=\sigma^{2}_{1}(\sigma^{-1}_{1}\sigma^{2}_{2}\sigma_{1})(\sigma^{-1}_{2}\sigma^{-1}_{2}\sigma^{2}_{3}\sigma_{2}\sigma_{1})\cdots(\sigma^{-1}_{1} \cdots \sigma^{-1}_{k-2} \sigma^{2}_{k-1}\sigma_{k-2}\cdots\sigma_{1})\notag\\
&=\sigma_{1}\cdots \sigma_{k-2} \sigma^{2}_{k-1} \sigma_{k-2} \cdots \sigma_{1},\label{eq:c1j2ja}\\
\intertext{and}
\prod^{k}_{j=2} C_{1,j}^{-1}C_{2,j} &=\sigma^{-2}_{1}(\sigma_{1}\sigma^{-2}_{2}\sigma^{-1}_{1})(\sigma_{2}\sigma_{2}\sigma^{-2}_{3}\sigma^{-1}_{2}\sigma^{-1}_{1})\cdots(\sigma_{1}\cdots \sigma_{k-2} \sigma^{-2}_{k-1} \sigma^{-1}_{k-2} \cdots \sigma^{-1}_{1})\notag\\
&=\sigma^{-1}_{1} \cdots \sigma^{-1}_{k-2} \sigma^{-2}_{k-1} \sigma^{-1}_{k-2} \cdots \sigma^{-1}_{1} = \left( \prod^{k}_{j=2} C_{1,j}C^{-1}_{2,j}\right)^{-1}. \notag%\label{eq:c1j2jb}
\end{align}
\end{remark}

\begin{teo}[\cite{GP}]\label{th:total} Let $n\geq 1$, and let $M$ be the torus $\mathbb{T}$ or the Klein bottle $\mathbb{K}$. The following constitutes a presentation of the braid group $B_n(M)$ of $M$:

\noindent
generating set: $\left\lbrace a, b, \sigma_1, \ldots, \sigma_{n-1}\right\rbrace $.

\noindent
relations:
\begin{enumerate}[(1)]
\item\label{it:full1} $\sigma_{i}\sigma_{i+1}\sigma_{i}=\sigma_{i+1}\sigma_{i}\sigma_{i+1}$ if $i=1,\ldots, n-2$.
	
\item\label{it:full2} $\sigma_{j}\sigma_{i}=\sigma_{i}\sigma_{j}$ if $1\leq i,j\leq n-1$ and $\lvert i-j\rvert \geq 2$.

\item\label{it:full3} $a\sigma_{j}=\sigma_{j}a$ if $2\leq j\leq n-1$.
	
\item\label{it:full4} $b\sigma_{j}=\sigma_{j}b$ if $2\leq j\leq n-1$.
	
\item\label{it:full5} $b^{-1}\sigma_{1}a=\sigma_{1}a\sigma_{1}b^{-1}\sigma_{1}$.

\item\label{it:full6} $a(\sigma_1a\sigma_1)=(\sigma_1a\sigma_1)a$.

\item\label{it:full7} $\begin{cases} b(\sigma^{-1}_{1}b\sigma^{-1}_{1})=(\sigma^{-1}_{1}b\sigma^{-1}_{1})b & \text{if $M=\mathbb{T}$}\\
b(\sigma^{-1}_{1}b\sigma_{1})=(\sigma^{-1}_{1}b\sigma^{-1}_{1})b & \text{if $M=\mathbb{K}$.}\end{cases}$
	
\item\label{it:full8} $\sigma_{1}\sigma_{2}\cdots\sigma_{n-2}\sigma^{2}_{n-1}\sigma_{n-2}\cdots\sigma_{2}\sigma_{1}=\begin{cases} bab^{-1}a^{-1}& \text{if $M=\mathbb{T}$}\\
ba^{-1}b^{-1}a^{-1}& \text{if $M=\mathbb{K}$.}\end{cases}$
\end{enumerate}
\end{teo}

\begin{remark}\label{rem:aba1b1}
In terms of the generators of $P_{n}(M)$, the generators $a$ and $b$ of Theorem~\ref{th:total} are equal to the generators $a_{1}$ and $b_{1}$ of Theorem~\ref{th:puras} respectively.
\end{remark}

In order to obtain a presentation of $B_m(M\setminus\left\{x_1,\ldots, x_n\right\})$, we require first a presentation of $P_m(M\setminus\left\{x_1, \ldots, x_n\right\})$. 

\begin{prop}\label{prop:prespmminusn}
Let $M=\mathbb{T}$ or $M=\mathbb{K}$, $n,m\geq 1$. The group $P_m(M\setminus\left\{x_1,\ldots,x_n\right\})$ admits the following presentation:

\noindent
generating set: $\left\lbrace a_{i},b_{i}, n+1\leq i\leq n+m\right\rbrace  \cup \left\lbrace C_{i,j},1\leq i< j,\, n+1\leq j\leq n+m\right\rbrace$.

\noindent
relations:
\begin{enumerate}[(1)]
%\footnotesize	
\item\label{it:l7} $a_ia_j=a_ja_i$, where $n+1\leq i<j\leq n+m$.
		
\item\label{it:l8} $a^{-1}_{i}b_ja_i=b_ja_jC^{-1}_{i,j}C_{i+1,j}a^{-1}_{j}$, where $n+1\leq i<j\leq n+m$.
	
\item\label{it:l9} $a^{-1}_{i}C_{j,k}a_i=\begin{cases} 
C_{j,k} & \text{if $i<j<k$ or $j<k<i$}\\
a_kC^{-1}_{i+1,k}C_{i,k}a^{-1}_{k}C_{j,k}C^{-1}_{i,k}C_{i+1,k} & \text{if  $j\leq i<k$,}\end{cases}$\\
where $n+1\leq i\leq n+m-1$, $1\leq j< k\leq n+m$ and $n+1\leq k$.
		
\item\label{it:l10} $C^{-1}_{i,l}C_{j,k}C_{i,l}=\begin{cases}
C_{j,k}& \text{if $i<l<j<k$ or $j\leq i<l<k$}\\
C_{i,k}C^{-1}_{l+1,k}C_{l,k}C^{-1}_{i,k}C_{j,k}C^{-1}_{l,k}C_{l+1,k} & \text{if $i\leq j\leq l<k$,}\end{cases}$\\
where $1\leq i,j,k,l \leq n+m$ and $n+1\leq k,l$.
	
\item\label{it:l11} $\begin{cases}
\prod^{n+m}_{j=i+1} C^{-1}_{i,j}C_{i+1,j}=a_{i}b_{i}C_{1,i}a^{-1}_{i}b^{-1}_{i} &  \text{if $n+1 \leq i\leq n+m$ and $M=\mathbb{T}$}\\
\prod^{n+m}_{j=i+1} C_{i,j}C^{-1}_{i+1,j}=b_{i}C_{1,i}a^{-1}_{i}b^{-1}_{i}a^{-1}_{i} & \text{if $n+1\leq i\leq n+m$ and $M=\mathbb{K}$.}\end{cases}$
	
\item\label{it:l12} $b_{j}b_{i}=\begin{cases} b_{i}b_{j} & \text{if $n+1\leq i<j\leq n+m$ and $M=\mathbb{T}$}\\
b_{i}b_{j} C_{i,j}C^{-1}_{i+1,j} & \text{if $n+1\leq i<j \leq n+m$ and $M=\mathbb{K}$.}\end{cases}$
	
\item\label{it:l13} $b^{-1}_{i}a_jb_i=\begin{cases}
a_jb_jC_{i,j}C^{-1}_{i+1,j}b^{-1}_{j} & \text{if $n+1\leq i<j\leq n+m$ and $M=\mathbb{T}$}\\
a_jb_j(C_{i,j}C^{-1}_{i+1,j})^{-1}b^{-1}_{j} & \text{if $n+1\leq i<j\leq n+m$ and $M=\mathbb{K}$.}\end{cases}$
	
\item\label{it:l14} $
b^{-1}_{i}C_{j,k}b_i=\begin{cases}\begin{cases} \text{$C_{j,k}$ if $i<j<k$ or $j<k<i$}\\
\text{$C_{i+1,k}C^{-1}_{i,k}C_{j,k}b_kC_{i,k}C^{-1}_{i+1,k}b^{-1}_{k}$, if $j\leq i<k$}\end{cases} \text{and $M=\mathbb{T}$}\\[15pt]
\begin{cases}
\text{$C_{j,k}$ if $i<j<k$ or $j<k<i$}\\
\text{$C_{i+1,k}C^{-1}_{i,k}C_{j,k}b_k(C_{i,k}C^{-1}_{i+1,k})^{-1}b^{-1}_{k}$ if $j\leq i<k$}\end{cases} \text{and $M=\mathbb{K}$,}\end{cases}$\\
where $n+1\leq i\leq n+m$, $1\leq j<k\leq n+m$ and $n+1\leq k$.
\end{enumerate}
\end{prop}

\begin{remark}
With respect to the short exact sequence~(\ref{eq:seqFN}) and the presentation of $P_{n+m}(M)$ given by Theorem~\ref{th:puras}, the generating set of $P_m(M\setminus\left\{x_1,\ldots,x_n\right\})$ given in Proposition~\ref{prop:prespmminusn} is obtained by taking those generators of $P_{n+m}(M)$ that belong to $P_m(M\setminus\left\{x_1,\ldots,x_n\right\})$, and the relations of $P_m(M\setminus\left\{x_1,\ldots,x_n\right\})$ are those relations of $P_{n+m}(M)$ that contain only elements of the given generating set. Another way of expressing this is that the presentation of $P_{n+m}(M)$ of Theorem~\ref{th:puras} is obtained by applying the methods of~\cite[Proposition~1, p.~139]{J} to the short exact sequence~(\ref{eq:seqFN}), where $P_m(M\setminus\left\{x_1,\ldots,x_n\right\})$ is taken to be equipped with the presentation given by Proposition~\ref{prop:prespmminusn}.
\end{remark}

\begin{proof}[Proof of Proposition~\ref{prop:prespmminusn}]
We prove the result by induction on $m\geq 1$. If $m=1$ then a generating set for $P_1(M\setminus\left\{x_1,\ldots,x_n\right\})$ is $\{ a_{n+1}, b_{n+1}, C_{i,n+1},\, 1\leq i\leq n\}$, subject to the (single) surface relation $a_{n+1} b_{n+1} C_{1,n+1} a_{n+1}^{-1} b_{n+1}^{-1}=1$ (resp.\ $b_{n+1} C_{1,n+1} a_{n+1}^{-1} b_{n+1}^{-1} a_{n+1}^{-1}=1$) if $M=\mathbb{T}$ (resp.\ $M=\mathbb{K}$) (this is relation~(\ref{it:puras5}) 
of Theorem~\ref{th:puras} in the case $i=n+1$), which is the presentation given in the statement of the proposition (in the case $m=1$, note that the only relation of~(\ref{it:l7})--(\ref{it:l14}) that exists is relation~(\ref{it:l11})). 

So let $m\geq 1$, and suppose that the presentation of the statement is valid for $P_m(M\setminus\left\{x_1,\ldots,x_n\right\})$. Making use of the short exact sequence~(\ref{eq:seqFN}) for both $M$ and $M\setminus\left\{x_1,\ldots,x_n\right\}$, we obtain the following commutative diagram of short exact sequences: 
\begin{equation}\label{eq:pmmminus}
\begin{tikzcd}[cramped, sep=scriptsize]
&  1 \arrow{d} & 1\arrow{d} &&\\
& P_{1}(M\setminus\left\{x_1,\ldots,x_{n+m}\right\}) \arrow[equal]{r} \arrow{d} & P_{1}(M\setminus\left\{x_1,\ldots,x_{n+m}\right\}) \arrow{d} & &\\
1 \arrow{r} & P_{m+1}(M\setminus\left\{x_1,\ldots,x_n\right\}) \arrow{r} \arrow{d}{p_{\ast}} & P_{n+m+1}(M) \arrow{r}{p_{\ast}} \arrow{d}{p_{\ast}} & P_{n}(M) \arrow{r} \arrow[equal]{d} & 1\\
1 \arrow{r} & P_{m}(M\setminus\left\{x_1,\ldots,x_n\right\}) \arrow{r} \arrow{d} & P_{n+m}(M) \arrow{r}{p_{\ast}} \arrow{d}  & P_{n}(M) \arrow{r} & 1,\\
&  1 & 1 && 
\end{tikzcd}
\end{equation}
where each of the homomorphisms $p_{\ast}$ is that of~(\ref{eq:seqFN}) for either $M$ or $M\setminus\left\{x_1,\ldots,x_n\right\}$, obtained by forgetting the appropriate number of strings. We now apply the methods of~\cite[Proposition~1, p.~139]{J} to the leftmost column of~(\ref{eq:pmmminus}). Taking the union of the set $\left\lbrace a_{i},b_{i}, n+1\leq i\leq n+m\right\rbrace \cup \left\lbrace C_{i,j},1\leq i< j,\, n+1\leq j\leq n+m\right\rbrace$ of coset representatives in $P_{m+1}(M\setminus\left\{x_1,\ldots,x_n\right\})$ of the given generating set of $P_{m}(M\setminus\left\{x_1,\ldots,x_n\right\})$ with the generating set $\{ a_{n+m+1}, b_{n+m+1}, C_{i,n+m+1},\, 1\leq i\leq n+m\}$ of $P_{1}(M\setminus\left\{x_1,\ldots,x_{n+m}\right\})$, we obtain the
generating set of $P_{m+1}(M\setminus\left\{x_1,\ldots,x_n\right\})$ of the statement. The corresponding relations are obtained as follows:
\begin{enumerate}[\textbullet]
\item all of the relations~(\ref{it:l7})--(\ref{it:l14}) of $P_{m}(M\setminus\left\{x_1,\ldots,x_n\right\})$ lift directly to relations of $P_{m+1}(M\setminus\left\{x_1,\ldots,x_n\right\})$, with the exception of relation~(\ref{it:l11}). We analyse the lift of this relation in $P_{m+1}(M\setminus\left\{x_1,\ldots,x_n\right\})$. Considering the inclusion of this group in $P_{n+m+1}(M)$, and using relation~(\ref{it:puras5}) of Theorem~\ref{th:puras}, for $n+1\leq i\leq n+m$, we have:
\begin{align*}
\left(\prod^{n+m}_{j=i+1} C^{-1}_{i,j} C_{i+1,j}\right)^{-1} a_{i}b_{i} C_{1,i} a^{-1}_{i} b^{-1}_{i}= C^{-1}_{i,n+m+1} C_{i+1,n+m+1} && \text{if $M=\mathbb{T}$}\\
\left(\prod^{n+m}_{j=i+1} C_{i,j}C^{-1}_{i+1,j}\right)^{-1} b_{i}C_{1,i}a^{-1}_{i}b^{-1}_{i}a^{-1}_{i} =C_{i,n+m+1}C^{-1}_{i+1,n+m+1} && \text{if $M=\mathbb{K}$.}
\end{align*}
Since the right-hand side of each of these equalities belongs to $P_{1}(M\setminus\left\{x_1,\ldots,x_{n+m}\right\})$, we obtain relation~(\ref{it:l11}) in $P_{m+1}(M\setminus\left\{x_1,\ldots, x_n\right\})$ for all $i=n+1, \ldots, n+m$. In particular, this yields relations~(\ref{it:l7})--(\ref{it:l14}) of $P_{m+1}(M\setminus\left\{x_1, \ldots, x_n\right\})$ for the possible values of the indices excluding the cases where some of the indices are equal to $n+m+1$. 

\item if $M=\mathbb{T}$ (resp.\ $M=\mathbb{K}$), the single relation $a_{n+m+1}b_{n+m+1} C_{1,n+m+1} a^{-1}_{n+m+1} b^{-1}_{n+m+1}=1$ (resp.\ $b_{i}C_{1,i}a^{-1}_{i}b^{-1}_{i}a^{-1}_{i}=1$) of $P_{1}(M\setminus\left\{x_1,\ldots,x_{n+m}\right\})$ gives rise to relation~(\ref{it:l11}) of $P_{m+1}(M\setminus\left\{x_1, \ldots,x_n\right\})$ for the case $i=n+m+1$.

\item the conjugates of the elements of the generating set $P_{1}(M\setminus\left\{x_1,\ldots,x_{n+m}\right\})$ by the coset representatives of the elements of the generating set of $P_{m}(M\setminus\left\{x_1, \ldots,x_n\right\})$. Using the corresponding relations of Theorem~\ref{th:puras}, we obtain relations~(\ref{it:l7})--(\ref{it:l10}) and~(\ref{it:l12})--(\ref{it:l14}) of the given presentation in the cases where some of the indices are equal to $n+m+1$. 
\end{enumerate}
Combining these relations, we obtain the presentation of $P_{m+1}(M\setminus\left\{x_1, \ldots,x_n\right\})$ given in the statement of the proposition.\qedhere
\end{proof}

The next step is to obtain a presentation for the group $B_m(M\setminus\left\{x_1,\ldots,x_n\right\})$ that appears in the short exact sequence~(\ref{eq:gen.seqFN}).

\begin{prop}\label{prop:kernel} Let $M=\mathbb{T}$ or $M=\mathbb{K}$, $n\geq1$ and $m\geq2$. The group $B_m(M\setminus\left\{x_1,\ldots,x_n\right\})$ admits the following presentation:

\noindent
generating set: $\left\lbrace a_{i},b_{i}, n+1\leq i\leq n+m\right\rbrace  \cup \left\lbrace C_{i,j},1\leq i< j,\, n+1\leq j\leq n+m\right\rbrace\cup\left\lbrace \sigma_{n+1},\ldots,\sigma_{n+m-1}\right\rbrace$.

\noindent
relations:
\begin{enumerate}[(1)]

\item\label{it:k0} relations~(\ref{it:l7})--(\ref{it:l14}) of Proposition~\ref{prop:prespmminusn}.

%\footnotesize	
\item\label{it:k1} $\sigma_{i}\sigma_{i+1}\sigma_{i}=\sigma_{i+1}\sigma_{i}\sigma_{i+1}$ if $i=n+1,\ldots, n+m-2$.
% {{\bf\color{blue} !!!C!!!MAY22 deveria ser $n+m-2$}}
%\textcolor{red}{DMAY23 Concordo e fiz a modificacao}	
\item\label{it:k2} $\sigma_{i}\sigma_{j}=\sigma_{j}\sigma_{i}$ if $n+1\leq i,j\leq n+m-1$ and $\lvert i-j\rvert \geq 2$.
	
\item\label{it:k3} $\sigma^{-1}_{i}a_{j}\sigma_{i}=\begin{cases} a_{j} & \text{if $j\neq i,i+1$}\\ 
\sigma^{-2}_{i}a_{i+1} & \text{if $j=i$}\\
a_{i}\sigma^{2}_{i} & \text{if $j=i+1$}\end{cases}$ where $n+1\leq i\leq n+m-1$ and $n+1\leq j\leq n+m$. %\comj{What about the case $j=i+1$?}
	
\item\label{it:k4} $\sigma^{-1}_{i}b_{j}\sigma_{i}=\begin{cases} b_{j} & \text{if $j\neq i,i+1$}\\
b_{i+1}\sigma^{2}_{i} & \text{if $j=i$}\\
\sigma^{-2}_{i}b_{i} & \text{if $j=i+1$}\end{cases}$ where $n+1\leq i\leq n+m-1$ and $n+1\leq j\leq n+m$. %\comj{What about the case $j=i+1$?}
	
\item\label{it:k5} $\sigma^{-1}_{i}C_{l,j}\sigma_{i}=\begin{cases} 
C_{l,j} & \text{if $i+1<l<j$, $l\leq i <j-1$ or $l<j<i$}\\ C^{-1}_{j,j+1}C_{l,j+1} & \text{if $i=j$}\\ 
C_{l,j-1}C_{j-1,j} & \text{if $i=j-1$}\\
C_{l-1,j}C^{-1}_{l,j}C_{l+1,j}  & \text{if $l=i+1$}\end{cases}$\\
where $n+1\leq i\leq n+m-1$, $1\leq l< j$ and $n+1\leq j \leq n+m$. %\comj{What about the case $i=j-1$?}
	
\item\label{it:k6} $C_{i,i+1}=\sigma^{2}_{i}$, where $n+1\leq i\leq n+m-1$.
\end{enumerate}
\end{prop}

\begin{proof}
We apply the methods of~\cite[Proposition~1, p.~139]{J} to the short exact  sequence~(\ref{eq:seq1}), where we replace $M$ by $M\setminus\left\lbrace x_{1},\ldots,x_{n}\right\rbrace$ and $n$ by $m$. 
%According to this method, 
A generating set of $B_{m}(M\setminus\left\lbrace x_{1},\ldots,x_{n}\right\rbrace)$ is given by the union of a generating set of $P_{m}(M\setminus\left\lbrace x_{1},\ldots,x_{n}\right\rbrace)$ with a set of coset representatives for the projection $B_m(M\setminus\left\lbrace x_{1},\ldots,x_{n}\right\rbrace)\to S_m$ of a generating set of $S_{m}$, and by Theorem~\ref{th:puras}, we may take $\left\lbrace a_{i},b_{i}, n+1\leq i\leq n+m\right\rbrace \cup \left\lbrace C_{i,j},1\leq i<j,\,n+1\leq j\leq n+m\right\rbrace$ and $\{ \sigma_{n+1},\ldots,\sigma_{n+m-1}\}$ respectively for these generating sets, which yields the generating set of the statement. %We can reduce the generators of $B_{m}(M\setminus\left\{x_{1},\ldots,x_{n}\right\})$ using that $C_{i,j}= \sigma_{j-1}\ldots\sigma^{2}_{i}\ldots\sigma_{j-1}$, if $n+1\leq i< j \leq n+m$. Now, to obtain the relations, notice that $B_{m}(M\setminus\left\{x_{1},\ldots,x_{n}\right\})\subset B_{n+m}(M)$, then the relations of $B_{m}(M\setminus\left\{x_{1},\ldots,x_{n}\right\})$ are the same relations that appears in $P_{n+m}(M)$ e $B_{n+m}(M)$ in Theorems~\ref{} and~\ref{} involving these generators.
	
The first type of relation among the elements of the given generating set is obtained by taking the relations of $P_{m}(M\setminus\left\lbrace  x_{1},\ldots,x_{n}\right\rbrace)$ given by Proposition~\ref{prop:prespmminusn}, which are relations~(\ref{it:k0}) of the statement.
%via the inclusion of $P_{m}(M\setminus\left\lbrace  x_{1},\ldots,x_{n}\right\rbrace)$ in $P_{n+m}(M)$, so they are all relations of $P_{n+m}(M)$ given in Theorem~\ref{th:puras} involving the elements $\left\lbrace a_{i},b_{i}, n+1\leq i\leq n+m\right\rbrace \cup \left\lbrace C_{i,j},1\leq i<j,\,n+1\leq j\leq n+m\right\rbrace$, which gives relations~(\ref{it:k7})--(\ref{it:k14}) of the statement \comj{Perhaps justify that this gives a complete set of relations for $P_{m}(M\setminus\left\lbrace  x_{1},\ldots,x_{n}\right\rbrace)$?}. 
We obtain the second type of relation by rewriting the relations of $S_{m}$ in terms of the given coset representatives, and expressing the corresponding element as words in terms of the generators of $B_{m}(M\setminus\left\lbrace x_{1},\ldots, x_{n}\right\rbrace)$. The group $S_{m}$ is generated by elements $s_{1},\ldots, s_{m-1}$, where for $i=1,\ldots, m-1$, $\sigma_i$ is a coset representative of $s_i$, and the generators are subject to the relations $s_{i}s_{i+1}s_{i}= s_{i+1}s_{i}s_{i+1}$ if $1\leq i \leq m-2$, $s_{i}s_{j}=s_{j}s_{i}$ if $1\leq i,j \leq m-1$ and $\lvert i-j\rvert \geq 2$, and $s^{2}_{i}=1$ if $1\leq i \leq m-1$. This yields relations~(\ref{it:k1}),~(\ref{it:k2}) and~(\ref{it:k6}) respectively. The third type of relation is obtained by writing the conjugates of the generators of the kernel by the coset representatives as words written entirely in terms of the generators of the kernel. This rewriting process may be carried out using the geometric description of the braids given in Figures~\ref{fig:artin} and~\ref{fig:geradores} (see also~\cite[equation~(5.7)]{GP}), which yields relations~(\ref{it:k3}) and~(\ref{it:k4}). We may use the Artin relations and Remark~\ref{rem:Cij} to obtain relation~(\ref{it:k5}).
\end{proof}

\begin{remark}
The presentation of $B_m(M\setminus\left\{x_1,\ldots,x_n\right\})$ given in Proposition~\ref{prop:kernel} may be simplified by eliminating some of the generators ($a_{i}$ and $b_{i}$, where $i=n+1,\ldots,n+m-1$, and $C_{i,j}$, where $n+1<i<j<n+m$, for example), but we shall not do so here. 
\end{remark}

In what follows, we will make use of certain quotients of the group $B_m(M\setminus\left\{x_1,\ldots,x_n\right\})$, one of which is described in the following proposition. If $G$ is a group, let $G\ab$ denote its Abelianisation.

\begin{prop}\label{prop:kernelab}
Let $M=\mathbb{T}$ or $M=\mathbb{K}$, and let $m\geq 2$ and $n\geq 1$. Then $B_m(M\setminus\left\{x_1,\ldots,x_n\right\})\ab$ is isomorphic to $\mathbb{Z}_{2}\oplus \mathbb{Z}^{n+1}$, where the factors of this decomposition are generated by elements $\sigma, x,y,\rho_{2},\ldots,\rho_{n}\in B_m(M\setminus\left\{x_1, \ldots,x_n\right\})\ab$ respectively, where for $i=2,\ldots,n$, $j=n+1,\ldots,n+m$ and $k=n+1,\ldots,n+m-1$, the elements $\sigma_{k}, a_{j}, b_{j}$ and $C_{i,j}$ 
%\comj{for $\sigma_{j}$, shouldn't $j=n+1,\ldots,n+m-1$?}{\textcolor{blue}{!!!C!!! concordo, veja a nova vers\~{a}o, coloquei indice $k$ para o sigma}} 
of $B_m(M\setminus\left\{x_1,\ldots,x_n\right\})$ are coset representatives  of $\sigma, x, y$ and $\rho_{i}$ respectively, and $\sigma$ is of order $2$.
\end{prop}

\begin{proof}
Let $m\geq 2$ and $n\geq 1$. To obtain a presentation of $B_m(M\setminus\left\{x_1,\ldots,x_n\right\})\ab$, we Abelianise the presentation of $B_m(M\setminus\left\{x_1,\ldots,x_n\right\})$ given in Proposition~\ref{prop:kernel}, making use of the presentation given by  Proposition~\ref{prop:prespmminusn} whose relations are relations~(\ref{it:k0}) of Proposition~\ref{prop:kernel}. By relation~(\ref{it:k1}) of that proposition, it follows that $\sigma_{j}=\sigma_{j+1}$ in $B_m(M\setminus\left\{x_1,\ldots,x_n\right\})\ab$ for all $j=n+1,\ldots,n+m-2$: we denote the coset of $\sigma_j$ by $\sigma$. By relation~(\ref{it:l8}) of Proposition~\ref{prop:prespmminusn}, we have $C^{-1}_{i,j}C_{i+1,j}=1$ in $B_m(M\setminus\left\{x_1,\ldots,x_n\right\})\ab$ for all $n+1\leq i < j\leq n+m$. In particular, if $i+1=j$, since $C_{j,j}=1$ by Remark~\ref{rem:Cij}, we see that $C_{i,i+1}=1$, and by induction we obtain $C_{i,j}=1$ for all $n+1\leq i < j\leq n+m$. It follows from relation~(\ref{it:k6}) that $\sigma^2=1$. Applying this to relations~(\ref{it:k3}) and~(\ref{it:k4}), we see that $a_{j}=a_{j+1}$ and $b_{j}=b_{j+1}$ in $B_m(M\setminus\left\{x_1,\ldots,x_n\right\})\ab$ for all $n+1\leq j \leq n+m-1$: we denote the coset of these elements by $x$ and $y$ respectively. Taking $i=j$ in relation~(\ref{it:k5}), where $n+1\leq i\leq n+m-1$, we see that $C_{l,j}=C_{l,j+1}$ for all $1\leq l \leq n$: we denote the coset of the element $C_{l,j}$ by $\rho_{l}$. By relation~(\ref{it:l11}) of Proposition~\ref{prop:prespmminusn} we have $\rho_{1}=C_{1,n+m}=1$ if $M=\mathbb{T}$, and $\rho_{1}=C_{1,n+m}=a^{2}_{n+m}=x^{2}$ if $M=\mathbb{K}$. Using the information that we have already obtained, the remaining relations of Proposition~\ref{prop:kernel} yield no new relations in $B_m(M\setminus\left\{x_1,\ldots,x_n\right\})\ab$. It follows that $B_m(M\setminus\left\{x_1,\ldots,x_n\right\})\ab$ is generated by the elements $\sigma, x,y,\rho_{1},\ldots,\rho_{n}$, subject to the relations that these elements commute pairwise, that $\sigma^2=1$, and that $\rho_{1}=1$ if $M=\mathbb{T}$, and $\rho_{1}=x^{2}$ if $M=\mathbb{K}$. We may thus remove $\rho_{1}$ from the generating set, and apart from the fact that the elements commute pairwise, the only relation is $\sigma^2=1$. The proposition then follows.
\end{proof}

\begin{remark} 
If $m=1$, then $B_1(M\setminus\left\{x_1,\ldots,x_n\right\})\ab$ is a free Abelian group of rank $n+1$.
\end{remark}

Using the same method to obtain a presentation of a group extension, the following result gives a presentation of the mixed braid group $B_{n,m}(M)$. %. Consider the short exact sequence~(\ref{eq:gen.seqFN}). Since we already know presentations for $B_m(M\setminus\left\{x_1,\ldots,x_n\right\})$ and $B_{n}(M)$, given by Proposition~\ref{prop:kernel} and by Theorem~\ref{th:total}, respectively, we can we use the extension method described in~\cite{J} to obtain a presentation for $B_{n,m}(M)$, similar as the previous results. 

\begin{prop}\label{prop:B_{n,m}}
Let $M=\mathbb{T}$ or $M=\mathbb{K}$, and let $m\geq 2$ and $n\geq 1$. Then $B_{n,m}(M)$ admits the following presentation:

\noindent
generating set: $\displaystyle \begin{array}{l} \left\lbrace a_{i},b_{i},\, n+1\leq i\leq n+m\right\rbrace \cup \left\lbrace C_{i,j},\, 1\leq i< j,\, n+1\leq j\leq n+m\right\rbrace\cup \\ \left\lbrace a,b \right\rbrace \cup\left\lbrace \sigma_{i},\, 1\leq i \leq n+m-1,\, i\neq n\right\rbrace.  \end{array}$

%$\left\lbrace a_{i},b_{i}, n+1\leq i\leq n+m\right\rbrace \cup \left\lbrace C_{i,j},1\leq i< j\leq n+m\right\rbrace\cup\left\lbrace a,b \right\rbrace \cup\left\lbrace \sigma_{i}, 1\leq i \leq n+m-1, i\neq n\right\rbrace $.

%\noindent\textbf{Type I:} $a_{n+1},\ldots,a_{n+m},b_{n+1},\ldots,b_{n+m},\sigma_{n+1},\ldots,\sigma_{n+m-1},C_{i,j}$, \footnotesize{$(1\leq i<j),\,(n+1\leq j\leq n+m)$};\normalsize
%
%\noindent\textbf{Type II:} $a,b,\sigma_{1},\ldots,\sigma_{n-1}.$

\noindent
relations:
\begin{itemize}
\item \textbf{Type I:} relations~(\ref{it:k0})--(\ref{it:k6}) of Proposition~\ref{prop:kernel}.

\item \textbf{Type II:} relations~(\ref{it:full1})--(\ref{it:full7}) of Theorem~\ref{th:total}, together with:
\begin{enumerate}[(1)]
\item\label{it:Bnm1} the surface relation: 
%\comj{21 June: I didn't see how to get the relation in the case $M=\mathbb{T}$. One problem I had is that by relation~(\ref{it:puras5}) of Theorem~\ref{th:puras}, we have $aba^{-1}b^{-1}=\prod^{n+m}_{j=2}C^{-1}_{1,j}C_{2,j}$, and the inverses are in the `wrong' places with respect to the equation below. My impression is that for $M=\mathbb{T}$, the relation should be $\prod^{m}_{i=1} C_{1,n+i}^{-1} C_{2,n+i}=(\sigma_{1}\cdots\sigma^{2}_{n-1}\cdots\sigma_{1}) aba^{-1}b^{-1}$.} {\color{red} !!C!!14JUL acho que a relacao esta certa, mas a demonstracao nao, vou adicionar outra demonstracao} \comj{19JUL: acho que est\'a certa agora.}
\begin{equation*}
\prod^{m}_{i=1} C_{1,n+i}C^{-1}_{2,n+i}=\begin{cases}
(\sigma_{1}\cdots\sigma_{n-2}\sigma^{2}_{n-1}\sigma_{n-2}\cdots\sigma_{1})^{-1}bab^{-1}a^{-1} & \text{if $M=\mathbb{T}$}\\
(\sigma_{1}\cdots \sigma_{n-2}\sigma^{2}_{n-1} \sigma_{n-2}\cdots\sigma_{1})^{-1}ba^{-1}b^{-1}a^{-1} & \text{if $M=\mathbb{K}$.}
\end{cases}
\end{equation*}	
\end{enumerate}

\item \textbf{Type III:} the conjugates of the generators of $B_m(M\setminus\left\{x_1,\ldots,x_n\right\})$ by the coset representatives of the generators of $B_{n}(M)$: 
\begin{enumerate}[(1)]
\setcounter{enumi}{1}

\item\label{it:Bnm2} for $n+1\leq j \leq n+m$, $a^{-1}a_{j}a=a_{j}$.

\item\label{it:Bnm3} for $n+1\leq j \leq n+m$, $a^{-1}b_{j}a= b_{j}a_{j}C^{-1}_{1,j}C_{2,j}a^{-1}_{j}$.

\item\label{it:Bnm5} for $n+1\leq j \leq n+m$, $b^{-1}a_{j}b=\begin{cases}
a_{j}b_{j}C_{1,j}C^{-1}_{2,j}b^{-1}_{j}& \text{if $M=\mathbb{T}$}\\
a_{j}b_{j}(C_{1,j}C^{-1}_{2,j})^{-1}b^{-1}_{j}& \text{if $M=\mathbb{K}$.}\end{cases}$

\item\label{it:Bnm4} for $n+1\leq j \leq n+m$, $b^{-1}b_{j}b= \begin{cases}
b_{j} & \text{if $M=\mathbb{T}$}\\
b_{j}C_{1,j}C^{-1}_{2,j} & \text{if $M=\mathbb{K}$.}
\end{cases}$

\item\label{it:Bnm6} for $1\leq i \leq n-1$ and $n+1\leq j \leq n+m$, $\sigma^{-1}_{i}a_{j}\sigma_{i}=a_{j}$ and  $\sigma^{-1}_{i}b_{j}\sigma_{i}=b_{j}$.

\item\label{it:Bnm8} for $1\leq i<j$, $n+1\leq j \leq n+m$, $a^{-1}C_{i,j}a=\begin{cases}
a_j C_{2,j}^{-1} C_{1,j} a_j^{-1} C_{2,j} & \text{if $i=1$}\\
C_{i,j} & \text{otherwise.}
\end{cases}$
\item\label{it:Bnm9} for $1\leq i<j$, $n+1\leq j \leq n+m$:
\begin{equation*}
b^{-1}C_{i,j}b=\begin{cases}
\left \{ \begin{aligned}
C_{2,j} b_j C_{1,j} C_{2,j}^{-1} b_j^{-1}\\ 
C_{2,j} b_j C_{2,j} C_{1,j}^{-1} b_j^{-1}
\end{aligned}\right.
& \begin{aligned}
\text{if $i=1$ and $M=\mathbb{T}$}\\
\text{if $i=1$ and $M=\mathbb{K}$}
\end{aligned}\\
C_{i,j} & \text{otherwise.}
\end{cases}
\end{equation*}

\item\label{it:Bnm7} for $1\leq i \leq n-1$ and $1\leq l<j$, $n+1\leq j \leq n+m$:
%\begin{equation*}
%\sigma^{-1}_{i}C_{l,j}\sigma_{i}=\begin{cases}
%C_{l,j} & \text{if $i+1<l<j$ or $l\leq i <j-1$ or $l<j<i$}\\ C^{-1}_{j,j+1}C_{l,j+1}& \text{if $i=j$}\\ C_{l-1,j}C^{-1}_{l,j}C_{l+1,j}& \text{if $l=i+1$.}\end{cases}
%\end{equation*} 
%\comj{For this relation, $i$ is necessarily less than $j$. So I think the above should be:
\begin{equation*}
\sigma^{-1}_{i}C_{l,j}\sigma_{i}=\begin{cases}
C_{l-1,j}C^{-1}_{l,j}C_{l+1,j}& \text{if $l=i+1$.}\\
C_{l,j} & \text{otherwise.} \end{cases}
\end{equation*}
%}

\item\label{it:Bnm10} for all $1\leq i\leq n-1$ and $n+1\leq j\leq n+m-1$, $[a,\sigma_j]= [b,\sigma_j]= [\sigma_i,\sigma_j]=1$.
\end{enumerate}
\end{itemize}
\end{prop}

\begin{proof}
Applying the methods of~\cite[Proposition~1, p.~139]{J} to the short exact sequence~(\ref{eq:gen.seqFN}), a set of generators of $B_{n,m}(M)$ is the union of the set of generators of $B_{m}(M\setminus\left\{x_1,\ldots,x_n\right\})$ given by Proposition~\ref{prop:kernel} with the set $\{a,b,\sigma_1, \ldots, \sigma_{n-1}\}$ of coset representatives for $p_{\ast}$ of the generating set of $B_n(M)$ given by Theorem~\ref{th:total}, and this is the generating set given in the statement. There are three types of relations in $B_{n,m}(M)$. The relations of Type~I are those of  $B_{m}(M\setminus\left\{x_1,\ldots,x_n\right\})$ given by Proposition~\ref{prop:kernel}. The relations of Type~II are obtained by lifting the relations~(\ref{it:full1})--(\ref{it:full8}) of $B_n(M)$ given by Theorem~\ref{th:total}, and rewriting the result in terms of the generators of $B_{m}(M\setminus\left\{x_1,\ldots,x_n\right\})$. With the exception of the surface relation~(\ref{it:full8}) of Theorem~\ref{th:total}, all of these lifted relations are also relations in $B_{n,m}(M)$. To lift this surface relation, notice that $bab^{-1}a^{-1}$ (resp.\ $ba^{-1}b^{-1}a^{-1}$) is equal to $\sigma_{1}\cdots \sigma_{n+m-2}\sigma^{2}_{n+m-1} \sigma_{n+m-2}\cdots\sigma_{1}$ in $B_{n+m}(\mathbb{T})$ (resp.\ in $B_{n+m}(\mathbb{K})$) by relation~(\ref{it:full8}) of Theorem~\ref{th:total}. Using once more this relation,
%~(\ref{it:full8}) of Theorem~\ref{th:total} 
and making use of~(\ref{eq:c1j2ja}),
%\comj{JUL19 maybe~(\ref{eq:c1j2ja}) is more appropriate?} 
%{\color{red}!!C!!21JUL Acho que esta tudo certo nessa parte}, 
we obtain:
\begin{align*}
\prod^{m}_{i=1}C_{1,n+i}C^{-1}_{2,n+i} &= \left( \prod^{n}_{i=2} C_{1,i} C^{-1}_{2,i} \right)^{-1} \sigma_{1}\cdots \sigma_{n+m-2}\sigma^{2}_{n+m-1} \sigma_{n+m-2}\cdots\sigma_{1}\\
&= \begin{cases}
(\sigma_{1}\cdots \sigma_{n-2}\sigma^{2}_{n-1}\sigma_{n-2} \cdots \sigma_{1})^{-1} bab^{-1}a^{-1} & \text{if $M=\mathbb{T}$}\\
(\sigma_{1}\cdots \sigma_{n-2}\sigma^{2}_{n-1}\sigma_{n-2} \cdots \sigma_{1})^{-1} ba^{-1}b^{-1}a^{-1} & \text{if $M=\mathbb{K}$,}
\end{cases}
\end{align*}
% $$(\sigma_{1}\cdots\sigma^{2}_{n-1}\sigma_{1})^{-1}(\sigma_{1}\cdots\sigma_{n-1}\sigma_{n}\cdots\sigma^{2}_{n+m-1}\cdots\sigma_{n}\sigma_{n-1}\cdots\sigma_{1})=\prod^{m}_{i=1}C_{1,n+i}C^{-1}_{2,n+i}.$$} 
%using Remark~\ref{rem:aba1b1} and taking $i=1$ in relation~(\ref{it:puras5}) of Theorem~\ref{th:puras}, we have $aba^{-1}b^{-1}=\prod^{n+m}_{j=2}C^{-1}_{1,j}C_{2,j}$ (resp.\ $ba^{-1}b^{-1}a^{-1} =\prod^{n+m}_{j=2}C_{1,j}C_{2,j}^{-1}$) if $M=\mathbb{T}$ (resp.\ $M=\mathbb{K}$), and the surface relation~(\ref{it:Bnm1}) of the statement follows \comj{June 21: I didn't see how to get the result in the case $M=\mathbb{T}$. Here is what I think it should be: To lift this surface relation, using Remark~\ref{rem:aba1b1} and taking $i=1$ in relation~(\ref{it:puras5}) of Theorem~\ref{th:puras}, we have $R=\prod_{j=2}^{n+m} D_{j}$ where $R=aba^{-1}b^{-1}$ and $D_j=C^{-1}_{1,j}C_{2,j}$ (resp.\ $R=ba^{-1}b^{-1}a^{-1}$ and $D_j=C_{1,j}C_{2,j}^{-1}$) if $M=\mathbb{T}$ (resp.\ $M=\mathbb{K}$). Taking $k=n$ in  equations~(\ref{eq:c1j2ja}) and~(\ref{eq:c1j2jb}), we thus obtain: 
%\begin{equation*}
%\prod_{j=n+1}^{n+m} D_{j}= \left(\prod_{j=2}^{n} D_{j}\right)^{-1}R= \begin{cases}
% (\sigma_{1}\cdots\sigma^{2}_{n-1}\cdots\sigma_{1}) R & \text{if $M=\mathbb{T}$}\\
%(\sigma_{1}\cdots\sigma^{2}_{n-1}\cdots\sigma_{1})^{-1}R & \text{if $M=\mathbb{K}$},
% \end{cases}
%\end{equation*}
which yields the surface relation~(\ref{it:Bnm1}) of the statement. 
Finally, the relations of Type~III are obtained by conjugating the generators of $B_m(M\setminus\left\{x_1,\ldots,x_n\right\})$ by the coset representatives of the generators of $B_{n}(M)$. Using once more  Remark~\ref{rem:aba1b1},  relations~(\ref{it:Bnm2})--(\ref{it:Bnm5}),~(\ref{it:Bnm8}) and~(\ref{it:Bnm9}) of the statement follow from relations~(\ref{it:puras1}),~(\ref{it:puras2}),~(\ref{it:puras6}),~(\ref{it:puras7}),~(\ref{it:puras3}) and~(\ref{it:puras8}) of Theorem~\ref{th:puras} respectively, and relations~(\ref{it:Bnm6}),~(\ref{it:Bnm7}) and~(\ref{it:Bnm10}) may be obtained geometrically using Figures~\ref{fig:artin} and~\ref{fig:geradores}.
%Since $a,b,a_{j},b_{j}$ belongs to $P_{n+m}$, the relations~(\ref{it:Bnm2}) to~(\ref{it:Bnm5}) follows from Theorem~\ref{th:puras}, and for the relations~({\ref{it:Bnm6}}) and~({\ref{it:Bnm7}}) it is straightforward.
\end{proof}

In order to prove Theorem~\ref{th:FNsplits}, we will make use of the following presentation of the quotient of $B_{n,m}(M)$ by its normal subgroup ${\Gamma_{2}(B_m(M \setminus\left\{x_1,\ldots,x_n\right\}))}$.

\begin{prop}\label{prop:B_{n,m}ab} Let $M$ be the torus or the Klein bottle, and let $m,n\geq 2$.
% \comj{and $n$?}\textcolor{red}{DMAY21 Carolina concordou colocar $n\geq 2$ e modifiquei}. 
Then the group  ${B_{n,m}(M)}/{\Gamma_{2}(B_m(M\setminus\left\{x_1,\ldots,x_n\right\}))}$ admits the following presentation:

\noindent
generators: $a,b,x,y,\sigma,\rho_{2},\ldots,\rho_{n},\sigma_{1},\ldots,\sigma_{n-1}$.

\noindent
relations:
\begin{enumerate}[(1)]
\item\label{it:s1} the surface relation $\begin{cases}(\sigma_{1}\cdots \sigma_{n-2} \sigma^{2}_{n-1} \sigma_{n-2} \cdots\sigma_{1})^{-1}bab^{-1}a^{-1}=\rho^{-m}_{2} & \text{if $M=\mathbb{T}$}\\
(\sigma_{1}\cdots \sigma_{n-2} \sigma^{2}_{n-1} \sigma_{n-2} \cdots\sigma_{1})^{-1}ba^{-1}b^{-1}a^{-1}=\rho^{-m}_{2}x^{2m}& \text{if $M=\mathbb{K}$.}\end{cases}$
%\comj{In the case $M=\mathbb{T}$, I think that the relation should be $bab^{-1}a^{-1}(\sigma_{1}\cdots\sigma_{n-2}\sigma^{2}_{n-1} \sigma_{n-2} \cdots\sigma_{1})^{-1}=\rho^{-m}_{2}$.}

\item\label{it:s2} the relations~(\ref{it:full1})--(\ref{it:full7}) of $B_{n}(M)$ given by Theorem~\ref{th:total}.
%, with the exception of the surface relation~(\ref{it:full8}).

\item\label{it:s3} $\sigma^{2}=1$.
	
\item\label{it:s4}$[x,y]=[a,x]=[x,\rho_{i}]=[y,\rho_{i}]=[a,\rho_{i}]=[b,\rho_{i}]=[x,\sigma_{j}]=[y,\sigma_{j}]=[\rho_{i},\rho_{k}]\linebreak =[\sigma,\sigma_{j}]=[\sigma,x]= [\sigma,y]= [\sigma,\rho_i]=[\sigma,a]=[\sigma,b]=1$, for all $i,k=2,\ldots,n$ and $j=1,\ldots,n-1$.
	
\item\label{it:s5} $a^{-1}ya=\begin{cases} y\rho_{2} & \text{if $M=\mathbb{T}$}\\ 
yx^{-2}\rho_{2} & \text{if $M=\mathbb{K}$.}\end{cases}$
	
\item\label{it:s6} $b^{-1}xb=\begin{cases} x\rho^{-1}_{2}& \text{if $M=\mathbb{T}$}\\ x^{-1}\rho_{2}& \text{if $M=\mathbb{K}$.} \end{cases}$
	
\item\label{it:s7} $b^{-1}yb=\begin{cases} y& \text{if $M=\mathbb{T}$}\\
yx^{2}\rho^{-1}_{2}& \text{if $M=\mathbb{K}$.} \end{cases}$
	
\item\label{it:s8} for all $i=1,\ldots, n-1$ and $j=2,\ldots,n$, $\sigma^{-1}_{i}\rho_{j} \sigma_{i}=\begin{cases} 
\rho_{j-1}\rho^{-1}_{j}\rho_{j+1} & \text{$i+1=j$}\\
\rho_{j} &\text{otherwise,}\end{cases}$ where $\rho_1=1$ (resp.\ $\rho_1=x^2$) if $M=\mathbb{T}$ (resp.\ $M=\mathbb{K}$), and $\rho_{n+1}$ is taken to be equal to $1$.
\end{enumerate}
\end{prop}

\begin{proof}
The result follows by applying the methods of~\cite[Proposition~1, p.~139]{J} to the following short exact sequence:
\begin{equation}\label{eq:sesquot}
1 \to B_m(M\setminus\left\{x_1,\ldots,x_n\right\})\ab \to B_{n,m}(M)/\Gamma_2(B_m(M\setminus \left\{x_1,\ldots,x_n\right\})) \to B_n(M) \to 1
\end{equation}
obtained from~(\ref{eq:gen.seqFN}), and using Propositions~\ref{prop:kernelab} and~\ref{prop:B_{n,m}}. 
%\comj{The proof has been rewritten slightly.} {\color{red}!!C!!21JUL Ok} 
Relations~(\ref{it:full1})--(\ref{it:full7}) of $B_{n}(M)$ given by Theorem~\ref{th:total} lift directly to $B_{n,m}(M)/\Gamma_2(B_m(M\setminus \left\{x_1,\ldots,x_n\right\}))$, and the surface relation~(\ref{it:full7}) of $B_{n}(M)$ given by Theorem~\ref{th:total} is a consequence of the surface relation~(\ref{it:Bnm1}) of Proposition~\ref{prop:B_{n,m}}, the proof of Proposition~\ref{prop:kernelab}, and the fact that $\rho_1=1$ (resp.\ $\rho_1=x^2$) if $M=\mathbb{T}$ (resp.\ $M=\mathbb{K}$). This yields relations~(\ref{it:s1}) and~(\ref{it:s2}) of the statement. Relation~(\ref{it:s3}) follows from Proposition~\ref{prop:kernelab}.
Relation~(\ref{it:s4}) of the statement is a consequence of Proposition~\ref{prop:kernelab} and relations~(\ref{it:Bnm2}) and~(\ref{it:Bnm6})--(\ref{it:Bnm10}) of Proposition~\ref{prop:B_{n,m}}, and relations~(\ref{it:s5})--(\ref{it:s8}) follow from relations~(\ref{it:Bnm3})--(\ref{it:Bnm4}) and~(\ref{it:Bnm7}) of Proposition~\ref{prop:B_{n,m}} respectively. For relation~(\ref{it:s8}) in the case $i=n-1$ and $j=n$, the element $C_{n+1,k}$, where $n+1\leq k\leq n+m$, which we take as a representative of $\rho_{n+1}$, is equal to $\sigma_{k-1}\cdots \sigma_{n-2} \sigma_{n-1}^2 \sigma_{n-2}\cdots \sigma_{k-1}$, but in $B_{n,m}(M)/\Gamma_2(B_m(M\setminus\left\{x_1,\ldots,x_n\right\}))$, this is equal to $\sigma^{2(k-n-1)}$, which in turn is equal to $1$ by relation~(\ref{it:s3}). This justifies the convention that $\rho_{n+1}=1$. 
\end{proof}

\begin{remarks}\label{rem:misc}\mbox{}
\begin{enumerate}[(a)]
\item If $m=1$, the presentation of $B_{n,m}(M)/\Gamma_2(B_m(M\setminus \left\{x_1,\ldots, x_n\right\}))$ given by Proposition~\ref{prop:B_{n,m}ab} remains valid provided we take $\sigma=1$.

%To simplify the notation in relation~(\ref{it:s8}) of Proposition~\ref{prop:B_{n,m}ab} , we are considering that $\rho_{n+1}=1$ and $\rho_{1}=1$, if $M=\mathbb{T}$, or $\rho_{1}=x^{2}$, if $M=\mathbb{K}$, by Proposition~\ref{prop:kernelab}. Also, i
\item It follows from Proposition~\ref{prop:B_{n,m}ab} that in the group $B_{n,m}(M)/\Gamma_2(B_m(M\setminus\left\{x_1,\ldots,x_n\right\}))$, the element $\sigma$ is central.

\item\label{it:miscc} Using Proposition~\ref{prop:B_{n,m}ab}, one may check that $bxb^{-1}=x\rho_2$ and $aya^{-1}=y\rho_2^{-1}$ (resp.\ $bxb^{-1}=x^{-1}\rho_2$ and $byb^{-1}=aya^{-1}=yx^2\rho_2^{-1}$),
% \comj{I added this because we need it later:} \textcolor{red}{DMAY21 Carolina agree and we should remove the comments}
that each of $x$ and $y$ commutes with $bab^{-1}a^{-1}$ (resp.\ with $ba^{-1}b^{-1} a^{-1}$) in the group $B_{n,m}(M)/\Gamma_2(B_m(M\setminus\left\{x_1,\ldots,x_n\right\}))$ if $M=\mathbb{T}$ (resp.\ $M=\mathbb{K}$), and that $\sigma_{i}\rho_{i+1} \sigma_{i}^{-1}= \rho_{i} \rho_{i+1}^{-1} \rho_{i+2}$ for $i=1,\ldots, n-1$.
\end{enumerate}
\end{remarks}

\subsection{A general framework for the existence of a section}\label{sec:contas} 
% \comj{Maybe something like  `A general framework for the existence of a section'?}\comc{ok}

% \comj{Some generalities added.}\comc{ok} 
In this section, we consider a more general framework in which the situations of Theorems~\ref{th:FNsplits} and~\ref{th:split3} may be analysed simultaneously. Let $t,m\in \mathbb N$, $s\geq 0$ and let $n=t+s$. Consider the homomorphism $\map{p_{\ast}}{B_{t,s,m}(M)}[B_{t,s}(M)]$ that geometrically forgets the final block of $m$ strings. Suppose that there exists an algebraic section $\map{\phi}{B_{t,s}(M)}[B_{t,s,m}(M)]$ for $p_{\ast}$. Let $H$ (resp.\ $H'$) be a normal subgroup of $B_{t,s,m}(M)$ (resp.\ of $B_{t,s}(M)$) such that $p_{\ast}(H)=H'$ and $\phi(H')\subset H$. 
% \begin{exs}\label{exs:2cases}\mbox{}
% \begin{enumerate}
% \item\label{it:2casesa} The situation of Section~\ref{sec:FNsplits} is represented by $s=0$, $H=\Gamma_{2}(B_{m}(M\setminus \brak{x_{1},\ldots,x_{n}}))$ and $H'=\brak{1}$.
% \item\label{it:2casesb} The situation of Section~\ref{sec:Bt,s,1} is represented by $s\neq 0$, $m=1$, $H=\Gamma_{3}(P_{n+1}(M))$ and $H'=\Gamma_{3}(P_{n}(M))$.
% \end{enumerate}
% \end{exs}
Letting $L=B_{m}(M\setminus \{x_{1},\ldots,x_{n}\})\cap H$, we thus have the following commutative diagram of short exact sequences:
\begin{equation}\label{eq:basic}
\begin{tikzcd}[cramped]
& 1 \arrow[d] & 1 \arrow[d] & 1 \arrow[d] &\\
1 \arrow[r] & L \arrow[d, hookrightarrow] \arrow{r} & H \arrow[r, shift left=0.7ex, "p_{\ast}\left\lvert_{H}\right."] \arrow[d, hookrightarrow] & H' \arrow[d, hookrightarrow] \arrow[r] \arrow[l, shift left=0.7ex, "\phi\left\lvert_{H'}\right.", dotted] & 1\\
1 \arrow[r] & B_{m}(M\setminus \{x_{1},\ldots,x_{n}\}) \arrow{r} \arrow[d] & B_{t,s,m}(M) \arrow[r, shift left=0.7ex, "p_{\ast}"] \arrow[d] & B_{t,s}(M) \arrow[d] \arrow[r] \arrow[l, shift left=0.7ex, "\phi", dotted] & 1\\
1 \arrow[r]  & B_{m}(M\setminus \{x_{1},\ldots,x_{n}\})/L \arrow[d] \arrow[r]  & B_{t,s,m}(M)/H \arrow[r, shift left=0.7ex, "\widehat{p}_{\ast}"] \arrow[d] & B_{t,s}(M)/H' \arrow[d] \arrow[r] \arrow[l, shift left=0.7ex, "\widehat{\phi}", dotted] & 1,\\
& 1  & 1 & 1  &
\end{tikzcd}
\end{equation}
where $\map{\widehat{p}_{\ast}}{B_{t,s,m}(M)/H}[B_{t,s}(M)/H']$ (resp.\ $\map{\widehat{\phi}}{B_{t,s}(M)/H'}[B_{t,s,m}(M)/H]$) is the homomorphism induced by $p_{\ast}$ (resp.\ by $\phi$). It follows from exactness and commutativity of~\reqref{basic} that the last row of the diagram splits, more precisely $\phi$ induces a section $\widehat{\phi}$ for $\widehat{p}_{\ast}$.

Let $X=\brak{n+1,\ldots, n+m}$ (resp.\ $X'=\brak{t+1,\ldots, n}$). %Note that $X'=\emptyset$ (resp.\ $X=\brak{n+1}$)
% in the case of \reexs{2cases}\ref{it:2casesa} (resp.\ in \reexs{2cases}\ref{it:2casesb}). 
In what follows, we take $B_{t,s}(M)$ to be generated by:
\begin{equation}\label{eq:gensBtsM}
\brak{a_{i},b_{i},\, i\in \brak{1}\cup X'}\cup \brak{C_{i,j}, \, 1\leq i<j, \, j\in X'} \cup \brak{\sigma_{k},\, k\in \brak{1,\ldots, n-1}\setminus \brak{t}}
\end{equation}
and $B_{m}(M\setminus \{x_{1},\ldots,x_{n}\})$ to be generated by:
\begin{equation*}
\brak{a_{i},b_{i},\, i\in X}\cup \brak{C_{i,j}, \, 1\leq i<j, \, j\in X} \cup \brak{\sigma_{k},\, k\in X \setminus\brak{n+m}},
\end{equation*}
so by the middle row of~\reqref{basic}, $B_{t,s,m}(M)$ is generated by the union of these two sets (in the case of the  first set, we take the corresponding coset representatives in $B_{t,s,m}(M)$). By abuse of notation, in what follows we will not distinguish notationally between the given generators of $B_{t,s,m}(M)$ and $B_{t,s}(M)$ and their cosets in the respective quotients $B_{t,s,m}(M)/H$ and $B_{t,s}(M)/H'$. We suppose that $H$ and $H'$ are such that the following relations hold: 
%\comj{I have not yet written down all the relations; they need to be completed, but in some sense, they are the relations common to Sections~3 and~4.2.} \comc{Acrescentei relations (i)-(iv) na remark}
\begin{enumerate}[(I)]
\item\label{eq:relnsa} in $B_{m}(M\setminus \{x_{1},\ldots,x_{n}\})/L$, $a_{i}=a_{n+1}$ and $b_{i}=b_{n+1}$ for all $i\in X$, $C_{i,j}=C_{i,n+1}$ for all $i=1,\ldots, n$ and $j \in X$, $a_{l}$ and $b_{l}$ commute with $C_{i,l}$ for all $i=1,\ldots,n$ and $l\in X'\cup \{n+1\}$,
% \comj{I am not sure that what I wrote before is correct: for example, by Proposition~\ref{prop:rel}, $a_{n+1}$ and $b_{n+1}$ do not commute (but they commute with $C_{i,n+1}$).} 
%$a_{n+1}$, $b_{n+1}$ and $C_{i,n+1}$ commute pairwise, where $i=1,\ldots,n$, 
$b_{n+1}^{-1}a_{n+1}b_{n+1}$ and $b_{n+1}a_{n+1}b_{n+1}^{-1}$ are words in $a_{n+1}$ and $C_{1,n+1}$,  $\sigma_{k}=\sigma_{n+1}$ for all $k\in X \setminus\brak{n+m}$, %(note that in \reexs{2cases}\ref{it:2casesb}, $\sigma=1$)
and $\sigma_{n+1}^{2}=1$. 
% \comj{This added:}\comc{ok} 
Let $\sigma=\sigma_{n+1}$ in $B_{m}(M\setminus \{x_{1},\ldots,x_{n}\})/L$. 

\item\label{eq:relnsb} in $B_{t,s,m}(M)/H$ and in $B_{t,s}(M)/H'$, the Artin relations hold among the $\sigma_{k}$ for $k\in \brak{1,\ldots, n-1}\setminus \brak{t}$, $a_{1}$ and $b_{1}$ commute with $\sigma_{l}$ for $l\in X' \setminus \brak{n}$, $a_{i}$ and $a_{j}$ commute for $i,j\in \brak{1}\cup X'$ (or $i,j\in \brak{1}\cup X'\cup X$ in $B_{t,s,m}(M)/H$), and for $k\in \brak{1,\ldots,n-1}\setminus \brak{t}$, and $1\leq l<n+m$:
\begin{equation}\label{rel:sigmaCij}
\sigma_{k}^{-1} C_{l,n+m} \sigma_{k}=\begin{cases}
C_{l,n+m} & \text{if $l\neq k+1$}\\
C_{l-1,n+m} C_{l,n+m}^{-1} C_{l+1,n+m} & \text{if $l= k+1$.}
\end{cases}
\end{equation}
We also have the surface relation:
\begin{equation}\label{eq:surface}
\prod^{d}_{i=1} C_{1,t+i}C^{-1}_{2,t+i} =
\begin{cases}
(\sigma_{1}\cdots\sigma^{2}_{t-1}\cdots \sigma_{1})^{-1}b_{1}a_{1}b^{-1}_{1}a^{-1}_{1} & \text{if $M=\mathbb{T}$}\\
(\sigma_{1}\cdots\sigma^{2}_{t-1}\cdots\sigma_{1})^{-1}b_{1}a^{-1}_{1}b^{-1}_{1}a^{-1}_{1} & \text{if $M=\mathbb{K}$,}
\end{cases}
\end{equation}
where $d=s$ (resp.\ $d=s+m$) in $B_{t,s}(M)/H'$ (resp.\ in $B_{t,s,m}(M)/H$).
   
\item\label{eq:relnsc} in $B_{t,s,m}(M)/H$, $\sigma$ is central.

\item\label{eq:relnsd} in $B_{t,s,m}(M)/H$, for all $j=1,\ldots,n$, $k\in X'$ and $1\leq i<k$, $C_{j,n+1}$ commutes with $C_{i,k}$.

% \comc{acrescentar item (d) acima para conseguir concluir equacao (2.32)}
\end{enumerate}

%{\color{green!50!black}{
\begin{remark}\label{rem:nova}
Let $\rho=C_{1,n+1}^{-1} C_{2,n+1}$. Since $a_{j}$ and $b_{j}$ commute with $C_{i,j}$
% commute pairwise \comj{To be checked~--~see the comment in~\reqref{relnsa} above.} 
in $B_{t,s,m}(M)/L$ for all $i=1,\ldots, n$ and $j\in X'\cup\{n+1\}$, it follows from relations~(\ref{it:puras2}),~(\ref{it:puras6}),~(\ref{it:puras7}) and~(\ref{it:puras8}) of $P_{n+m}(M)$ (considered as a subgroup of $B_{t,s,m}(M)$) of Theorem~\ref{th:puras} that the following relations hold in $B_{t,s,m}(M)/L$:
\begin{enumerate}[(i)]
\item\label{it:rn1}  $a_{1} b_{n+1} a_{1}^{-1}= b_{n+1} \rho^{-1}$ and $a_{1}^{-1} b_{n+1} a_{1}= b_{n+1}\rho$.
\item\label{it:rn2} $b_{1} a_{n+1} b_{1}^{-1}= a_{n+1} \rho$ and $b_{1}^{-1} a_{n+1} b_{1}=\begin{cases} a_{n+1} \rho^{-1} & \text{if $M=\mathbb{T}$}\\ a_{n+1}\rho & \text{if $M=\mathbb{K}$.}\end{cases}$ 
\item\label{it:rn3} $b_{1}^{-1} b_{n+1} b_{1}=\begin{cases} b_{n+1}& \text{if $M=\mathbb{T}$}\\ b_{1} b_{n+1} b_{1}^{-1} = b_{n+1}\rho^{-1}& \text{if $M=\mathbb{K}$.}\end{cases}$
\item\label{it:rn4} for $j\in X'\cup\{n+1\}$, $C_{1,j}C^{-1}_{2,j}=\begin{cases}a^{-1}_{j}b^{-1}_{1}a_{j}b_{1} & \text{if $M=\mathbb{T}$}\\(a^{-1}_{j}b^{-1}_{1}a_{j}b_{1})^{-1} & \text{if $M=\mathbb{K}$.}\end{cases}$ 
%\comc{relacao 9 da proposicao 4.5}
\end{enumerate}
\end{remark}
% }}

For the homomorphism $\map{\widehat{\phi}}{B_{t,s}(M)/H'}[B_{t,s,m}(M)/H]$ of~\reqref{basic}, we set: 
\begin{equation}\label{eq:phi1a} 
\begin{cases}
\widehat{\phi}(\sigma_{i})= \sigma_{i}\cdot a^{s_{i,1}}_{n+1} b^{s_{i,2}}_{n+1} \sigma^{s_{i,3}} C^{r_{i,1}}_{1,n+1}\cdots C^{r_{i,n}}_{n,n+1} & \text{for $i=1,\ldots,t-1,t+1,\ldots n-1$}\\
\widehat{\phi}(a_{i})= a_{i}\cdot a^{\alpha_{i,1}}_{n+1} b^{\alpha_{i,2}}_{n+1} \sigma^{\alpha_{i,3}} C^{x_{i,1}}_{1,n+1} \cdots C^{x_{i,n}}_{n,n+1} & \text{for $i=1,t+1,t+2,\ldots,n$}\\
\widehat{\phi}(b_{i})= b_{i}\cdot a^{\beta_{i,1}}_{n+1} b^{\beta_{i,2}}_{n+1} \sigma^{\beta_{i,3}} C^{y_{i,1}}_{1,n+1} \cdots C^{y_{i,n}}_{n,n+1} & \text{for $i=1,t+1,t+2,\ldots,n$,}
\end{cases}
\end{equation}
where $s_{i,j},r_{i,j},\alpha_{i,j},\beta_{i,j},x_{i,j}, y_{i,j} \in \mathbb{Z}$ for the relevant values of $i$ and $j$. 
%\comc{acrescentei a definicao de canonical form, esse texto tava na secao 3}
%{\color{green!50!black}{
If $w\in B_{t,s}(M)/H'$ is written in terms of (the coset representatives of) the generators of~\reqref{gensBtsM} then the element of $B_{t,s,m}(M)/H$ written in terms of the corresponding generators, that we also denote by $w$, satisfies $\widehat{\phi}(w)=wz$, where $z \in \ker{\widehat{p}_{\ast}}$. If $z$ is written in terms of the generators of $B_m(M\setminus\left\{x_1,\ldots,x_n\right\})/L$, the decomposition $\widehat{\phi}(w)=wz$ shall be referred to as the \emph{canonical form} of $\widehat{\phi}(w)$.
% }}

We will now take the image by $\widehat{\phi}$ of some of the relations of Theorem~\ref{th:total} to obtain relations in $B_{t,s,m}(M)/H$. This will enable us to obtain information about the coefficients that appear in~(\ref{eq:phi1a}) above. 

%We now reprove some parts of Lemma~3.4 of~[CDJ] in a more general setting that includes both cases of \reexs{2cases}, and reprove (modulo some of the lemmas of~[CDJ] concerning the coefficients) Theorems~1.1 and~1.4.

We first compute $\widehat{\phi}((\sigma_{1}\cdots \sigma_{t-2}\sigma_{t-1}^{2} \sigma_{t-2}\cdots \sigma_{1})^{-1})$ and put the resulting expression into canonical form. We start by analysing the expression $\widehat{\phi}(\sigma_{1}\cdots \sigma_{t-2}\sigma_{t-1})$. Using the fact that for $1\leq i\leq n$, $C_{i,n+1}$ commutes with $a_{n+1}$ and $b_{n+1}$,
% commute pairwise \comj{Check this, especially for $a_{n+1}$ and $b_{n+1}$.}, that $\sigma$ is central, 
and that $\sigma_{i}$ commutes with $C_{j,n+1}$ for all $1\leq i\leq t-1$ and $1\leq j\leq n$ for which $j\neq i+1$, we obtain:
\begin{equation}\label{eq:wu}
\widehat{\phi}(\sigma_{1}\cdots \sigma_{t-2}\sigma_{t-1}) = \prod_{i=1}^{t-1} \sigma_{i}\cdot a^{s_{i,1}}_{n+1} b^{s_{i,2}}_{n+1} \sigma^{s_{i,3}} C^{r_{i,1}}_{1,n+1}\cdots C^{r_{i,n}}_{n,n+1}= wu,
\end{equation}
where:
\begin{align}
w&= \sigma^{\sum_{i=1}^{t-1} s_{i,3}} C_{1,n+1}^{\sum_{i=1}^{t-1} r_{i,1}} \prod_{k=3}^{t} C_{k,n+1}^{\sum_{i=1}^{k-2} r_{i,k}} \prod_{k=t+1}^{n} C_{k,n+1}^{\sum_{l=1}^{t-1} r_{l,k}} \prod_{i=1}^{t-1} a^{s_{i,1}}_{n+1} b^{s_{i,2}}_{n+1},\; \text{and}\label{eq:defw}\\
u&=\prod_{k=1}^{t-1} \sigma_{k} C_{k+1,n+1}^{\alpha_{k}},\, \text{where $\alpha_{k}= \sum_{i=k}^{t-1} r_{i,k+1}$.}\notag
\end{align}
If $\alpha\in \mathbb Z$, let us show by induction on $1\leq i\leq t-1$ that:
\begin{equation}\label{eq:indprodsigmai}
\sigma_{1}\cdots \sigma_{i} C_{i+1,n+1}^{\alpha}= C_{1,n+1}^{\alpha} C_{2,n+1}^{-\alpha} C_{i+2,n+1}^{\alpha} \sigma_{1}\cdots \sigma_{i}. 
\end{equation}
If $i=1$ then the result follows from the relation:
\begin{equation}\label{eq:conjsigma}
\text{$\sigma_{i} C_{i+1,n+1}^{\alpha} \sigma_{i}^{-1}= C_{i,n+1}^{\alpha} C_{i+1,n+1}^{-\alpha} C_{i+2,n+1}^{\alpha}$, where $1\leq i\leq t-1$.}
\end{equation}
So suppose that~\reqref{indprodsigmai} holds for some $1\leq i\leq t-2$. Then by~\reqref{conjsigma} and the induction hypothesis, we have:
\begin{equation*}
\sigma_{1}\cdots \sigma_{i+1} C_{i+2,n+1}^{\alpha} = \sigma_{1}\cdots \sigma_{i} C_{i+1,n+1}^{\alpha} C_{i+2,n+1}^{-\alpha} C_{i+3,n+1}^{\alpha} \sigma_{i+1} = C_{1,n+1}^{\alpha} C_{2,n+1}^{-\alpha} C_{i+3,n+1}^{\alpha} \sigma_{1}\cdots \sigma_{i} \sigma_{i+1},
\end{equation*}
which proves~\reqref{indprodsigmai} for all $1\leq i\leq t-1$.

Let us prove by induction on $1\leq i\leq t-1$ that:
\begin{equation}\label{eq:indu}
u=C_{1,n+1}^{\sum_{l=1}^{i-1} \alpha_{l}} C_{2,n+1}^{-\sum_{l=1}^{i-1} \alpha_{l}} \left(\prod_{j=3}^{i+1} C_{j,n+1}^{\alpha_{j-2}} \right)\sigma_{1}\cdots \sigma_{i} C_{i+1,n+1}^{\alpha_{i}} \prod_{k=i+1}^{t-1} \sigma_{k} C_{k+1,n+1}^{\alpha_{k}}.
\end{equation}
If $i=1$ then~\reqref{indu} is just the definition of $u$. So suppose that~\reqref{indu} holds for some $1\leq i \leq t-2$. By~\reqref{indprodsigmai}, we have: 
\begin{align*}
u & =C_{1,n+1}^{\sum_{l=1}^{i-1} \alpha_{l}} C_{2,n+1}^{-\sum_{l=1}^{i-1} \alpha_{l}} \left(\prod_{j=3}^{i+1} C_{j,n+1}^{\alpha_{j-2}} \right) C_{1,n+1}^{\alpha_{i}} C_{2,n+1}^{-\alpha_{i}} C_{i+2,n+1}^{\alpha_{i}} \sigma_{1}\cdots \sigma_{i}  \prod_{k=i+1}^{t-1} \sigma_{k} C_{k+1,n+1}^{\alpha_{k}}\\
&= C_{1,n+1}^{\sum_{l=1}^{i} \alpha_{l}} C_{2,n+1}^{-\sum_{l=1}^{i} \alpha_{l}} \left(\prod_{j=3}^{i+2} C_{j,n+1}^{\alpha_{j-2}} \right) \sigma_{1}\cdots \sigma_{i+1} C_{i+2,n+1}^{\alpha_{i+1}} \prod_{k=i+2}^{t-1} \sigma_{k} C_{k+1,n+1}^{\alpha_{k}},
\end{align*}
which is \req{indu} in the case $i+1$. Taking $i=t-1$ in~\reqref{indu} and using~\reqref{indprodsigmai}, we obtain:
\begin{equation}\label{eq:indu2}
u=C_{1,n+1}^{\sum_{l=1}^{t-2} \alpha_{l}} C_{2,n+1}^{-\sum_{l=1}^{t-2} \alpha_{l}} \left(\prod_{j=3}^{t} C_{j,n+1}^{\alpha_{j-2}} \right)\sigma_{1}\cdots \sigma_{t-1} C_{t,n+1}^{\alpha_{t-1}}= v\sigma_{1}\cdots \sigma_{t-1},
\end{equation}
where:
\begin{equation}\label{eq:defv}
v=C_{1,n+1}^{\sum_{l=1}^{t-1} \alpha_{l}} C_{2,n+1}^{-\sum_{l=1}^{t-1} \alpha_{l}} \left(\prod_{j=3}^{t+1} C_{j,n+1}^{\alpha_{j-2}} \right).
\end{equation}
% \comj{In case~(a), using the fact that $r_{j,i}=0$ for $j>i$ and $r_{n,n}=0$ (see the middle of page~17) then the coefficient of $C_{2,n+1}$ at this point is equal to $\sum_{l=1}^{t-1} \alpha_{l}=\sum_{i=1}^{t-1} (r_{i,i+1}+r_{i+1,i+1})=-s_{2}$ (see equation~(3.28) and the definition of $s_{2}$ just after equation~(3.26)).}

We now analyse the expression $\widehat{\phi}(\sigma_{t-1}\sigma_{t-2}\cdots \sigma_{1})$. Using the fact that for $1\leq i\leq n$, $C_{i,n+1}$ commutes with $a_{n+1}$ and $b_{n+1}$,
% the elements $C_{i,n+1}$, where $1\leq i\leq n$, $a_{n+1}$ and $b_{n+1}$ commute pairwise \comj{And again\ldots}, 
that $\sigma$ is central, and that $\sigma_{i}$ commutes with $C_{j,n+1}$ for all $1\leq i\leq t-1$, $1\leq j\leq n$ for which $j\neq i+1$, we obtain:
\begin{equation}\label{eq:wpup}
\widehat{\phi}(\sigma_{t-1} \sigma_{t-2} \cdots \sigma_{1}) = \prod_{i=1}^{t-1} \sigma_{t-i}\cdot a^{s_{t-i,1}}_{n+1} b^{s_{t-i,2}}_{n+1} \sigma^{s_{t-i,3}} C^{r_{t-i,1}}_{1,n+1}\cdots C^{r_{t-i,n}}_{n,n+1}= w'u',
\end{equation}
where:
\begin{equation}\label{eq:defwprime}
w' = \sigma^{\sum_{i=1}^{t-1} s_{t-i,3}} \prod_{k=1}^{t-1} C_{k,n+1}^{\sum_{i=k}^{t-1} r_{i,k}} \prod_{k=t+1}^{n} C_{k,n+1}^{\sum_{l=1}^{t-1} r_{l,k}} \prod_{i=1}^{t-1} a^{s_{t-i,1}}_{n+1} b^{s_{t-i,2}}_{n+1}\,\text{and}\, u' =\prod_{k=1}^{t-1} \sigma_{t-k} C_{t-k+1,n+1}^{\beta_{t-k}},
\end{equation}
where $\beta_{k}= \sum_{i=1}^{k} r_{i,k+1}$ for $k=1,\ldots,t-1$. If $\alpha\in \mathbb Z$, let us show by induction on $1\leq i\leq t-1$ that:
\begin{equation}\label{eq:indprodsigmaii}
\sigma_{t-1}\cdots \sigma_{t-i} C_{t-i+1,n+1}^{\alpha}= 
C_{t+1,n+1}^{\alpha} C_{t,n+1}^{-\alpha} C_{t-i,n+1}^{\alpha} \sigma_{t-1}\cdots \sigma_{t-i}. 
\end{equation}
If $i=1$ then the result follows from~\reqref{conjsigma}. So suppose that~\reqref{indprodsigmaii} holds for some $1\leq i\leq t-2$. Then by~\reqref{conjsigma} and the induction hypothesis, we have:
\begin{align*}
\sigma_{t-1}\cdots \sigma_{t-i} \sigma_{t-i-1} C_{t-i,n+1}^{\alpha} &= \sigma_{t-1}\cdots \sigma_{t-i}  C_{t-i+1,n+1}^{\alpha} C_{t-i,n+1}^{-\alpha} C_{t-i-1,n+1}^{\alpha} \sigma_{t-i-1}\\
& = C_{t+1,n+1}^{\alpha} C_{t,n+1}^{-\alpha} C_{t-i-1,n+1}^{\alpha} \sigma_{t-1}\cdots \sigma_{t-i} \sigma_{t-i-1},
\end{align*}
which proves~\reqref{indprodsigmaii} for all $1\leq i\leq t-1$. Let us prove by induction on $1\leq i\leq t-1$ that:
\begin{equation}\label{eq:indu3}
u'=C_{t+1,n+1}^{\sum_{l=t-i+1}^{t-1} \beta_{l}} C_{t,n+1}^{-\sum_{l=t-i+1}^{t-1} \beta_{l}} 
\left( \prod_{j=t-i+1}^{t-1} C_{j,n+1}^{\beta_{j}} \right) \sigma_{t-1}\cdots \sigma_{t-i} C_{t-i+1,n+1}^{\beta_{t-i}} \prod_{k=i+1}^{t-1} \sigma_{t-k} C_{t-k+1,n+1}^{\beta_{t-k}}.
\end{equation}
If $i=1$ then~\reqref{indu3} is just the definition of $u'$. So suppose that~\reqref{indu2} holds for some $1\leq i \leq t-2$. By~\reqref{indprodsigmaii}, we have: 
\begin{align*}
u' &=C_{t+1,n+1}^{\sum_{l=t-i+1}^{t-1} \beta_{l}} C_{t,n+1}^{-\sum_{l=t-i+1}^{t-1} \beta_{l}} \left( \prod_{j=t-i+1}^{t-1} C_{j,n+1}^{\beta_{j}} \right) C_{t+1,n+1}^{\beta_{t-i}} C_{t,n+1}^{-\beta_{t-i}} C_{t-i,n+1}^{\beta_{t-i}} \sigma_{t-1}\cdots \sigma_{t-i} \prod_{k=i+1}^{t-1} \sigma_{t-k} C_{t-k+1,n+1}^{\beta_{t-k}}\\
&=C_{t+1,n+1}^{\sum_{l=t-i}^{t-1} \beta_{l}} C_{t,n+1}^{-\sum_{l=t-i}^{t-1} \beta_{l}} \left( \prod_{j=t-i}^{t-1} C_{j,n+1}^{\beta_{j}} \right) \sigma_{t-1}\cdots \sigma_{t-i} \sigma_{t-i-1} C_{t-i,n+1}^{\beta_{t-i-1}} \prod_{k=i+2}^{t-1} \sigma_{t-k} C_{t-k+1,n+1}^{\beta_{t-k}},
\end{align*}
which is \req{indu3} in the case $i+1$. Taking $i=t-1$ in~\reqref{indu3} and using~\reqref{indprodsigmaii}, we see that:
\begin{equation}\label{eq:uprime}
u'= C_{t+1,n+1}^{\sum_{l=2}^{t-1} \beta_{l}} C_{t,n+1}^{-\sum_{l=2}^{t-1} \beta_{l}} 
\left( \prod_{j=2}^{t-1} C_{j,n+1}^{\beta_{j}} \right) \sigma_{t-1}\cdots \sigma_{1} C_{2,n+1}^{\beta_{1}}= v' \sigma_{t-1}\cdots \sigma_{1},
\end{equation}
where:
\begin{equation}\label{eq:defvprime}
v'=C_{t+1,n+1}^{\sum_{l=1}^{t-1} \beta_{l}} C_{t,n+1}^{-\sum_{l=1}^{t-1} \beta_{l}} 
\prod_{j=1}^{t-1} C_{j,n+1}^{\beta_{j}}.
\end{equation} 
So by~\reqref{wu},~\reqref{indu2},~\reqref{wpup} and~\reqref{uprime}, we obtain:
\begin{equation}\label{eq:fullphi}
\widehat{\phi}((\sigma_{1}\cdots \sigma_{t-2}\sigma_{t-1}^{2} \sigma_{t-2}\cdots \sigma_{1})^{-1})= \sigma_{1}^{-1}\cdots \sigma_{t-1}^{-1} \omega^{-1} v^{-1} w^{-1},
\end{equation} 
where $\omega= \sigma_{1}\cdots \sigma_{t-1} w'v'$. Let $z= \sigma^{\sum_{i=1}^{t-1} s_{t-i,3}} C_{t+1,n+1}^{\sum_{l=1}^{t-1} \beta_{l}+ \sum_{l=1}^{t-1} r_{l,t+1}} \prod_{k=t+2}^{n} C_{k,n+1}^{\sum_{l=1}^{t-1} r_{l,k}} \prod_{i=1}^{t-1} a^{s_{t-i,1}}_{n+1} b^{s_{t-i,2}}_{n+1}$, and for $k=1,\ldots,t-1$, let $\gamma_{k}= \beta_{k}+\sum_{i=k}^{t-1} r_{i,k}$. Then by~\reqref{indprodsigmai},~\reqref{defwprime} and~\reqref{defvprime}, we have:
\begin{align}
\omega &= z \sigma_{1}\cdots \sigma_{t-1} C_{t,n+1}^{-\sum_{l=1}^{t-1} \beta_{l}} \prod_{k=1}^{t-1} C_{k,n+1}^{\gamma_{k}}\notag\\
&= z  (C_{1,n+1} C_{2,n+1}^{-1} C_{t+1,n+1})^{-\sum_{l=1}^{t-1} \beta_{l}} C_{1,n+1}^{\gamma_{1}} \prod_{k=2}^{t-1} (C_{1,n+1} C_{2,n+1}^{-1} C_{k+1,n+1})^{\gamma_{k}} \sigma_{1}\cdots \sigma_{t-1}.\label{eq:imwpvp}
\end{align}
The coefficient of $C_{2,n+1}$ in $(\sigma_{1}\cdots \sigma_{t-1} w'v')^{-1}$ is thus equal to:
\begin{equation*}
\sum_{l=2}^{t-1} \gamma_{l}-\sum_{l=1}^{t-1} \beta_{l}= \sum_{l=2}^{t-1} \left(\beta_{l}+\sum_{i=l}^{t-1} r_{i,l}\right)-\sum_{l=1}^{t-1} \beta_{l}= \sum_{l=2}^{t-1} \sum_{i=l}^{t-1} r_{i,l}-\beta_{1} = \sum_{l=2}^{t-1} \sum_{i=l}^{t-1} r_{i,l}-r_{1,2}.
\end{equation*}
% \comj{In case~(a), using the fact that $r_{j,i}=0$ for $j>i$ then the coefficient of $C_{2,n+1}$ at this point is equal to $\sum_{l=2}^{t-1} r_{l,l}-r_{1,2}=-t_{2}-r_{1,2}$ (see equation~(3.40) and the definition of $t_{2}$ just after equation~(3.39)).}
Combining~\reqref{defw},~\reqref{defv},~\reqref{fullphi} and~\reqref{imwpvp}, and making use of the relation $\sigma^{2}=1$ and the fact that $b_{n+1}a_{n+1}b_{n+1}^{-1}$ is a word in $a_{n+1}$ and $C_{1,n+1}$, 
% \comj{(this is $xy=yx$ in case~(a) and relation~(3) of Proposition~4.8 in case~(b))}, 
it follows that:
\begin{equation}\label{eq:relny}
\widehat{\phi}((\sigma_{1}\cdots \sigma_{t-2}\sigma_{t-1}^{2} \sigma_{t-2}\cdots \sigma_{1})^{-1}) = (\sigma_{1}\cdots \sigma_{t-2}\sigma_{t-1}^{2} \sigma_{t-2}\cdots \sigma_{1})^{-1} C_{2,n+1}^{\alpha} \xi,
\end{equation}
where 
%\comj{This element was called $y$ before, but I changed it to avoid confusion with $y$ in the previous part.}\comc{ok} 
$\xi= \xi(C_{1,n+1},C_{3,n+1},\ldots, C_{n,n+1}, a_{n+1},b_{n+1})$ is in canonical form, and $\alpha= \sum_{l=1}^{t-1} \sum_{i=l}^{t-1} r_{i,l+1} +\sum_{l=2}^{t-1} \sum_{i=l}^{t-1} r_{i,l}-r_{1,2}$. Note that $\alpha$ can be simplified:
\begin{align}
\alpha &= \sum_{l=1}^{t-1} \sum_{i=l}^{t-1} r_{i,l+1} +\sum_{l=2}^{t-1} \sum_{i=l}^{t-1} r_{i,l}-r_{1,2}= \sum_{l=2}^{t-1} \sum_{i=l-1}^{t-1} r_{i,l} +\sum_{l=2}^{t-1} \sum_{i=l}^{t-1} r_{i,l}-r_{1,2}\notag\\
&= 2 \sum_{l=2}^{t-1} \sum_{i=l}^{t-1} r_{i,l}+ \sum_{l=2}^{t-1} r_{l-1,l}-r_{1,2}= 2 \sum_{l=2}^{t-1} \sum_{i=l}^{t-1} r_{i,l}+ \sum_{l=2}^{t-1} r_{l,l+1}.\label{eq:defalpha}
\end{align}
% \comj{Note that this yields the left-hand side of equation~(4.40) in case~(b).}

%Let $M=\mathbb{T}$ (resp.\ $M=\mathbb{K}$), and let $\rho=C_{1,n+1}^{-1} C_{2,n+1}$. We now determine $\widehat{\phi}(R)$, where $R=b_{1}a_{1} b_{1}^{-1}a_{1}^{-1}$ (resp.\ $R=b_{1}a_{1}^{-1} b_{1}^{-1}a_{1}^{-1}$). Note that in $B_{t,s,m}(M)/H$, $a_{1}^{-1} b_{n+1} a_{1}= b_{n+1} \rho$, $a_{1} b_{n+1} a_{1}^{-1}= b_{n+1} \rho^{-1}$, $b_{1}^{-1} a_{n+1} b_{1}= a_{n+1} \rho^{-1}$ (resp.\ $b_{1}^{-1} a_{n+1} b_{1}= a_{n+1} \rho$), $b_{1} a_{n+1} b_{1}^{-1}= a_{n+1} \rho$, $b_{1}^{-1} b_{n+1} b_{1}= b_{n+1}$ (resp.\ $b_{1} b_{n+1} b_{1}^{-1} = b_{1}^{-1} b_{n+1} b_{1}= b_{n+1} \rho^{-1}$), where we have used the relations $b_{1} C_{1,n+1} b_{1}^{-1} = b_{1}^{-1} C_{1,n+1} b_{1}= C_{1,n+1}^{-1} C_{2,n+1}^{2}$ if $M=\mathbb{K}$ (this is a consequence of relation~(8) of Proposition~2.5).
We now determine $\widehat{\phi}(R)$, where $R=b_{1}a_{1} b_{1}^{-1}a_{1}^{-1}$ (resp.\ $R=b_{1}a_{1}^{-1} b_{1}^{-1}a_{1}^{-1}$).

If $M=\mathbb{T}$, we have:
\begin{align*}%\label{eq:phiT}
\widehat{\phi}(R) =& b_{1} a^{\beta_{1,1}}_{n+1} b^{\beta_{1,2}}_{n+1} \sigma^{\beta_{1,3}} C^{y_{1,1}}_{1,n+1} \cdots C^{y_{1,n}}_{n,n+1} a_{1} a^{\alpha_{1,1}}_{n+1} b^{\alpha_{1,2}}_{n+1} \sigma^{\alpha_{1,3}} C^{x_{1,1}}_{1,n+1} \cdots C^{x_{1,n}}_{n,n+1}\cdot\notag\\
& \sigma^{-\beta_{1,3}} C_{1,n+1}^{-y_{1,1}} \cdots C_{n,n+1}^{-y_{1,n}} b_{n+1}^{-\beta_{1,2}}  a_{n+1}^{-\beta_{1,1}} b_{1}^{-1} \sigma^{-\alpha_{1,3}} C_{1,n+1}^{-x_{1,1}} \cdots C_{n,n+1}^{-x_{1,n}} b_{n+1}^{-\alpha_{1,2}} a_{n+1}^{-\alpha_{1,1}} a_{1}^{-1}\notag\\
=& b_{1} a_{1} a^{\beta_{1,1}}_{n+1} b^{\beta_{1,2}}_{n+1} a^{\alpha_{1,1}}_{n+1} b_{n+1}^{\alpha_{1,2}-\beta_{1,2}} a_{n+1}^{-\beta_{1,1}} \sigma^{\alpha_{1,3}} C_{1,n+1}^{-\beta_{1,2}+x_{1,1}} C_{2,n+1}^{\beta_{1,2}+x_{1,2}} C_{3,n+1}^{x_{1,3}}  \cdots C_{n,n+1}^{x_{1,n}} 
b_{1}^{-1}\cdot\notag\\
& \sigma^{-\alpha_{1,3}} C_{1,n+1}^{-x_{1,1}} \cdots C_{n,n+1}^{-x_{1,n}} b_{n+1}^{-\alpha_{1,2}} a_{n+1}^{-\alpha_{1,1}} a_{1}^{-1}\notag\\
=& b_{1} a_{1} b_{1}^{-1} a^{\beta_{1,1}}_{n+1} b^{\beta_{1,2}}_{n+1} a^{\alpha_{1,1}}_{n+1} b_{n+1}^{\alpha_{1,2}-\beta_{1,2}} a_{n+1}^{-\beta_{1,1}} b_{n+1}^{-\alpha_{1,2}} a_{n+1}^{-\alpha_{1,1}} C_{1,n+1}^{-\beta_{1,2}-\alpha_{1,1}} C_{2,n+1}^{\beta_{1,2}+\alpha_{1,1}} a_{1}^{-1}\notag\\
=& b_{1} a_{1} b_{1}^{-1} a_{1}^{-1}  a^{\beta_{1,1}}_{n+1} b^{\beta_{1,2}}_{n+1} a^{\alpha_{1,1}}_{n+1} b_{n+1}^{\alpha_{1,2}-\beta_{1,2}} a_{n+1}^{-\beta_{1,1}} b_{n+1}^{-\alpha_{1,2}} a_{n+1}^{-\alpha_{1,1}} C_{1,n+1}^{-\beta_{1,2}-\alpha_{1,1}} C_{2,n+1}^{\beta_{1,2}+\alpha_{1,1}}.
\end{align*}
% \comj{By the notation of Section~3 of~[CDJ], $\beta_{1,2}+\alpha_{1,1}=k_{4}+k_{1}$, which agrees with equation~(3.24).} 
If $M=\mathbb{K}$, we have:
\begin{align*}%\label{eq:phiK}
\widehat{\phi}(R) =& b_{1} a^{\beta_{1,1}}_{n+1} b^{\beta_{1,2}}_{n+1} \sigma^{\beta_{1,3}} C^{y_{1,1}}_{1,n+1} \cdots C^{y_{1,n}}_{n,n+1} \sigma^{-\alpha_{1,3}} C_{1,n+1}^{-x_{1,1}} \cdots C_{n,n+1}^{-x_{1,n}} b_{n+1}^{-\alpha_{1,2}} a_{n+1}^{-\alpha_{1,1}} a_{1}^{-1}\cdot\notag\\
& \sigma^{-\beta_{1,3}} C_{1,n+1}^{-y_{1,1}} \cdots C_{n,n+1}^{-y_{1,n}} b_{n+1}^{-\beta_{1,2}}  a_{n+1}^{-\beta_{1,1}} b_{1}^{-1} \sigma^{-\alpha_{1,3}} C_{1,n+1}^{-x_{1,1}} \cdots C_{n,n+1}^{-x_{1,n}} b_{n+1}^{-\alpha_{1,2}} a_{n+1}^{-\alpha_{1,1}} a_{1}^{-1}\notag\\
=& b_{1} a_{1}^{-1} a^{\beta_{1,1}}_{n+1} b^{\beta_{1,2}-\alpha_{1,2}}_{n+1} a_{n+1}^{-\alpha_{1,1}} b_{n+1}^{-\beta_{1,2}}  a_{n+1}^{-\beta_{1,1}} \sigma^{-\alpha_{1,3}} C_{1,n+1}^{-x_{1,1}-\alpha_{1,2}+\beta_{1,2}} C_{2,n+1}^{-x_{1,2}+\alpha_{1,2}-\beta_{1,2}} C_{3,n+1}^{-x_{1,3}}  \cdots C_{n,n+1}^{-x_{1,n}} b_{1}^{-1}\cdot\notag\\
& \sigma^{-\alpha_{1,3}} C_{1,n+1}^{-x_{1,1}} \cdots C_{n,n+1}^{-x_{1,n}} b_{n+1}^{-\alpha_{1,2}} a_{n+1}^{-\alpha_{1,1}} a_{1}^{-1}\notag\\
=& b_{1} a_{1}^{-1} b_{1}^{-1} a^{\beta_{1,1}}_{n+1} b^{\beta_{1,2}-\alpha_{1,2}}_{n+1} a_{n+1}^{-\alpha_{1,1}} b_{n+1}^{-\beta_{1,2}}  a_{n+1}^{-\beta_{1,1}} b_{n+1}^{-\alpha_{1,2}} a_{n+1}^{-\alpha_{1,1}}\sigma^{-2\alpha_{1,3}} C_{1,n+1}^{\alpha_{1,1}-\beta_{1,2}} C_{2,n+1}^{-2(x_{1,1}+x_{1,2})-\alpha_{1,1}+\beta_{1,2}}\cdot\notag\\
& C_{3,n+1}^{-2x_{1,3}}  \cdots C_{n,n+1}^{-2x_{1,n}} a_{1}^{-1}\notag\\
=& b_{1} a_{1}^{-1} b_{1}^{-1} a_{1}^{-1} a^{\beta_{1,1}}_{n+1} b^{\beta_{1,2}-\alpha_{1,2}}_{n+1} a_{n+1}^{-\alpha_{1,1}} b_{n+1}^{-\beta_{1,2}}  a_{n+1}^{-\beta_{1,1}} b_{n+1}^{-\alpha_{1,2}} a_{n+1}^{-\alpha_{1,1}}\sigma^{-2\alpha_{1,3}} C_{1,n+1}^{\alpha_{1,1}-\beta_{1,2}-2\alpha_{1,2}}\cdot\notag\\
& C_{2,n+1}^{2(\alpha_{1,2}-x_{1,1}-x_{1,2})-\alpha_{1,1}+\beta_{1,2}} C_{3,n+1}^{-2x_{1,3}}  \cdots C_{n,n+1}^{-2x_{1,n}}.
\end{align*}
% \comj{In case~(a), $x_{1,1}=0$, $\alpha_{1,2}=k_{2}=0$ by Lemma~3.2, and so the coefficient of $C_{2,n+1}$ is equal to $-2x_{1,2}+\beta_{1,2}- \alpha_{1,1}=-2i_{2}+k_{4}-k_{1}$, which agrees with equation~(3.25).} 
So if $M=\mathbb{T}$ or $\mathbb{K}$ then:
\begin{equation}\label{eq:relnw}
\widehat{\phi}(R) =R C_{2,n+1}^{\delta} w,
\end{equation}
where:
\begin{equation}\label{eq:defdelta}
\delta=\begin{cases}
\beta_{1,2}+\alpha_{1,1} & \text{if $M=\mathbb{T}$}\\
2(\alpha_{1,2}-x_{1,1}-x_{1,2})-\alpha_{1,1}+\beta_{1,2} & \text{if $M=\mathbb{K}$,}
\end{cases}
\end{equation}
and $w=w(C_{1,n+1},C_{3,n+1},\ldots, C_{n,n+1}, a_{n+1},b_{n+1})$ is in canonical form. 

We now determine $\widehat{\phi}(\prod_{j=1}^{s} C_{1,t+j} C_{2,t+j}^{-1})$. If $M=\mathbb{T}$ (resp.\ $M=\mathbb{K}$), by Remark~\ref{rem:nova}(\ref{it:rn4}) and $i=t+1,\ldots,n$, $\widehat{\phi}(C_{1,i} C_{2,i}^{-1})=\widehat{\phi}(a_{i}^{-1} b_{1}^{-1} a_{i} b_{1})$ (resp.\ $\widehat{\phi}(C_{1,i} C_{2,i}^{-1})=\widehat{\phi}((a_{i}^{-1} b_{1}^{-1} a_{i} b_{1})^{-1})$). Let $i=t+1,\ldots,n$. Then $a_{i}^{-1} b_{n+1}a_{i}= b_{n+1} C_{i,n+1}^{-1} C_{i+1,n+1}$ and $a_{i} b_{n+1}a_{i}^{-1}= b_{n+1} C_{i,n+1} C_{i+1,n+1}^{-1}$. If $M=\mathbb{T}$ then:
\begin{align*}
\widehat{\phi}(C_{1,i} C_{2,i}^{-1})=& \sigma^{-\alpha_{i,3}} C_{1,n+1}^{-x_{i,1}} \cdots C_{n,n+1}^{-x_{i,n}} b_{n+1}^{-\alpha_{i,2}} a_{n+1}^{-\alpha_{i,1}} a_{i}^{-1} \sigma^{-\beta_{1,3}} C_{1,n+1}^{-y_{1,1}} \cdots C_{n,n+1}^{-y_{1,n}} b_{n+1}^{-\beta_{1,2}}  a_{n+1}^{-\beta_{1,1}} b_{1}^{-1}\cdot\\
& a_{i} a^{\alpha_{i,1}}_{n+1} b^{\alpha_{i,2}}_{n+1} \sigma^{\alpha_{i,3}} C^{x_{i,1}}_{1,n+1} \cdots C^{x_{i,n}}_{n,n+1} b_{1} a^{\beta_{1,1}}_{n+1} b^{\beta_{1,2}}_{n+1} \sigma^{\beta_{1,3}} C^{y_{1,1}}_{1,n+1} \cdots C^{y_{1,n}}_{n,n+1}\\
=& a_{i}^{-1} \sigma^{-\alpha_{i,3}-\beta_{1,3}} b_{n+1}^{-\alpha_{i,2}} a_{n+1}^{-\alpha_{i,1}} b_{n+1}^{-\beta_{1,2}}  a_{n+1}^{-\beta_{1,1}} C_{1,n+1}^{-x_{i,1}-y_{1,1}} \cdots C_{i-1,n+1}^{-x_{i,i-1}-y_{1,i-1}} C_{i,n+1}^{-x_{i,i}-y_{1,i}-\alpha_{i,2}}\cdot\\ 
& C_{i+1,n+1}^{-x_{i,i+1}-y_{1,i+1}+\alpha_{i,2}} C_{i+2,n+1}^{-x_{i,i+2}-y_{1,i+2}} \cdots C_{n,n+1}^{-x_{i,n}-y_{1,n}} b_{1}^{-1}\cdot\\
& a_{i} a^{\alpha_{i,1}}_{n+1} b^{\alpha_{i,2}}_{n+1} \sigma^{\alpha_{i,3}} C^{x_{i,1}}_{1,n+1} \cdots C^{x_{i,n}}_{n,n+1} b_{1} a^{\beta_{1,1}}_{n+1} b^{\beta_{1,2}}_{n+1} \sigma^{\beta_{1,3}} C^{y_{1,1}}_{1,n+1} \cdots C^{y_{1,n}}_{n,n+1}\\
=& a_{i}^{-1} b_{1}^{-1} \sigma^{-\alpha_{i,3}-\beta_{1,3}} b_{n+1}^{-\alpha_{i,2}} a_{n+1}^{-\alpha_{i,1}} b_{n+1}^{-\beta_{1,2}}  a_{n+1}^{-\beta_{1,1}} C_{1,n+1}^{-x_{i,1}-y_{1,1}+\alpha_{i,1}+\beta_{1,1}} C_{2,n+1}^{-x_{i,2}-y_{1,2}-\alpha_{i,1}-\beta_{1,1}}\cdot\\
& C_{3,n+1}^{-x_{i,3}-y_{1,3}} \cdots C_{i-1,n+1}^{-x_{i,i-1}-y_{1,i-1}} C_{i,n+1}^{-x_{i,i}-y_{1,i}-\alpha_{i,2}} C_{i+1,n+1}^{-x_{i,i+1}-y_{1,i+1}+\alpha_{i,2}} \cdot\\
& C_{i+2,n+1}^{-x_{i,i+2}-y_{1,i+2}} \cdots C_{n,n+1}^{-x_{i,n}-y_{1,n}} 
a_{i} a^{\alpha_{i,1}}_{n+1} b^{\alpha_{i,2}}_{n+1} \sigma^{\alpha_{i,3}} C^{x_{i,1}}_{1,n+1} \cdots C^{x_{i,n}}_{n,n+1} b_{1} a^{\beta_{1,1}}_{n+1} b^{\beta_{1,2}}_{n+1} \sigma^{\beta_{1,3}}\cdot\\
& C^{y_{1,1}}_{1,n+1} \cdots C^{y_{1,n}}_{n,n+1}\\
=& a_{i}^{-1} b_{1}^{-1} a_{i} \sigma^{-\beta_{1,3}} b_{n+1}^{-\alpha_{i,2}} a_{n+1}^{-\alpha_{i,1}} b_{n+1}^{-\beta_{1,2}}  a_{n+1}^{\alpha_{i,1}-\beta_{1,1}} b^{\alpha_{i,2}}_{n+1} C_{1,n+1}^{-y_{1,1}+\alpha_{i,1}+\beta_{1,1}} C_{2,n+1}^{-y_{1,2}-\alpha_{i,1}-\beta_{1,1}} \cdot\\
& C_{3,n+1}^{-y_{1,3}} \cdots C_{i-1,n+1}^{-y_{1,i-1}} C_{i,n+1}^{-y_{1,i}+\beta_{1,2}} C_{i+1,n+1}^{-y_{1,i+1}-\beta_{1,2}} C_{i+2,n+1}^{-y_{1,i+2}} \cdots C_{n,n+1}^{-y_{1,n}} b_{1} a^{\beta_{1,1}}_{n+1} b^{\beta_{1,2}}_{n+1} \sigma^{\beta_{1,3}}\cdot\\
& C^{y_{1,1}}_{1,n+1} \cdots C^{y_{1,n}}_{n,n+1}\\
=& a_{i}^{-1} b_{1}^{-1} a_{i} b_{1} b_{n+1}^{-\alpha_{i,2}} a_{n+1}^{-\alpha_{i,1}} b_{n+1}^{-\beta_{1,2}}  a_{n+1}^{\alpha_{i,1}-\beta_{1,1}} b^{\alpha_{i,2}}_{n+1} a^{\beta_{1,1}}_{n+1} b^{\beta_{1,2}}_{n+1} C_{1,n+1}^{\alpha_{i,1}} C_{2,n+1}^{-\alpha_{i,1}} C_{i,n+1}^{\beta_{1,2}} C_{i+1,n+1}^{-\beta_{1,2}}.
\end{align*}
If $M=\mathbb{K}$ then:
\begin{align*}
\widehat{\phi}(C_{1,i} C_{2,i}^{-1})=&
\sigma^{-\beta_{1,3}} C_{1,n+1}^{-y_{1,1}} \cdots C_{n,n+1}^{-y_{1,n}} b_{n+1}^{-\beta_{1,2}}  a_{n+1}^{-\beta_{1,1}} b_{1}^{-1} \sigma^{-\alpha_{i,3}} C_{1,n+1}^{-x_{i,1}} \cdots C_{n,n+1}^{-x_{i,n}} b_{n+1}^{-\alpha_{i,2}} a_{n+1}^{-\alpha_{i,1}} a_{i}^{-1} \cdot\\
& b_{1} a^{\beta_{1,1}}_{n+1} b^{\beta_{1,2}}_{n+1} \sigma^{\beta_{1,3}} C^{y_{1,1}}_{1,n+1} \cdots C^{y_{1,n}}_{n,n+1} a_{i} a^{\alpha_{i,1}}_{n+1} b^{\alpha_{i,2}}_{n+1} \sigma^{\alpha_{i,3}} C^{x_{i,1}}_{1,n+1} \cdots C^{x_{i,n}}_{n,n+1}\\
=& b_{1}^{-1} \sigma^{-\beta_{1,3}-\alpha_{i,3}} b_{n+1}^{-\beta_{1,2}}  a_{n+1}^{-\beta_{1,1}} b_{n+1}^{-\alpha_{i,2}} a_{n+1}^{-\alpha_{i,1}} C_{1,n+1}^{y_{1,1}-x_{i,1}+\beta_{1,1}-\beta_{1,2}} C_{2,n+1}^{-y_{1,2}-x_{i,2}-2y_{1,1}-\beta_{1,1}+\beta_{1,2}}\cdot\\
& C_{3,n+1}^{-y_{1,3}-x_{i,3}} \cdots C_{n,n+1}^{-y_{1,n}-x_{i,n}} a_{i}^{-1} b_{1} a^{\beta_{1,1}}_{n+1} b^{\beta_{1,2}}_{n+1} \sigma^{\beta_{1,3}} C^{y_{1,1}}_{1,n+1} \cdots C^{y_{1,n}}_{n,n+1}\cdot\\
& a_{i} a^{\alpha_{i,1}}_{n+1} b^{\alpha_{i,2}}_{n+1} \sigma^{\alpha_{i,3}} C^{x_{i,1}}_{1,n+1} \cdots C^{x_{i,n}}_{n,n+1}\\
=& b_{1}^{-1} a_{i}^{-1} \sigma^{-\beta_{1,3}-\alpha_{i,3}} b_{n+1}^{-\beta_{1,2}}  a_{n+1}^{-\beta_{1,1}} b_{n+1}^{-\alpha_{i,2}} a_{n+1}^{-\alpha_{i,1}} C_{1,n+1}^{y_{1,1}-x_{i,1}+\beta_{1,1}-\beta_{1,2}} C_{2,n+1}^{-y_{1,2}-x_{i,2}-2y_{1,1}-\beta_{1,1}+\beta_{1,2}}\cdot\\
& C_{3,n+1}^{-y_{1,3}-x_{i,3}} \cdots C_{i-1,n+1}^{-y_{1,i-1}-x_{i,i-1}} C_{i,n+1}^{-y_{1,i}-x_{i,i}-\beta_{1,2}-\alpha_{i,2}} C_{i+1,n+1}^{-y_{1,i+1}-x_{i,i+1}+\beta_{1,2}+\alpha_{i,2}}\cdot\\
& C_{i+2,n+1}^{-y_{1,i+2}-x_{i,i+2}} \cdots C_{n,n+1}^{-y_{1,n}-x_{i,n}} b_{1} a^{\beta_{1,1}}_{n+1} b^{\beta_{1,2}}_{n+1} \sigma^{\beta_{1,3}} C^{y_{1,1}}_{1,n+1} \cdots C^{y_{1,n}}_{n,n+1} a_{i} a^{\alpha_{i,1}}_{n+1} b^{\alpha_{i,2}}_{n+1} \sigma^{\alpha_{i,3}}\cdot\\
& C^{x_{i,1}}_{1,n+1} \cdots C^{x_{i,n}}_{n,n+1}\\
=& b_{1}^{-1} a_{i}^{-1} b_{1} \sigma^{-\alpha_{i,3}} b_{n+1}^{-\beta_{1,2}}  a_{n+1}^{-\beta_{1,1}} b_{n+1}^{-\alpha_{i,2}} a_{n+1}^{\beta_{1,1}-\alpha_{i,1}} b^{\beta_{1,2}}_{n+1} C_{1,n+1}^{x_{i,1}+\alpha_{i,1}-\alpha_{i,2}} C_{2,n+1}^{-x_{i,2}-2x_{i,1}-\alpha_{i,1}+\alpha_{i,2}}\cdot\\
& C_{3,n+1}^{-x_{i,3}} \cdots C_{i-1,n+1}^{-x_{i,i-1}} C_{i,n+1}^{-x_{i,i}-\beta_{1,2}-\alpha_{i,2}} C_{i+1,n+1}^{-x_{i,i+1}+\beta_{1,2}+\alpha_{i,2}} C_{i+2,n+1}^{-x_{i,i+2}} \cdots C_{n,n+1}^{-x_{i,n}}\cdot\\
& a_{i} a^{\alpha_{i,1}}_{n+1} b^{\alpha_{i,2}}_{n+1} \sigma^{\alpha_{i,3}} C^{x_{i,1}}_{1,n+1} \cdots C^{x_{i,n}}_{n,n+1}\\
=& b_{1}^{-1} a_{i}^{-1} b_{1} a_{i} b_{n+1}^{-\beta_{1,2}}  a_{n+1}^{-\beta_{1,1}} b_{n+1}^{-\alpha_{i,2}} a_{n+1}^{\beta_{1,1}-\alpha_{i,1}} b^{\beta_{1,2}}_{n+1} a^{\alpha_{i,1}}_{n+1} b^{\alpha_{i,2}}_{n+1}
C_{1,n+1}^{2x_{i,1}+\alpha_{i,1}-\alpha_{i,2}} C_{2,n+1}^{-2x_{i,1}-\alpha_{i,1}+\alpha_{i,2}}\cdot \\
& C_{i,n+1}^{-\beta_{1,2}} C_{i+1,n+1}^{\beta_{1,2}}. 
\end{align*}
Using Remark~\ref{rem:nova}(\ref{it:rn4}) and relation~(\ref{eq:relnsd}), it follows from these computations that if $M=\mathbb{T}$ or $\mathbb{K}$ and $i=t+1,\ldots,n$, $\widehat{\phi}(C_{1,i} C_{2,i}^{-1})=C_{1,i} C_{2,i}^{-1} C_{2,n+1}^{\gamma_{i}} z_{i}$, where $z_{i}=z_{i}(C_{1,n+1},C_{3,n+1},\ldots, C_{n,n+1}, a_{n+1},b_{n+1})$ is in canonical form, and:
\begin{equation}\label{eq:defgammai}
\gamma_{i}= \begin{cases}
 -\alpha_{i,1} & \text{if $M=\mathbb{T}$}\\
 -2x_{i,1}-\alpha_{i,1}+\alpha_{i,2} & \text{if $M=\mathbb{K}$.}
\end{cases}
\end{equation}
Applying also 
%. Using once more the fact that $C_{1,i} C_{2,i}^{-1}=a_{i}^{-1} b_{1}^{-1} a_{i} b_{1}$ (resp.\ $C_{1,i} C_{2,i}^{-1}=(a_{i}^{-1} b_{1}^{-1} a_{i} b_{1})^{-1}$) if $M=\mathbb{T}$ (resp.\ $M=\mathbb{K}$) and {\color{green!50!black}{using
relations~(\ref{it:puras3}) and~(\ref{it:puras8}) of Theorem~\ref{th:puras},
% and relation~(\ref{eq:relnsd}), 
one may check that the word $a_{i}^{-1} b_{1}^{-1} a_{i} b_{1}$ (resp.\ $(a_{i}^{-1} b_{1}^{-1} a_{i} b_{1})^{-1}$) commutes with $a_{n+1},b_{n+1}$ and $C_{j,n+1}$ for $j=1,\ldots,n$, from which it follows that:
\begin{equation}\label{eq:relnz}
\widehat{\phi}\left(\prod_{j=1}^{s} C_{1,t+j} C_{2,t+j}^{-1}\right)= \left(\prod_{j=1}^{s} C_{1,t+j} C_{2,t+j}^{-1} \right) C_{2,n+1}^{\gamma} z,
\end{equation}
where:
\begin{equation}\label{eq:defgamma}
\gamma= \sum_{i=1}^{s} \gamma_{t+i},
\end{equation}
and $z=\prod_{i=1}^{s} z_{t+i}$ is in canonical form. 

Taking the image by $\widehat{\phi}$ of~\reqref{surface} with $d=s$ and applying~\reqref{relny},~\reqref{relnw} and~\reqref{relnz}, we obtain the following equality in $B_{t,s,m}(M)/H$:
\begin{equation}\label{eq:imphi}
\left(\prod_{j=1}^{s} C_{1,t+j} C_{2,t+j}^{-1}\right) C_{2,n+1}^{\gamma} z = (\sigma_{1}\cdots \sigma_{t-2}\sigma_{t-1}^{2} \sigma_{t-2}\cdots \sigma_{1})^{-1} C_{2,n+1}^{\alpha} \xi R C_{2,n+1}^{\delta} w. 
\end{equation}
One may check using Proposition~4.8 that $R$ commutes with $a_{n+1},b_{n+1}$ and $C_{j,n+1}$ for all $j=1,\ldots,n$. By relations~\reqref{relnsa},
% on page~\pageref{eq:relnsa}, 
we have:
\begin{equation}\label{eq:prodC12}
\prod_{j=1}^{m} C_{1,n+j} C_{2,n+j}^{-1} =C_{1,n+j}^{m} C_{2,n+j}^{-m}.
\end{equation}
It follows from~\reqref{surface} with $d=s+m$,~\reqref{imphi} and~\reqref{prodC12} that:
\begin{equation}\label{eq:final}
C_{2,n+1}^{\gamma+m} z = C_{2,n+1}^{\alpha} \cdot \xi C_{2,n+1}^{\delta} w,
\end{equation}
which up to collecting terms in $\ker{\widehat{\phi}_{\ast}}$, is in canonical form. 

% \comc{essa parte de baixo nao estava no texto do john, mas como sera usado para o teorema 1.1 e 1.4 entao achei que poderia ficar aqui. sao lemas que estavam na secao 3}

\begin{lem}\label{lem:exp.M1} 
%\comj{After discussion, the statement has been modified to take account of the general situation. We will see later that $r_{i,j}=0$ for all $1\leq i\leq t-1$ and $j=2,\ldots,i-1,i+3,\ldots,t$ in the setting of Theorem~\ref{th:FNsplits}.}
With the notation of~(\ref{eq:phi1a}):
\begin{enumerate}[(a)]
\item\label{it:exp.M1a} let $t\geq 4$, and let $i,j\in \brak{1,\ldots, n-1}\setminus\brak{t}$, where $\abs{i-j}\geq 2$. In $\ker{\widehat{p}_{\ast}}$ we have:
\begin{equation}\label{eq:commba1}
[b_{n+1}^{-s_{i,2}} a_{n+1}^{-s_{i,1}}, b_{n+1}^{-s_{j,2}} a_{n+1}^{-s_{j,1}}]=
C_{i,n+1}^{r_{j,i+1}} C_{i+1,n+1}^{-2r_{j,i+1}} C_{i+2,n+1}^{r_{j,i+1}}
C_{j,n+1}^{-r_{i,j+1}} C_{j+1,n+1}^{2r_{i,j+1}} C_{j+2,n+1}^{-r_{i,j+1}}
\end{equation}
%\item \comj{This is the version in the general situation of what Carolina wrote in green after the previous version of Lemma~\ref{lem:exp.M1}.} 
\item\label{it:exp.M1b} if $t\geq 3$, for all $1\leq i\leq t-2$, 
% \comj{check??}\comc{ok! tudo certo!}, 
in $\ker{\widehat{p}_{\ast}}$ we have:
\begin{equation}\label{eq:commba2}
\delta= \prod_{\substack{k=1\\ k\neq i+1,i+2}}^n C_{k,n+1}^{r_{i+1,k}-r_{i,k}} \ldotp 
C_{i,n+1}^{-r_{i,i+1}-r_{i,i+2}} 
C_{i+1,n+1}^{\rho}
C_{i+2,n+1}^{-\rho}
C_{i+3,n+1}^{r_{i+1,i+1}+r_{i+1,i+2}},
\end{equation}
where $\rho=2r_{i,i+2}+ 2r_{i+1,i+1}+r_{i,i+1}+r_{i+1,i+2}$, $\delta= \beta^{-1}\alpha^{-1}\beta^{-1} \alpha \beta \alpha \sigma^{s_{i,3}-s_{i+1,3}}$, with $\alpha=a_{n+1}^{s_{i,1}} b_{n+1}^{s_{i,2}}$ and $\beta= a_{n+1}^{s_{i+1,1}} b_{n+1}^{s_{i+1,2}}$. 
\end{enumerate}
% Let $t\geq 4$. With the notation of~(\ref{eq:phi1a}), \comj{The two statements combined:} $r_{i,j}=0$ for all $1\leq i\leq t-1$ and $j=2,\ldots,i-1,i+3,\ldots,t$.
% \begin{enumerate}
% \item\label{it:expM1a} $r_{j,i}=0$ for $2\leq i <j \leq t-1$.
% \item\label{it:expM1b} $r_{i,j}=0$ for $i\geq 1$ and $i+3\leq j\leq t$.
% \end{enumerate}
\end{lem}

\begin{proof}
Let $t\geq 4$, and let $1\leq i,j\leq t-1$, where $\abs{i-j}\geq 2$, and consider the Artin relation $\sigma_{i}\sigma_{j}=\sigma_{j}\sigma_{i}$ in $B_{t,s}(M)/H'$. %~(\ref{it:full2}) of Theorem~\ref{th:total}. 
%\textbf{Relation: $\sigma_{i}\sigma_{j}=\sigma_{j}\sigma_{i},\,\left|i-j\right|\geq 2$:}
By relation~(\ref{rel:sigmaCij}) and relation~(\ref{it:Bnm6}) of Proposition~\ref{prop:B_{n,m}}, the only generators of $B_m(M\setminus\left\{x_1,\ldots,x_n\right\})/L$ that do not both commute with $\sigma_{i}$ and $\sigma_{j}$ in $B_{t,s,m}(M)/H$ are $C_{i+1,n+1}$ and $C_{j+1,n+1}$ respectively. Taking the image of $\sigma_{i}\sigma_{j}=\sigma_{j}\sigma_{i}$ by $\widehat{\phi}$ and making use of~(\ref{eq:phi1a}), it follows that the coefficients of $\sigma$ and of the terms $C_{k,n+1}$, $k=1,\ldots,n$, $k\neq i+1,j+1$ cancel pairwise, and applying~(\ref{rel:sigmaCij}), we obtain the following relation:
\begin{equation*}
\sigma_i \sigma_j \tau_{i} \tau_{j} C_{j,n+1}^{r_{i,j+1}} C_{j+1,n+1}^{r_{j,j+1}-r_{i,j+1}} C_{j+2,n+1}^{r_{i,j+1}} C_{i+1,n+1}^{r_{i,i+1}+r_{j,i+1}}= \sigma_j \sigma_i \tau_{j} \tau_{i}
C_{i,n+1}^{r_{j,i+1}} C_{i+1,n+1}^{r_{i,i+1}-r_{j,i+1}} C_{i+2,n+1}^{r_{j,i+1}} C_{j+1,n+1}^{r_{j,j+1}+r_{i,j+1}},
\end{equation*}
where $\tau_{i}=a_{n+1}^{s_{i,1}} b_{n+1}^{s_{i,2}}$ and $\tau_{j}=a_{n+1}^{s_{j,1}} b_{n+1}^{s_{j,2}}$. 
Equation~\reqref{commba1} then follows using the lift of relation~(\ref{it:full2}) of Theorem~\ref{th:total}.

We obtain equation~\reqref{commba2} in a similar manner by considering the image by $\widehat{\phi}$ of the Artin relation $\sigma_i \sigma_{i+1} \sigma_i =\sigma_{i+1}\sigma_i \sigma_{i+1}$ for $1\leq i\leq t-2$, where $t\geq 3$.
\end{proof}

\section{Proof of Theorem~\ref{th:FNsplits}}\label{sec:FNsplits}

This section is devoted to proving Theorem~\ref{th:FNsplits}. We start by showing that the condition is sufficient.

\begin{prop}\label{prop:split}
Let $M$ be the torus or the Klein bottle, and let $m,n\geq 1$. If $n$ divides $m$ then the generalised Fadell-Neuwirth short exact sequence~(\ref{eq:gen.seqFN}) splits.
\end{prop}

\begin{proof}
Suppose that $m=ln$ for some $l\in \mathbb N$. To conclude that there exists a section, we proceed in a manner similar to that of~\cite{FN} in the case of the short exact sequence~(\ref{eq:seqFN}). If $M$ is the $2$-torus or the Klein bottle, let $\nu$ be a non-vanishing vector field in $M$ and let $d$ be a metric on $M$. We shall construct a cross-section on the level of configuration spaces, from which the result will follow by taking the induced homomorphism on the level of fundamental groups. Let $s\colon\thinspace {D_{n}(M)} \to D_{n,m}(M)$ be the map defined by $s(x)=(x,s_{1}(x),\ldots,s_{n}(x))$ for all $x=(x_{1},\ldots, x_{n})\in D_n(M)$ (note that notationally, we do not distinguish between ordered and unordered tuples), where for $i=1,\ldots, n$: 
%with $s(x)=s(x_{1},\ldots,x_{n})$ is defined as: 
%Para cada $n$-?pla $x=(x_{1},\ldots,x_{n})\in {F_{n}(M)}/{S_{n}}$, tome , e assim a sec??o $s$ ? dada por
%\begin{equation}\label{sec}\left(x,x_{1}+\frac{1.v(x).\epsilon(x)}{2(l+1)},\ldots,x_{1}+\frac{l.\nu(x).\epsilon(x)}{2(l+1)}, \cdots, x_{n}+\frac{1.v(x).\epsilon(x)}{2(l+1)},\ldots,x_{n}+\frac{l.v(x).\epsilon(x)}{2(l+1)}\right),
%\end{equation}
\begin{equation}\label{eq:sec1}
s_{i}(x)=\left(x_{i}+\frac{\nu(x).\epsilon(x)}{2(l+1)},x_{i}+\frac{2\nu(x).\epsilon(x)}{2(l+1)},\ldots,x_{i}+\frac{l\nu(x).\epsilon(x)}{2(l+1)}\right),
\end{equation} 
and $\displaystyle \epsilon(x)=\min_{1\leq k < j \leq n}\big\{ d(x_{k},x_{j})\big\}>0$. So for all $i=1,\ldots,n$, $s_i(x)$ consists of $l$ distinct unordered points of $M$, and the union of these points yields $m$ distinct unordered points of $M$ that are also distinct from the $n$ points of $x$. Therefore $s$ is a well-defined continuous map, and it is a cross-section for the map $p \colon\thinspace D_{n,m}(M)\longrightarrow D_{n}(M)$. 
\end{proof}

Proposition~\ref{prop:split} gives a sufficient condition for the short exact sequence~(\ref{eq:gen.seqFN}) to split. We now prove that it is also necessary.  
%\subsection{Existence of the section}
Suppose then that the short exact sequence~(\ref{eq:gen.seqFN}) splits. If $n=1$ then there is nothing to prove. So suppose that $n\geq 2$. We will use the computations of Section~\ref{sec:contas} and the commutative diagram~(\ref{eq:basic}), with $s=0$, $t=n$, $L=H= \Gamma_{2}(B_{m}(M\setminus \brak{x_{1},\ldots,x_{n}}))$ and $H'=\brak{1}$. Note that $X'=\varnothing$ in this case, so relations~(\ref{eq:relnsd}) do not exist. Also, relation~(\ref{eq:relnsa}) follows from Proposition~\ref{prop:kernelab}, and relations~(\ref{eq:relnsb}) and~(\ref{eq:relnsc}) follow from Proposition~\ref{prop:B_{n,m}ab} and Theorems~\ref{th:puras} and~\ref{th:total}.
%The canonical homomorphism $B_{n,m}(M) \to B_{n,m}(M)/\Gamma_2(B_m(M\setminus\left\{x_1, \ldots,x_n\right\}))$ gives rise to a commutative diagram involving the short exact sequences~(\ref{eq:gen.seqFN}) and~(\ref{eq:sesquot}), and the section for the homomorphism $p_{\ast}$ of~(\ref{eq:gen.seqFN}) induces a section  $\phi\colon\thinspace B_{n}(M)\longrightarrow B_{n,m}(M)/\Gamma_{2}(B_m(M\setminus\left\{x_1, \ldots,x_n\right\}))$ for the homomorphism $p_{\ast}'\colon\thinspace B_{n,m}(M)/\Gamma_2(B_m(M\setminus\left\{x_1, \ldots,x_n\right\})) \to B_n(M)$ of~(\ref{eq:sesquot}). If $w\in B_n(M)$ is written in terms of the generators given by Theorem~\ref{th:total} then the element of $B_{n,m}(M)/\Gamma_2(B_m(M\setminus\left\{x_1, \ldots,x_n\right\}))$ written with the corresponding generators of Proposition~\ref{prop:B_{n,m}ab}, that we also denote by $w$, satisfies $\phi(w)=wz$, where $z$ belongs to $\ker{p_{\ast}'}$. Writing $z$ with respect to the generators of $B_m(M\setminus\left\{x_1,\ldots,x_n\right\})\ab$ given by Proposition~\ref{prop:kernelab}, we shall refer to the decomposition $\phi(w)=wz$ as the canonical form of $\phi(w)$. \comc{a definicao de canonical form aparece aqui pela primeira vez, deve ir para secao 2} 
Making use of the presentation of $\ker{\widehat{p}_{\ast}}= B_m(M\setminus\left\{x_1,\ldots,x_n\right\})\ab$ given by Proposition~\ref{prop:kernelab}, it follows that: %\comc{oct02 - estou mantendo a notacao do texto antigo para esse caso especifico} \comj{OK. See the comment below.}
\begin{equation}\label{eq:exp}
\left\{ \begin{aligned}
\widehat{\phi}(a)&=a\cdot x^{k_{1}}y^{k_{2}}\sigma^{i_{0}}\rho^{i_{2}}_{2}\cdots\rho^{i_{n}}_{n}\\
\widehat{\phi}(b)&=b\cdot x^{k_{3}}y^{k_{4}}\sigma^{j_{0}}\rho^{j_{2}}_{2}\cdots\rho^{j_{n}}_{n}\\
\widehat{\phi}(\sigma_{i})&= \text{$\sigma_{i}\cdot x^{l_{i,1}}y^{l_{i,2}} \sigma^{r_{i,0}}\rho^{r_{i,2}}_{2}\cdots\rho^{r_{i,n}}_{n}$ for $i=1,\ldots,n-1$,}
\end{aligned}\right.\end{equation}
where $k_1,\ldots,k_4,l_{i,1}, l_{i,2}, i_q,j_q,r_{i,q}\in \mathbb{Z}$ for $q=0,2,\ldots,n$, and $i_0,j_0,r_{i,0}$ are defined modulo $2$. 
% \comj{I added this:}\comc{ok} 
Comparing the notation of~\reqref{phi1a} with that of~\reqref{exp}, by Proposition~\ref{prop:kernelab}, $a_{n+1}=x$, $b_{n+1}=y$, $\rho_i=C_{i,n+1}$ for $i=2,\ldots,n$, $\rho_1=C_{1,n+1}=1$ if $M=\mathbb{T}$ and $\rho_1=C_{1,n+1}=x^2$ if $M=\mathbb{K}$. Note also that the exponents in~\reqref{exp} have been renamed with respect to~\reqref{phi1a}. To simplify the notation in what follows, for $q=0,2,\ldots,n$, let $r_{1,q}=r_{q}$ and for $p=1,2$, let $l_{1,p}=l_{p}$.
 % } $l_{1,q}=l_{q}$ \comj{for $l_{1,q}$, aren't the allowed values $q=1,2$?}\textcolor{blue}{Sim}. 
 We also set $r_{1,1}=0$. We will now take the image by $\widehat{\phi}$ of some of the relations of Theorem~\ref{th:total} to obtain relations in $B_{n,m}(M)/\Gamma_2(B_m(M\setminus\left\{x_1, \ldots,x_n\right\}))$ that we will simplify using Proposition~\ref{prop:B_{n,m}ab}. This will enable us to obtain information about the coefficients appearing in~(\ref{eq:exp}) above.

We first apply this procedure to relations~(\ref{it:full5}) and~(\ref{it:full6}) of Theorem~\ref{th:total}. 

\begin{lem}\label{lem:exp.M} 
With the above notation, we have: 
%have the following relations about the exponents that appear in~(\ref{eq:exp}):
\begin{enumerate}
\item\label{it:M1} $l_{1}=0$ if $M=\mathbb{T}$, and $l_{1}=k_{4}-r_{2}$ if $M=\mathbb{K}$.

%$l_{1}=\left\{\begin{array}{lr}0,&M=\mathbb{T}\\ k_{4}-r_{2},&M=\mathbb{K}\end{array}\right.$ 
\item\label{it:M2} $l_{2}=0$.

\item\label{it:M3} $k_{4}=k_{1}$ if $M=\mathbb{T}$, and $k_2=0$ and  $k_{4}=-k_{1}-2i_{2}$ if $M=\mathbb{K}$.

%$k_{4}=\left\{\begin{array}{lr}k_{1},&M=\mathbb{T}\\ -k_{1}-2i_{2},&M=\mathbb{K}\end{array}\right.$ 
	
\item\label{it:M4} If $n\geq3$, $k_{4}=-r_{2}-2r_{3}$.
\end{enumerate}
\end{lem}

\begin{proof}
We start by studying the image by $\widehat{\phi}$ of relation~(\ref{it:full5}) of Theorem~\ref{th:total}, where we substitute each term of the image by the corresponding term of~(\ref{eq:exp}). The left- and right-hand sides yield respectively:
\begin{equation}\label{eq:r.1}
\widehat{\phi}(b^{-1}\sigma_{1}a)= (\rho^{-j_{n}}_{n}\cdots \rho^{-j_{2}}_{2} \sigma^{-j_{0}}y^{-k_{4}}x^{-k_{3}}b^{-1})(\sigma_{1} x^{l_{1}}y^{l_{2}} \sigma^{r_{0}} \rho^{r_{2}}_{2}\cdots\rho^{r_{n}}_{n})(a x^{k_{1}}y^{k_{2}} \sigma^{i_{0}}\rho^{i_{2}}_{2}\cdots\rho^{i_{n}}_{n})
\end{equation}
and
%\begin{equation}\label{eq:r.2}
\begin{align}\label{eq:r.2}
\widehat{\phi}(\sigma_{1}a\sigma_{1}b^{-1}\sigma_{1})=& (\sigma_{1} x^{l_{1}}y^{l_{2}}\sigma^{r_{0}}\rho^{r_{2}}_{2}\cdots\rho^{r_{n}}_{n}) (a x^{k_{1}}y^{k_{2}}\sigma^{i_{0}}\rho^{i_{2}}_{2}\cdots\rho^{i_{n}}_{n})(\sigma_{1} x^{l_{1}}y^{l_{2}}\sigma^{r_{0}} \rho^{r_{2}}_{2}\cdots \rho^{r_{n}}_{n})\cdot \notag\\
& (\rho^{-j_{n}}_{n}\cdots\rho^{-j_{2}}_{2}\sigma^{-j_{0}}y^{-k_{4}}x^{-k_{3}}b^{-1})(\sigma_{1} x^{l_{1}}y^{l_{2}} \sigma^{r_{0}} \rho^{r_{2}}_{2}\cdots\rho^{r_{n}}_{n}).
\end{align}
%\end{equation}
Using Proposition~\ref{prop:B_{n,m}ab}, we see that the conjugate by $a,b$ or $\sigma_1$ of $x,y,\rho_2$ or $\rho_3$ is a word in $x,y,\rho_2$ and $\rho_3$, and that $a,b$ and $\sigma_1$ commute with each of $\rho_4,\ldots, \rho_n$ and $\sigma$. In this way, the terms of~(\ref{eq:r.1}) and~(\ref{eq:r.2}) involving $\rho_4,\ldots, \rho_n$ and $\sigma$ commute with all of the other terms, so they may be gathered together on the right-hand side of each of the expressions, and the remaining terms in the canonical form do not involve the elements $\rho_4,\ldots, \rho_n$ or $\sigma$. In particular, identifying the coefficients of these elements, for $k=4,\ldots,n$, we obtain $i_{k}-j_{k}+r_{k}=i_{k}-j_{k}+3r_{k}$, in other words:
\begin{equation*}%\label{it:01}
\text{$r_k=0$ for $k=4,\ldots,n$.}
\end{equation*}
It follows from~(\ref{eq:r.1}) and~(\ref{eq:r.2}) that:
\begin{multline}\label{eq:wlwr}
(\rho^{-j_{3}}_{3}\rho^{-j_{2}}_{2} y^{-k_{4}}x^{-k_{3}}b^{-1})(\sigma_{1} x^{l_{1}}y^{l_{2}} \rho^{r_{2}}_{2} \rho^{r_{3}}_{3})(a x^{k_{1}}y^{k_{2}} \rho^{i_{2}}_{2} \rho^{i_{3}}_{3})= (\sigma_{1} x^{l_{1}}y^{l_{2}} \rho^{r_{2}}_{2} \rho^{r_{3}}_{3}) (a x^{k_{1}}y^{k_{2}} \rho^{i_{2}}_{2}\rho^{i_{3}}_{3})\ldotp\\
(\sigma_{1} x^{l_{1}}y^{l_{2}} \rho^{r_{2}}_{2} \rho^{r_{3}}_{3}) (\rho^{-j_{3}}_{3} \rho^{-j_{2}}_{2} y^{-k_{4}}x^{-k_{3}}b^{-1})(\sigma_{1} x^{l_{1}}y^{l_{2}}  \rho^{r_{2}}_{2} \rho^{r_{3}}_{3}).
\end{multline}
Let $w_L$ and $w_R$ denote the left- and right-hand side of~(\ref{eq:wlwr}) respectively. We now put each of $w_L$ and $w_R$ in canonical form using relations~(\ref{it:s4})--(\ref{it:s8}) of Proposition~\ref{prop:B_{n,m}ab} and Remarks~\ref{rem:misc}(\ref{it:miscc}). First suppose that $M=\mathbb{T}$. Then:
\begin{equation*}
(\sigma_{1} x^{l_{1}}y^{l_{2}} \rho^{r_{2}}_{2} \rho^{r_{3}}_{3}) (a x^{k_{1}}y^{k_{2}} \rho^{i_{2}}_{2}\rho^{i_{3}}_{3})= \sigma_{1}a(x^{l_{1}}y^{l_{2}} \rho^{l_{2}}_{2}\rho^{r_{2}}_{2} \rho^{r_{3}}_{3})(x^{k_{1}}y^{k_{2}}\rho^{i_{2}}_{2} \rho^{i_{3}}_{3})= \sigma_{1}a(x^{l_{1}+k_{1}}y^{l_{2}+k_{2}}\rho^{l_{2}+r_{2}+i_{2}}_{2}\rho^{r_{3}+i_{3}}_{3}). 
\end{equation*}
It follows that: 
\begin{align*} 
w_L &= b^{-1}(x^{-k_{3}}y^{-k_{4}}\rho^{-k_{3}-j_{2}}_{2} \rho^{-j_{3}}_{3}) \sigma_{1} a(x^{l_{1}+k_{1}}y^{l_{2}+k_{2}}\rho^{l_{2}+r_{2}+i_{2}}_{2}\rho^{r_{3}+i_{3}}_{3})\\ 
&=b^{-1}\sigma_{1}(x^{-k_{3}}y^{-k_{4}}\rho^{k_{3}+j_{2}}_{2}\rho^{-k_{3}-j_{2}-j_{3}}_{3}) a(x^{l_{1}+k_{1}} y^{l_{2}+k_{2}} \rho^{l_{2}+r_{2}+i_{2}}_{2}\rho^{r_{3}+i_{3}}_{3})\\
%&=b^{-1}\sigma_{1}a(x^{-k_{3}+l_{1}}y^{-k_{4}+l_{2}} \rho_2^{-k_{4}+l_{2}}\rho^{k_{3}+j_{2}+r_{2}}_{2}\rho^{-k_{3}-j_{2}-j_{3}+r_{3}}_{3}) (x^{l_{1}+k_{1}} y^{l_{2}+k_{2}} \rho^{l_{2}+r_{2}+i_{2}}_{2} \rho^{r_{3}+i_{3}}_{3})\\
%
%
%&= b^{-1}\sigma_{1}(x^{-k_{3}}y^{-k_{4}}\rho^{k_{3}+j_{2}}_{2}\rho^{-k_{3}-j_{2}}_{3}\rho^{-j_{3}}_{3})(x^{l_{1}}y^{l_{2}}\rho^{r_{2}}_{2}\rho^{r_{3}}_{3})a(x^{k_{1}}y^{k_{2}}\rho^{i_{2}}_{2}\rho^{i_{3}}_{3})\\
%&= b^{-1}\sigma_{1}(x^{-k_{3}+l_{1}}y^{-k_{4}+l_{2}}\rho^{k_{3}+j_{2}+r_{2}}_{2}\rho^{-k_{3}-j_{2}-j_{3}+r_{3}}_{3})a(x^{k_{1}}y^{k_{2}}\rho^{i_{2}}_{2}\rho^{i_{3}}_{3})\\ 
%&= b^{-1}\sigma_{1}a(x^{-k_{3}+l_{1}}y^{-k_{4}+l_{2}}\rho^{-k_{4}+l_{2}}_{2}\rho^{k_{3}+j_{2}+r_{2}}_{2}\rho^{-k_{3}-j_{2}-j_{3}+r_{3}}_{3})(x^{k_{1}}y^{k_{2}}\rho^{i_{2}}_{2}\rho^{i_{3}}_{3})\\
&= b^{-1}\sigma_{1}a(x^{-k_{3}+l_{1}+k_{1}}y^{-k_{4}+l_{2}+k_{2}}\rho^{-k_{4}+k_{3}+j_{2}+l_{2}+r_{2}+i_{2}}_{2}\rho^{-k_{3}-j_{2}-j_{3}+r_{3}+i_{3}}_{3})\\
 %\end{array}\] 
 %So, the canonical form of $\phi(b^{-1}\sigma_{1}a)$ is \[b^{-1}\sigma_{1}a\theta^{-j_{1}+r_{1}+i_{1}}\rho^{-k_{4}+l_{2}+k_{3}+j_{2}+r_{2}+i_{2}}_{2}\rho^{-k_{3}-j_{2}-j_{3}+r_{3}+i_{3}}_{3}\rho^{-j_{4}+r_{4}+i_{4}}_{4}\cdots\rho^{-j_{n}+r_{n}+i_{n}}_{n}x^{-k_{3}+l_{1}+k_{1}}y^{-k_{4}+l_{2}+k_{2}}\]
%\[\begin{array}{l}\sigma_{1}(x^{l_{1}}y^{l_{2}}\rho^{r_{2}}_{2}\rho^{r_{3}}_{3})a(x^{k_{1}}y^{k_{2}}\rho^{i_{2}}_{2}\rho^{i_{3}}_{3})\sigma_{1}(x^{l_{1}}y^{l_{2}}\rho^{r_{2}}_{2}\rho^{r_{3}}_{3})(x^{-k_{3}}y^{-k_{4}}\rho^{-j_{2}}_{2}\rho^{-j_{3}}_{3})b^{-1}\sigma_{1}(x^{l_{1}}y^{l_{2}}\rho^{r_{2}}_{2}\rho^{r_{3}}_{3})\\
w_R&=\sigma_{1}a(x^{l_{1}+k_{1}}y^{l_{2}+k_{2}}\rho^{l_{2}+r_{2}+i_{2}}_{2}\rho^{r_{3}+i_{3}}_{3}) \sigma_{1}(x^{l_{1}-k_{3}} y^{l_{2}-k_{4}} \rho^{r_{2}-j_{2}}_{2}\rho^{r_{3}-j_{3}}_{3})b^{-1}\sigma_{1}(x^{l_{1}}y^{l_{2}}\rho^{r_{2}}_{2}\rho^{r_{3}}_{3})\\
%&=\sigma_{1}a(x^{l_{1}}y^{l_{2}}\rho^{l_{2}}_{2}\rho^{r_{2}}_{2}\rho^{r_{3}}_{3})(x^{k_{1}}y^{k_{2}}\rho^{i_{2}}_{2}\rho^{i_{3}}_{3})\sigma_{1}(x^{l_{1}-k_{3}}y^{l_{2}-k_{4}}\rho^{r_{2}-j_{2}}_{2}\rho^{r_{3}-j_{3}}_{3})b^{-1}\sigma_{1}(x^{l_{1}}y^{l_{2}}\rho^{r_{2}}_{2}\rho^{r_{3}}_{3})\\
%&=\sigma_{1}a(x^{l_{1}+k_{1}}y^{l_{2}+k_{2}}\rho^{l_{2}+r_{2}+i_{2}}_{2}\rho^{r_{3}+i_{3}}_{3})\sigma_{1}(x^{l_{1}-k_{3}}y^{l_{2}-k_{4}}\rho^{r_{2}-j_{2}}_{2}\rho^{r_{3}-j_{3}}_{3})b^{-1}\sigma_{1}(x^{l_{1}}y^{l_{2}}\rho^{r_{2}}_{2}\rho^{r_{3}}_{3})\\
%&=\sigma_{1}a\sigma_{1}(x^{l_{1}+k_{1}}y^{l_{2}+k_{2}}\rho^{-l_{2}-r_{2}-i_{2}}_{2}\rho^{l_{2}+r_{2}+i_{2}}_{3}\rho^{r_{3}+i_{3}}_{3})(x^{l_{1}-k_{3}}y^{l_{2}-k_{4}}\rho^{r_{2}-j_{2}}_{2}\rho^{r_{3}-j_{3}}_{3})b^{-1}\sigma_{1}(x^{l_{1}}y^{l_{2}}\rho^{r_{2}}_{2}\rho^{r_{3}}_{3})\\
&=\sigma_{1}a\sigma_{1}(x^{2l_{1}+k_{1}-k_{3}}y^{2l_{2}+k_{2}-k_{4}}\rho^{-l_{2}-i_{2}-j_{2}}_{2}\rho^{l_{2}+r_{2}+i_{2}+2r_{3}+i_{3}-j_{3}}_{3})b^{-1}\sigma_{1}(x^{l_{1}}y^{l_{2}}\rho^{r_{2}}_{2}\rho^{r_{3}}_{3})\\
%&=\sigma_{1}a\sigma_{1}b^{-1}(x^{2l_{1}+k_{1}-k_{3}}\rho^{2l_{1}+k_{1}-k_{3}}_{2}y^{2l_{2}+k_{2}-k_{4}}\rho^{-l_{2}-i_{2}-j_{2}}_{2}\rho^{l_{2}+r_{2}+i_{2}+2r_{3}+i_{3}-j_{3}}_{3})\sigma_{1}(x^{l_{1}}y^{l_{2}}\rho^{r_{2}}_{2}\rho^{r_{3}}_{3})\\
&=\sigma_{1}a\sigma_{1}b^{-1}(x^{2l_{1}+k_{1}-k_{3}}y^{2l_{2}+k_{2}-k_{4}}\rho^{-l_{2}-i_{2}-j_{2}+2l_{1}+k_{1}-k_{3}}_{2}\rho^{l_{2}+r_{2}+i_{2}+2r_{3}+i_{3}-j_{3}}_{3})\sigma_{1}(x^{l_{1}}y^{l_{2}}\rho^{r_{2}}_{2}\rho^{r_{3}}_{3})\\
%&=\sigma_{1}a\sigma_{1}b^{-1}\sigma_{1}(x^{2l_{1}+k_{1}-k_{3}}y^{2l_{2}+k_{2}-k_{4}}\rho^{+l_{2}+i_{2}+j_{2}-2l_{1}-k_{1}+k_{3}}_{2}\rho^{-l_{2}-i_{2}-j_{2}+2l_{1}+k_{1}-k_{3}}_{3}\rho^{l_{2}+r_{2}+i_{2}+2r_{3}+i_{3}-j_{3}}_{3})(x^{l_{1}}y^{l_{2}}\rho^{r_{2}}_{2}\rho^{r_{3}}_{3})\\
&=\sigma_{1}a\sigma_{1}b^{-1}\sigma_{1}(x^{3l_{1}+k_{1}-k_{3}}y^{3l_{2}+k_{2}-k_{4}}\rho^{l_{2}+i_{2}+j_{2}-2l_{1}-k_{1}+k_{3}+r_{2}}_{2}\rho^{-j_{2}+2l_{1}+k_{1}-k_{3}+r_{2}+3r_{3}+i_{3}-j_{3}}_{3}).
\end{align*}
%Thus, the canonical form of $\phi(\sigma_{1}a\sigma_{1}b^{-1}\sigma_{1})$ is \[\sigma_{1}a\sigma_{1}b^{-1}\sigma_{1}\theta^{-j_{1}+3r_{1}+i_{1}}\rho^{l_{2}+i_{2}+j_{2}-2l_{1}-k_{1}+k_{3}+r_{2}}_{2}\rho^{-j_{2}+2l_{1}+k_{1}-k_{3}+r_{2}+3r_{3}+i_{3}-j_{3}}_{3}\rho^{-j_{4}+3r_{4}+i_{4}}_{4}\cdots\rho^{-j_{n}+3r_{n}+i_{n}}_{n}x^{3l_{1}+k_{1}-k_{3}}y^{3l_{2}+k_{2}-k_{4}}.\]
Thus $w_L$ and $w_R$ are now in canonical form, and applying relation~(\ref{it:s2}) (the lift of relation~(\ref{it:full5}) of Theorem~\ref{th:total}) of Proposition~\ref{prop:B_{n,m}ab}, and comparing the coefficients of $x,y, \rho_2$ and $\rho_3$, we obtain parts~(\ref{it:M1})--(\ref{it:M4}) respectively of the statement, and the lemma is proved in the case $M=\mathbb{T}$.
%\begin{equation}\label{T5}l_{1}=0\quad\mbox{and}\quad l_{2}=0.\end{equation} 
%
%From~(\ref{T1}),~(\ref{T3}) and~(\ref{T5}) we prove of Proposition~\ref{exp.toro}(\ref{et1}). 
%Now, comparing exponents of $\rho_{2}$, \begin{equation}\label{T6}k_{4}=k_{1},\end{equation} %we prove Proposition~\ref{exp.toro}(\ref{et2}), 
%If $n\geq3$, we compared the exponents of $\rho_{3}$, if $n\geq3$, we have\begin{equation}\label{T7}-k_{1}=r_{2}+2r_{3}.\end{equation}% which is Proposition~\ref{exp.toro}(\ref{et3}) and the exponents of $\rho_{i}$, $4\leq i \leq n$,  which is the proof of Proposition~\ref{exp.toro}(\ref{et4}).

Now suppose that $M=\mathbb{K}$. Then:
\begin{align}
(\sigma_{1} x^{l_{1}}y^{l_{2}} \rho^{r_{2}}_{2} \rho^{r_{3}}_{3}) (a x^{k_{1}}y^{k_{2}} \rho^{i_{2}}_{2}\rho^{i_{3}}_{3})&= \sigma_{1}a(x^{l_{1}} (y^{l_{2}}x^{-2l_{2}}\rho^{l_{2}}_{2})\rho^{r_{2}}_{2}\rho^{r_{3}}_{3})(x^{k_{1}}y^{k_{2}}\rho^{i_{2}}_{2}\rho^{i_{3}}_{3})\notag\\
&= \sigma_{1}a(x^{l_{1}-2l_{2}+k_{1}}y^{l_{2}+k_{2}}\rho^{l_{2}+r_{2}+i_{2}}_{2}\rho^{r_{3}+i_{3}}_{3}),\label{eq:siga}
\end{align}
and thus:
%By~(\ref{K1}) and~(\ref{K2}), we already know that some exponents are zero, so they will note appear in the calculus below. We have that:
%\[\begin{array}{l} (x^{-k_{3}}y^{-k_{4}}\rho^{-j_{2}}_{2}\rho^{-j_{3}}_{3})b^{-1}\sigma_{1}(x^{l_{1}}y^{l_{2}}\rho^{r_{2}}_{2}\rho^{r_{3}}_{3})a(x^{k_{1}}y^{k_{2}}\rho^{i_{2}}_{2}\rho^{i_{3}}_{3})\\ 
\begin{align*}
w_L
%& =b^{-1}((x^{k_{3}}\rho^{-k_{3}}_{2})(y^{-k_{4}}x^{-2k_{4}}\rho^{k_{4}}_{2})\rho^{-j_{2}}_{2}\rho^{-j_{3}}_{3})\sigma_{1}(x^{l_{1}}y^{l_{2}}\rho^{r_{2}}_{2}\rho^{r_{3}}_{3})a(x^{k_{1}}y^{k_{2}}\rho^{i_{2}}_{2}\rho^{i_{3}}_{3})\\
&=b^{-1}(x^{k_{3}-2k_{4}}y^{-k_{4}}\rho^{-k_{3}+k_{4}-j_{2}}_{2}\rho^{-j_{3}}_{3}) \sigma_{1}a(x^{l_{1}-2l_{2}+k_{1}}y^{l_{2}+k_{2}}\rho^{l_{2}+r_{2}+i_{2}}_{2}\rho^{r_{3}+i_{3}}_{3})\\
%\sigma_{1}(x^{l_{1}}y^{l_{2}}\rho^{r_{2}}_{2}\rho^{r_{3}}_{3})a(x^{k_{1}}y^{k_{2}}\rho^{i_{2}}_{2}\rho^{i_{3}}_{3})\\
%&=b^{-1}\sigma_{1}(x^{k_{3}-2k_{4}}y^{-k_{4}}(x^{2(-k_{3}+k_{4}-j_{2})} \rho^{+k_{3}-k_{4}+j_{2}}_{2}\rho^{-k_{3}+k_{4}-j_{2}}_{3}) \rho^{-j_{3}}_{3}) a(x^{l_{1}-2l_{2}+k_{1}}y^{l_{2}+k_{2}} \rho^{l_{2}+r_{2}+i_{2}}_{2}\rho^{r_{3}+i_{3}}_{3})\\
&= b^{-1}\sigma_{1}(x^{-k_{3}-2j_{2}} y^{-k_{4}} \rho^{k_{3}-k_{4}+j_{2}}_{2} \rho^{-k_{3}+k_{4}-j_{2}-j_{3}}_{3}) a(x^{l_{1}-2l_{2}+k_{1}}y^{l_{2}+k_{2}} \rho^{l_{2}+r_{2}+i_{2}}_{2} \rho^{r_{3}+i_{3}}_{3})\\
%a(x^{k_{1}}y^{k_{2}}\rho^{i_{2}}_{2}\rho^{i_{3}}_{3})\\
%&=b^{-1}\sigma_{1}a(x^{-k_{3}-2j_{2}+l_{1}}(y^{-k_{4}+l_{2}}x^{2k_{4}-2l_{2}}\rho^{-k_{4}+l_{2}}_{2})\rho^{+k_{3}-k_{4}+j_{2}+r_{2}}_{2}\rho^{-k_{3}+k_{4}-j_{2}-j_{3}+r_{3}}_{3})(x^{k_{1}}y^{k_{2}}\rho^{i_{2}}_{2}\rho^{i_{3}}_{3})\\
&=b^{-1}\sigma_{1}a(x^{-k_{3}-2j_{2}+2k_{4}+l_{1}-2l_{2}+k_{1}}y^{-k_{4}+l_{2}+k_{2}}\rho^{k_{3}-2k_{4}+j_{2}+l_{2}+r_{2}+i_{2}}_{2}\rho^{-k_{3}+k_{4}-j_{2}-j_{3}+r_{3}+i_{3}}_{3})\\
%.\end{array}\]
%On the other hand:
%\[\begin{array}{l}\sigma_{1}(x^{l_{1}}y^{l_{2}}\rho^{r_{2}}_{2}\rho^{r_{3}}_{3})a(x^{k_{1}}y^{k_{2}}\rho^{i_{2}}_{2}\rho^{i_{3}}_{3})\sigma_{1}(x^{l_{1}}y^{l_{2}}\rho^{r_{2}}_{2}\rho^{r_{3}}_{3})(x^{-k_{3}}y^{-k_{4}}\rho^{-j_{2}}_{2}\rho^{-j_{3}}_{3})b^{-1}\sigma_{1}(x^{l_{1}}y^{l_{2}}\rho^{r_{2}}_{2}\rho^{r_{3}}_{3})\\
w_R 
%&=\sigma_{1}(x^{l_{1}}y^{l_{2}}\rho^{r_{2}}_{2}\rho^{r_{3}}_{3})a(x^{k_{1}}y^{k_{2}}\rho^{i_{2}}_{2}\rho^{i_{3}}_{3})\sigma_{1}(x^{l_{1}-k_{3}}y^{l_{2}-k_{4}}\rho^{r_{2}-j_{2}}_{2}\rho^{r_{3}-j_{3}}_{3})b^{-1}\sigma_{1}(x^{l_{1}}y^{l_{2}}\rho^{r_{2}}_{2}\rho^{r_{3}}_{3})\\
%&=\sigma_{1}a(x^{l_{1}}(y^{l_{2}}x^{-2l_{2}}\rho^{l_{2}}_{2})\rho^{r_{2}}_{2}\rho^{r_{3}}_{3})(x^{k_{1}}y^{k_{2}}\rho^{i_{2}}_{2}\rho^{i_{3}}_{3})\sigma_{1}(x^{l_{1}-k_{3}}y^{l_{2}-k_{4}}\rho^{r_{2}-j_{2}}_{2}\rho^{r_{3}-j_{3}}_{3})b^{-1}\sigma_{1}(x^{l_{1}}y^{l_{2}}\rho^{r_{2}}_{2}\rho^{r_{3}}_{3})\\
&=\sigma_{1}a(x^{l_{1}-2l_{2}+k_{1}}y^{l_{2}+k_{2}}\rho^{l_{2}+r_{2}+i_{2}}_{2}\rho^{r_{3}+i_{3}}_{3})\sigma_{1}(x^{l_{1}-k_{3}}y^{l_{2}-k_{4}}\rho^{r_{2}-j_{2}}_{2}\rho^{r_{3}-j_{3}}_{3})b^{-1}\sigma_{1}(x^{l_{1}}y^{l_{2}}\rho^{r_{2}}_{2}\rho^{r_{3}}_{3})\\
%&=\sigma_{1}a\sigma_{1}(x^{l_{1}-2l_{2}+k_{1}}y^{l_{2}+k_{2}}(x^{2(l_{2}+r_{2}+i_{2})}\rho^{-l_{2}-r_{2}-i_{2}}_{2}\rho^{l_{2}+r_{2}+i_{2}}_{3})\rho^{r_{3}+i_{3}}_{3})(x^{l_{1}-k_{3}}y^{l_{2}-k_{4}}\rho^{r_{2}-j_{2}}_{2}\rho^{r_{3}-j_{3}}_{3})b^{-1}\sigma_{1}(x^{l_{1}}y^{l_{2}}\rho^{r_{2}}_{2}\rho^{r_{3}}_{3})\\
&=\sigma_{1}a\sigma_{1}(x^{2l_{1}+k_{1}+2r_{2}+2i_{2}-k_{3}}y^{2l_{2}+k_{2}-k_{4}}\rho^{-l_{2}-i_{2}-j_{2}}_{2}\rho^{l_{2}+r_{2}+i_{2}+2r_{3}+i_{3}-j_{3}}_{3})b^{-1}\sigma_{1}(x^{l_{1}}y^{l_{2}}\rho^{r_{2}}_{2}\rho^{r_{3}}_{3})\\
%&=\sigma_{1}a\sigma_{1}b^{-1}((x^{-2l_{1}-k_{1}-2r_{2}-2i_{2}+k_{3}}\rho^{2l_{1}+k_{1}+2r_{2}+2i_{2}-k_{3}}_{2})(y^{2l_{2}+k_{2}-k_{4}}x^{2(2l_{2}+k_{2}-k_{4})}\rho^{-2l_{2}-k_{2}+k_{4}}_{2})\\
%& \rho^{-l_{2}-i_{2}-j_{2}}_{2}\rho^{l_{2}+r_{2}+i_{2}+2r_{3}+i_{3}-j_{3}}_{3})\sigma_{1}(x^{l_{1}}y^{l_{2}}\rho^{r_{2}}_{2}\rho^{r_{3}}_{3})\\
&=\sigma_{1}a\sigma_{1}b^{-1}(x^{-2l_{1}-k_{1}-2r_{2}-2i_{2}+k_{3}+4l_{2}+2k_{2}-2k_{4}}y^{2l_{2}+k_{2}-k_{4}}\rho^{2l_{1}+k_{1}+2r_{2}+i_{2}-k_{3}-3l_{2}-k_{2}+k_{4}-j_{2}}_{2}\cdot\\
& \quad \rho^{l_{2}+r_{2}+i_{2}+2r_{3}+i_{3}-j_{3}}_{3})\sigma_{1}(x^{l_{1}}y^{l_{2}}\rho^{r_{2}}_{2}\rho^{r_{3}}_{3})\\
%&=\sigma_{1}a\sigma_{1}b^{-1}\sigma_{1}(x^{-2l_{1}-k_{1}-2r_{2}-2i_{2}+k_{3}+4l_{2}+2k_{2}-2k_{4}}y^{2l_{2}+k_{2}-k_{4}}\\
%&(x^{2(2l_{1}+k_{1}+2r_{2}+i_{2}-k_{3}-3l_{2}-k_{2}+k_{4}-j_{2})}\rho^{-2l_{1}-k_{1}-2r_{2}-i_{2}+k_{3}+3l_{2}+k_{2}-k_{4}+j_{2}}_{2}\rho^{2l_{1}+k_{1}+2r_{2}+i_{2}-k_{3}-3l_{2}-k_{2}+k_{4}-j_{2}}_{3})\\
%&\rho^{l_{2}+r_{2}+i_{2}+2r_{3}+i_{3}-j_{3}}_{3})(x^{l_{1}}y^{l_{2}}\rho^{r_{2}}_{2}\rho^{r_{3}}_{3})\\
&=\sigma_{1}a\sigma_{1}b^{-1}\sigma_{1}(x^{3l_{1}+k_{1}+2r_{2}-k_{3}-2l_{2}-2j_{2}}y^{3l_{2}+k_{2}-k_{4}}\rho^{-2l_{1}-k_{1}-r_{2}-i_{2}+k_{3}+3l_{2}+k_{2}-k_{4}+j_{2}}_{2}\cdot\\
& \quad \rho^{2l_{1}+k_{1}+3r_{2}+2i_{2}-k_{3}-2l_{2}-k_{2}+k_{4}-j_{2}+3r_{3}+i_{3}-j_{3}}_{3}).
\end{align*}

Thus $w_L$ and $w_R$ are now in canonical form, and applying relation~(\ref{it:s2}) (the lift of relation~(\ref{it:full5}) of Theorem~\ref{th:total}) of Proposition~\ref{prop:B_{n,m}ab}, and comparing the coefficients of $x$ and $y$, we obtain parts~(\ref{it:M1}) and~(\ref{it:M2}) respectively of the statement. 
%, and the lemma is proved in the case $M=\mathbb{K}$.
%Now, comparing the exponents of $x$,\begin{equation}\label{K6} k_{4}=l_{1}+r_{2},\end{equation} from the exponent of $y$ follows that \begin{equation}\label{K6.1}l_{2}=0\end{equation} 
Comparing the coefficients of $\rho_{2}$ and using part~(\ref{it:M2}) of the statement, we obtain:
\begin{equation}\label{eq:K7}
-k_{4}=-2(l_{1}+r_{2})-k_{1}-2i_{2}+k_{2}.
\end{equation} 
It follows from part~(\ref{it:M1}) of the statement and~(\ref{eq:K7}) that: \begin{equation}\label{eq:K.7}
k_{4}=-k_{1}+k_{2}-2i_{2}.
\end{equation} 
If $n\geq 3$,  comparing the coefficients of $\rho_{3}$ and using parts~(\ref{it:M1}) and~(\ref{it:M2}) of the statement, we see that: 
\begin{equation}\label{eq:K8} 
0=2k_4+k_{1}+r_{2}+2i_{2}-k_{2}+2r_{3}.
\end{equation}
%Thus, using~(\ref{K6}),~(\ref{K.7}) and~(\ref{K8}) we proof Proposition~\ref{exp.klein}(\ref{ek4}).
To obtain parts~(\ref{it:M3}) and~(\ref{it:M4}) of the statement, we analyse the image of relation~(\ref{it:full6}) of Theorem~\ref{th:total} by $\widehat{\phi}$ using the fact that $l_2=0$. We need only analyse the coefficients of $y$ and $\rho_{2}$, since these are the only generators of $B_m(M\setminus\left\{x_1,\ldots,x_n\right\})\ab$ that do not commute both with $a$ and $\sigma_{1}$. Let $w_L'=a(y^{k_{2}}\rho^{i_{2}}_{2}) \sigma_{1}(\rho^{r_{2}}_{2})a(y^{k_{2}}\rho^{i_{2}}_{2})\sigma_{1}(\rho^{r_{2}}_{2})$, and let $w_R'=\sigma_{1}(\rho^{r_{2}}_{2})a(y^{k_{2}}\rho^{i_{2}}_{2})\sigma_{1}(\rho^{r_{2}}_{2})a(y^{k_{2}}\rho^{i_{2}}_{2})$. By a computation similar to that of~(\ref{eq:siga}), we have $a(y^{k_{2}}\rho^{i_{2}}_{2}) \sigma_{1}(\rho^{r_{2}}_{2})= a\sigma_{1}(x^{2i_{2}}y^{k_{2}} \rho^{r_{2}-i_{2}}_{2} \rho^{i_{2}}_{3})$ and $\sigma_{1}(\rho^{r_{2}}_{2}) a(y^{k_{2}}\rho^{i_{2}}_{2}) = \sigma_{1} a(y^{k_{2}} \rho_{2}^{r_{2}+i_{2}})$, so:
%In the case of the Klein bottle we have to analyze one more relation, the relation $a\sigma_{1}a\sigma_{1}=\sigma_{1}a\sigma_{1}a$, we are going to use that $l_{2}=0$ by~(\ref{K6.1}) and analyze only the elements $y$ and $\rho_{2}$, because are the only ones that do not commutes with $a$ and $\sigma_{1}$.% by relations~(\ref{it:s4}) and~(\ref{it:s8}), the elements $x,\theta,\rho_{4},\ldots,\rho_{n}$ commutes with $a$ and $\sigma_{1}$, so the exponent of each of them that appears in the canonical form is the sum of the exponents that appears in~(\ref{z1}) and~(\ref{z2}). The element $\rho_{3}$ also commutes with $a$ and $\sigma_{1}$, so $2r_{3}+2i_{3}$ will appears in the exponent of $\rho_{3}$ in the canonical form, but we have to look more carefully this case because to pass $\rho_{2}$ through $\sigma_{1}$ we add $\rho_{3}$ by relation~(\ref{it:s8}). And we also have to analyze what will happen to the exponents of $y$ (that commutes with $\sigma_{1}$ by~(\ref{it:s4}) but do not commute with $a$ by~(\ref{it:s5})) and $\rho_{2}$ (that commutes with $a$ by~(\ref{it:s4}) but do not commute with $\sigma_{1}$ by~(\ref{it:s8})). We have that,
%
%\[\begin{array}{l} a(y^{k_{2}}\rho^{i_{2}}_{2})\sigma_{1}(\rho^{r_{2}}_{2})a(y^{k_{2}}\rho^{i_{2}}_{2})\sigma_{1}(\rho^{r_{2}}_{2})\\
\begin{align*}
w_L' 
%&=a\sigma_{1}(y^{k_{2}}(x^{2i_{2}}\rho^{-i_{2}}_{2}\rho^{i_{2}}_{3}))(\rho^{r_{2}}_{2})a(y^{k_{2}}\rho^{i_{2}}_{2})\sigma_{1}(\rho^{r_{2}}_{2})\\ 
&= a\sigma_{1}(x^{2i_{2}}y^{k_{2}}\rho^{-i_{2}+r_{2}}_{2}\rho^{i_{2}}_{3}) a\sigma_{1}(x^{2i_{2}}y^{k_{2}}\rho^{-i_{2}+r_{2}}_{2}\rho^{i_{2}}_{3})\\
%&= a\sigma_{1}a(x^{2i_{2}}(y^{k_{2}}x^{-2k_{2}}\rho^{k_{2}}_{2})\rho^{-i_{2}+r_{2}}_{2}\rho^{i_{2}}_{3})(y^{k_{2}}\rho^{i_{2}}_{2})\sigma_{1}(\rho^{r_{2}}_{2})\\
&=a\sigma_{1}a(x^{2(i_{2}-k_{2})} y^{k_{2}} \rho_{2}^{k_{2}+r_{2}-i_2} \rho_{3}^{i_{2}})\sigma_{1}(x^{2i_{2}}y^{k_{2}}\rho^{r_{2}-i_{2}}_{2}\rho^{i_{2}}_{3})%\\
%&= a\sigma_{1}a\sigma_{1}(x^{2(i_{2}-k_{2})}y^{2k_{2}}(x^{2(k_{2}+r_{2})}\rho^{-k_{2}-r_{2}}_{2}\rho^{k_{2}+r_{2}}_{3})\rho^{i_{2}}_{3})(\rho^{r_{2}}_{2})\\
= a\sigma_{1}a\sigma_{1}(x^{2(i_{2}+r_{2})}y^{2k_{2}}\rho^{-k_{2}}_{2}\rho^{i_{2}+k_{2}+r_{2}}_{3})\\
%. \end{array}\]
%On the other hand
%\[\begin{array}{l} \sigma_{1}(\rho^{r_{2}}_{2})a(y^{k_{2}}\rho^{i_{2}}_{2})\sigma_{1}(\rho^{r_{2}}_{2})a(y^{k_{2}}\rho^{i_{2}}_{2})\\
w_R' &=\sigma_{1} a(y^{k_{2}} \rho_{2}^{r_{2}+i_{2}})\sigma_{1} a(y^{k_{2}} \rho_{2}^{r_{2}+i_{2}})%\\
=\sigma_{1}a\sigma_{1}(x^{2(r_{2}+i_{2})} y^{k_{2}} \rho^{-i_{2}-r_2}_{2} \rho^{r_{2}+i_{2}}_{3}) a(y^{k_{2}}\rho^{r_2+i_{2}}_{2})\\
%&=\sigma_{1}a\sigma_{1} a(x^{2(r_{2}+i_{2})}(y^{k_{2}}x^{-2k_{2}}\rho^{k_{2}}_{2})\rho^{-i_{2}}_{2}\rho^{r_{2}+i_{2}}_{3})(y^{k_{2}}\rho^{i_{2}}_{2})\\
&=\sigma_{1}a\sigma_{1} a(x^{2(r_{2}+i_{2}-k_{2})}y^{2k_{2}}\rho^{k_{2}}_{2}\rho^{r_{2}+i_{2}}_{3}).
\end{align*}
%\end{array}\]
Part~(\ref{it:M3}) of the statement follows by comparing the coefficients of $x$ and~(\ref{eq:K.7}), and part~(\ref{it:M4}) is a consequence of~(\ref{eq:K8}) and part~(\ref{it:M3}). 
%we can conclude that \begin{equation}\label{K8.1} k_{2}=0.\end{equation}
\end{proof}

\begin{lem}\label{lem:phiartin}\mbox{}
% \textcolor{blue}{!!!C!!! Sugestao: 
% \comj{The following rewritten slightly to avoid some repetition:}\textcolor{blue}{!!!C!!! OK}
\begin{enumerate}[(a)]
\item\label{it:phiartin0} 
% \comj{This is the statement of a previous version of Lemma~\ref{lem:exp.M1}:} 
Let $M=\mathbb T$ or $\mathbb K$, and let $n\geq 4$. Then $r_{j,k}=0$ for all $1\leq j\leq n-1$ and $k=2,\ldots, j-1,j+3,\ldots, n$.
\item\label{it:phiartina} Let $n\geq 3$. If $M=\mathbb T$ (resp.\ $M=\mathbb K$), and $2\leq i \leq n-2$,
% $1\leq i \leq n-2$ (resp.\ $2\leq i \leq n-2$) \comj{Should this be $2\leq i \leq n-2$ in both cases? \comc{15abril - sim! concordo com o enunciado proposto} (I modified slightly what follows to take into account the case $i=1$).}, 
then $l_{i,j}=l_{i+1,j}=l_{1,j}$ (resp.\ $l_{i,j}=l_{i+1,j}$) for $j=1,2$, $r_{1,0} \equiv r_{i,0}\equiv r_{i+1,0} \bmod{2}$, and: 
\begin{equation}\label{eq:ar4}
-2r_{i+1,i+1}-r_{i+1,i+2}=-2r_{i,i}-r_{i,i+1}. 
\end{equation}
Further: 
% \comj{Carolina suggested adding this, we use it a bit later on:}
\begin{equation}\label{eq:ar5}
\text{$-2r_{i,i}-r_{i,i+1}=2r_{3}+r_{2}$ for all $i=2,\ldots,n-1$.}
\end{equation}
% \comc{acho que poderia acrescentar no enunciado: $$2r_{1,3}+r_{1,2}=-2r_{i,i}-r_{i,i+1}, i=2,\ldots,n-1$$ que \'{e} consequencia da equa\c{c}\~{a}o acima, mas usa tambem a equa\c{c}\~{a}o (3.14) no caso i=1, pois vamos usar essas duas informa\c{c}\~{o}es na equa\c{c}\~{a}o (3.19)}
\item\label{it:phiartinb} If $M=\mathbb{K}$, $l_{k,1}=0$ for all $2\leq k\leq n-1$.
\end{enumerate}
\end{lem}

\begin{proof}
% \comj{The following rewritten somewhat to take account of the new version of Lemma~\ref{lem:exp.M1} and the addition of part~(\ref{it:phiartin0}).} 
We first prove part~(\ref{it:phiartin0}). Let $M=\mathbb T$ or $\mathbb K$. Recall from the proof of Proposition~\ref{prop:kernelab} that $\rho_1=C_{1,n+1}$ is equal to $1$ (resp.\ to $x^{2}$) if $M=\mathbb{T}$ (resp.\ if $M=\mathbb{K}$). First let $n\geq 4$, and let $1\leq i,j\leq n-1$ be such that $\abs{i-j}\geq 2$. Applying Proposition~\ref{prop:kernelab} and~(\ref{eq:commba1}), we have:
\begin{equation*}
\rho_{i}^{r_{j,i+1}} \rho_{i+1}^{-2r_{j,i+1}} \rho_{i+2}^{r_{j,i+1}}
\rho_{j}^{-r_{i,j+1}} \rho_{j+1}^{2r_{i,j+1}} \rho_{j+2}^{-r_{i,j+1}}=1,
\end{equation*}
which is in canonical form (possibly up to permutation of some of the factors). Comparing the coefficients of $\rho_{i+1}$ (resp.\ $\rho_{j+1}$) if $i<j$ (resp.\ $i>j$) and using once more Proposition~\ref{prop:kernelab}, we see that $r_{j,i+1}=0$ (resp.\ $r_{i,j+1}=0$). So for all $1\leq j\leq n-1$, $r_{j,k}=0$ for all $k=2,\ldots,j-1$ (resp.\ for all $k=j+3,\ldots,n$), which proves part~(\ref{it:phiartin0}).

Now let $n\geq 3$, and let $1\leq i\leq n-2$. 
% Follows from relations~~(\ref{eq:v1canon}) and~(\ref{eq:v2canon}).
% if $M=\mathbb{T}$ (resp.\ $M=\mathbb{K}$). 
Using Proposition~\ref{prop:kernelab}, equation~(\ref{eq:commba2}) may be written as:
\begin{equation}\label{eq:rhorho}
x^{l_{i+1,1}-l_{i,1}} y^{l_{i+1,2}-l_{i,2}} \sigma^{r_{i+1,0}-r_{i,0}} \prod_{\substack{k=1\\ k\neq i+1,i+2}}^n  \rho_{k}^{r_{i+1,k}-r_{i,k}} \ldotp \rho_{i}^{-r_{i,i+1}-r_{i,i+2}} 
\rho_{i+1}^{\rho}
\rho_{i+2}^{-\rho}
\rho_{i+3}^{\rho'}=1,
\end{equation}
% \comc{apenas para seguir a nota\c{c}\~{a}o de ~(\ref{eq:commba2}), teria que passar $\delta$ pro outro lado da igualdade, e acho que seria $x^{l_{i+1,1}-l_{i,1}} y^{l_{i+1,2}-l_{i,2}} \sigma^{r_{i+1,0}-r_{i,0}}$ (s\'{o} alterar o sinal na primeira parte)}
where $\rho=2r_{i,i+2}+ 2r_{i+1,i+1}+r_{i,i+1}+r_{i+1,i+2}$ and $\rho'=r_{i+1,i+1}+r_{i+1,i+2}$. Now~\reqref{rhorho} is in canonical form,
%\comj{This is what was written before. The expressions involving $v_k$ have been replaced by~(\ref{eq:commba2}). I am having some difficulty with some of the cases. If $n\geq 3$ and $i=n-2$ then $i+3=n+1$, and we can't compare coefficients of $\rho_{i+3}$. Probably I am missing something.} Since the canonical form of $v_k$, $k=1,2$, given in~(\ref{eq:v1canon}) and~(\ref{eq:v2canon}) does not contain $x,y$ or $\sigma$, we have $v_1=v_2$ and $w_1=w_2$. 
and using the fact that $r_{i+1,i}= r_{i,i+3}=0$ by part~(\ref{it:phiartin0}), and comparing the coefficients of $\rho_i, \rho_{i+1}$, $\rho_{i+3}$, $x$, $y$ and $\sigma$,
% in the former (resp.\ latter) relation, and using relation~(\ref{it:s2})  of Proposition~\ref{prop:B_{n,m}ab} (which is the lift of  relation~(\ref{it:full2}) of Theorem~\ref{th:total}), 
we deduce that: 
% \comj{If I am not mistaken, the relations are valid for different values of $i$, so I suggest giving them in each case:}\comc{15abril - concordo}
\begin{gather}
r_{i,i}+r_{i,i+1}+r_{i,i+2}=0 \; \text{for $2\leq i\leq n-2$}\label{eq:ar1}\\
-2r_{i+1,i+1}-r_{i+1,i+2}=2r_{i,i+2}+r_{i,i+1} \; \text{for $1\leq i\leq n-2$}\label{eq:ar2}\\
r_{i+1,i+1}+r_{i+1,i+2}+r_{i+1,i+3}=0 \;\text{for $1\leq i\leq n-3$}.\notag\\ %\label{eq:ar3}\\
\text{$l_{i,j}=l_{i+1,j}$, where $j=1,2$ and $1\leq i\leq n-2$ (resp.\ $2\leq i\leq n-2$) if $M=\mathbb T$ (resp.\ $M=\mathbb K$)}\label{eq:ar0i}\\
\text{$r_{i,0} \equiv r_{i+1,0} \bmod{2}$ for $1\leq i\leq n-2$,} \notag %\label{eq:ar0}
\end{gather}
from which we obtain the relations involving $l_{i,j}$ and $r_{i,0}$. Equation~\reqref{ar4} follows from~\reqref{ar1} and~\reqref{ar2} for all $2\leq i\leq n-2$. Replacing $i$ by $i-1$ in~\reqref{ar4} and using induction on $i$, we see that $-2r_{i,i}- r_{i,i+1}= -2r_{2,2}-r_{2,3}$ for all $2\leq i\leq n-1$. Equation~\reqref{ar5} then follows from this by applying~\reqref{ar2} with $i=1$.
% \textcolor{red}{!!!D!!! December21-Seems that after Carolin's message it suffices to  write here:  "exact for $i=n-2$ where we get only the equations (12)  and  (13). But these two equations are enough to show to prove (11) for  $i=n-2$." See if makes sense and please proceed to change to something like above} 
% \comj{Do we actually need~\reqref{ar3} to prove~\reqref{ar4}? Doesn't~\reqref{ar4} follow just from~\reqref{ar1} and~\reqref{ar2}?}
This proves part~(\ref{it:phiartina}).

% of the statement follows from~(\ref{eq:ar0}). 
To prove part~(\ref{it:phiartinb}), let $M=\mathbb{K}$. Since $\rho_1=x^2$, it follows by taking $i=1$ in~\reqref{rhorho} and comparing the coefficients of $\rho_1$ that $l_{2,1}-l_{1}+2(r_{2,1}-r_{1}-r_{2}-r_{3})=0$. Since $r_{1}=0$ by definition and $r_{2,1}=0$ by part~(\ref{it:phiartina}), we see that: 
\begin{equation}\label{eq:i1K}
\text{$2(r_2+r_3)+l_1=l_{2,1}$, $l_{1,2}=l_{2,2}$, and $r_{1,0} \equiv r_{2,0} \bmod{2}$,}
\end{equation}
where we also compare the coefficients of $y$ and $\sigma$ in~\reqref{rhorho}. Now by Lemma~\ref{lem:exp.M}(\ref{it:M1}) and~(\ref{it:M4}), $l_1=k_4-r_2=-2(r_2+r_3)$, and we deduce from~(\ref{eq:i1K}) that $l_{2,1}=0$, and then from~(\ref{eq:ar0i}) that $l_{k,1}=0$ for all $2\leq k\leq n-1$. This proves part~(\ref{it:phiartinb}) of the statement.
% If $M=\mathbb T$ (resp.\ $M=\mathbb K$), and $1\leq i \leq n-2$ (resp.\ $2\leq i \leq n-2$), equation~(\ref{eq:ar4}) then follows using~(\ref{eq:ar1}) and~(\ref{eq:ar2}), and this completes the proof of the lemma.
\end{proof}

%\begin{prop}Falar que muitos na imagem de sigma sao zero.
%\end{prop}
%%%%%%%%%%%%%%%%%%%%%%%%%%%%%%%%%%%%%%%%%%%%%%%%%%%%%%%%%%%%%%%%%%%%%%%

We now complete the proof of Theorem~\ref{th:FNsplits}.%  of the surface relation~(\ref{it:full8}) of Theorem~\ref{th:total}. 

\begin{proof}[Proof of Theorem~\ref{th:FNsplits}] 
%\comc{ja tinha no texto, mas alguns comentarios a mais para ajustar a notacao} \comj{I think that we can unify (to some extent) the cases $n=2$ and $n\geq 3$. See if it is OK (we can de-comment the previous version if it isn't).} \comc{15abril - concordo com essa versao}
% We now compare the exponent of $C_{2,n+1}$ in~\reqref{final} in the two cases of \reexs{2cases}. 
Recall that $\rho_{1}=1$ (resp.\ $\rho_{1}=x^{2}$) if $M=\mathbb{T}$ (resp.\ if $M=\mathbb{K}$). We compare the exponents of $\rho_{2}$ in~\reqref{final}. Using Proposition~\ref{prop:kernelab}, we obtain: 
\begin{equation}\label{eq:coeffcaseab}
\gamma+m=\alpha + \delta.
\end{equation}
where $\alpha,\gamma$ and $\delta$ are given by~\reqref{defalpha} (where $t=n$),~\reqref{defgamma} and~\reqref{defdelta} respectively. Since $s=0$, $\gamma=0$ by~\reqref{defgamma}, and using~\reqref{defdelta}, equation~(\ref{eq:exp}) and Lemma~\ref{lem:exp.M}(\ref{it:M3}), we see that:
\begin{align}
\delta &= \begin{cases}
\beta_{1,2}+\alpha_{1,1} =k_{4}+k_{1} & \text{if $M=\mathbb{T}$}\\
2(\alpha_{1,2}-x_{1,1}-x_{1,2})-\alpha_{1,1}+\beta_{1,2}= -2i_{2}-k_{1}+k_{4} & \text{if $M=\mathbb{K}$}
\end{cases}\notag\\
& =2k_{4},\label{eq:calcdeltaa}
\end{align}
using the fact that in the case $M=\mathbb{K}$ that $x_{1,2}=i_{2}$, and $x_{1,1}=0$, and that $\alpha_{1,2}=k_{2}=0$ by Lemma~\ref{lem:exp.M}(\ref{it:M3}). If $n=2$ then $\alpha=0$ by~\reqref{defalpha}, and it follows from~\reqref{coeffcaseab} and~\reqref{calcdeltaa} that $m=nk_4$ as required. Suppose now that $n\geq 4$. Using~\reqref{defalpha}, \relem{phiartin}(\ref{it:phiartin0}) and~\reqref{ar4}, we have:
\begin{equation}\label{eq:calcalphaa}
\alpha= \sum_{l=2}^{n-1} (2r_{l,l}+r_{l,l+1})=(n-2)(2r_{2,2}+r_{2,3}).
\end{equation}
Observe that~\reqref{calcalphaa} also holds if $n=3$. Applying~\reqref{ar5} and \relem{exp.M}(\ref{it:M4}), for all $n\geq 3$, we obtain:
\begin{equation*}%\label{eq:calcalphab}
\alpha= (n-2)(2r_{2,2}+r_{2,3})= -(n-2)(2r_{3}+r_{2})=(n-2)k_4,
\end{equation*}
which using~\reqref{coeffcaseab} and~\reqref{calcdeltaa} implies that $m=nk_4$. This completes the proof of the theorem.
\end{proof}

\section{Generalisation to several factors}

In this section, we turn our attention to the case of the short exact sequence~\reqref{p} involving mixed braid groups with more than two factors. 

\subsection{The case $q=k-1$}

We start by proving Theorem~\ref{th:split1} which is a straightforward consequence of Theorem~\ref{th:FNsplits}.

\begin{proof}[Proof of Theorem~\ref{th:split1}] 
Let $M=\mathbb{T}$ or $\mathbb{K}$, and assume that $q=k-1$, so the homomorphism in the short exact sequence~(\ref{eq:p}) is $p_{\ast}\colon\thinspace B_{n_1,\ldots,n_{k}}(M) \to B_{n_1}(M)$. 

Suppose first that $n_1$ divides $n_i$ for $i=2,\ldots,k$, in other words there exists $l_{i}\in\mathbb{N}$ such that $n_{i}=l_{i}n_{1}$. In a manner similar to that of the proof of Theorem~\ref{th:FNsplits}, we may construct a cross-section on the level of configuration spaces using the non-vanishing vector field of $M$ as follows. Let $s_{n,l}\colon \thinspace D_n(M) \to D_{ln}(M\setminus \left\lbrace x_{1},\ldots, x_{n}\right\rbrace)$ be the map defined by:
\begin{equation}\label{eq:sec2}
s_{n,l}(x)=(s_{1}(x),\ldots,s_{n}(x)),
\end{equation}
where for $i=1,\ldots,n$, the map $s_{i}$ is defined by~(\ref{eq:sec1}). Then the map $s \colon \thinspace D_{n_{1}}(M) \longrightarrow D_{n_{1},l_{2}n_{1},\ldots,l_{k}n_{1}}(M)$ defined by $s(x)=(x,s_{n_{1},l_{2}}(x),\ldots, s_{n_{1},l_{k}}(x))$ for all $x\in D_{n_{1}}(M)$ is well defined and continuous, and it is a cross-section for $p$. Hence the induced homomorphism $s_{\ast} \colon \thinspace B_{n_{1}}(M) \longrightarrow B_{n_{1},l_{2}n_{1},\ldots, l_{k}n_{1}}(M)$ is a section for $p_{\ast}$. 

Conversely, suppose that $s_{\ast} \colon \thinspace B_{n_{1}}(M)\longrightarrow B_{n_1,\ldots,n_{k}}(M)$ is a section for $p_{\ast}$. For $i=2,\ldots, k$, let $(p_i)_{\ast} \colon \thinspace B_{n_1,n_{i}}(M)\longrightarrow B_{n_{1}}(M)$ (resp.\ $q_{i}\colon \thinspace B_{n_1,\ldots,n_{k}}(M)\longrightarrow B_{n_{1},n_{i}}(M)$) be the projection obtained by forgetting the second block (resp.\ all of the blocks with the exception of the $1$\textsuperscript{st} and the $i$\textsuperscript{th} block). Then $(p_i)_{\ast}\circ q_i= p_{\ast}$, and it follows that $q_{i}\circ s_{\ast}$ is a section for $(p_i)_{\ast}$. So by Theorem~\ref{th:FNsplits}, $n_{1}$ divides $n_{i}$ for all $i=2,\ldots,k$ as required.
\end{proof}

%\comj{I added this to justify some of what follows.} {\textcolor{blue}{!!!C!!! ok}} 
We shall make use of the following lemma to prove Theorem~\ref{th:split2}.

\begin{lem}\label{lem:induced}
%\comj{The statement and the proof have been augmented, in a form that will be used in two pages' time.} 
Let $k\geq 2$, let $n_{1},\ldots,n_{k}\in \mathbb{N}$, let $s=\sum_{i=2}^{k-1} n_{i}$ and $n=n_{1}+s$. Let $\iota_{1} \colon\thinspace B_{n_{1},\ldots,n_{k}}(M) \to B_{n_{1},s,n_{k}}(M)$ and $\iota_{2} \colon\thinspace B_{n_{1},s,n_{k}}(M) \to B_{n,n_{k}}(M)$ denote the corresponding inclusions, and let $K=B_{n_{k}}(M\setminus \brak{x_{1},\ldots, x_{n}})$. Then $\iota_{1}$ and $\iota_{2}$ induce injective homomorphisms $\widehat{\iota}_{1} \colon\thinspace B_{n_{1},\ldots,n_{k}}(M)/\Gamma_{2}(K) \to B_{n_{1},s,n_{k}}(M)/\Gamma_{2}(K)$ and $\widehat{\iota}_{2} \colon\thinspace B_{n_{1},s,n_{k}}(M)/\Gamma_{2}(K) \to B_{n,n_{k}}(M)/\Gamma_{2}(K)$.
% let $f\colon\thinspace B_{n,n_{k}}(M) \longrightarrow B_{n}(M)$ be given geometrically by forgetting the last $n_{k}$ strings, and let $K=\ker{f}$. Then $\iota$ induces an injective homomorphism $\widehat{\iota} \colon\thinspace B_{n_{1},\ldots,n_{k}}(M)/\Gamma_{2}(K) \to B_{n,n_{k}}(M)/\Gamma_{2}(K)$.
\end{lem}

\begin{proof}
Let $f\colon\thinspace B_{n,n_{k}}(M) \longrightarrow B_{n}(M)$ be the homomorphism given geometrically by forgetting the last $n_{k}$ strings, and let $g\colon\thinspace B_{n_{1},s,n_{k}}(M) \longrightarrow B_{n_{1},s}(M)$ and $h\colon\thinspace B_{n_{1},\ldots,n_{k}}(M) \longrightarrow B_{n_{1},\ldots,n_{k-1}}(M)$ denote the restriction of $f$ to $B_{n_{1},\ldots,n_{k}}(M)$ and to $B_{n_{1},s,n_{k}}(M)$ respectively. Then $K= \ker{f}= \ker{g}=\ker{h}$, and we have the following commutative diagram of short exact sequences:
\begin{equation*}%\label{eq:pmmminus}
\begin{tikzcd}%[cramped, sep=scriptsize]
1 \arrow{r} & K \arrow{r} \arrow[equal]{d} & B_{n_{1},\ldots,n_{k}}(M) \arrow{r}{h} \arrow{d}{\iota_{1}} & B_{n_{1},\ldots,n_{k-1}}(M) \arrow{r} \arrow{d}{\overline{\iota}_{1}} & 1\\
1 \arrow{r} & K \arrow{r} \arrow[equal]{d} & B_{n_{1},s,n_{k}}(M) \arrow{r}{g} \arrow{d}{\iota_{2}}  & B_{n_{1},s}(M) \arrow{r} \arrow{d}{\overline{\iota}_{2}} & 1\\
1 \arrow{r} & K \arrow{r} & B_{n,n_{k}}(M) \arrow{r}{f}  & B_{n}(M) \arrow{r} & 1,
%
%&  1 \arrow{d} & 1\arrow{d} &&\\
%& P_{1}(M\setminus\left\{x_1,\ldots,x_{n+m}\right\}) \arrow[equal]{r} \arrow{d} & P_{1}(M\setminus\left\{x_1,\ldots,x_{n+m}\right\}) \arrow{d} & &\\
%1 \arrow{r} & P_{m+1}(M\setminus\left\{x_1,\ldots,x_n\right\}) \arrow{r} \arrow{d}{p_{\ast}} & P_{n+m+1}(M) \arrow{r}{p_{\ast}} \arrow{d}{p_{\ast}} & P_{n}(M) \arrow{r} \arrow[equal]{d} & 1\\
%1 \arrow{r} & P_{m}(M\setminus\left\{x_1,\ldots,x_n\right\}) \arrow{r} \arrow{d} & P_{n+m}(M) \arrow{r}{p_{\ast}} \arrow{d}  & P_{n}(M) \arrow{r} & 1,\\
%&  1 & 1 && 
\end{tikzcd}
\end{equation*}
where $\overline{\iota}_{1} \colon\thinspace B_{n_{1},\ldots,n_{k-1}}(M) \to B_{n_{1},s}(M)$ and $\overline{\iota}_{2} \colon\thinspace B_{n_{1},s}(M) \to B_{n}(M)$ denote the corresponding inclusions. This gives rise to the following commutative diagram of short exact sequences:
\begin{equation}\label{eq:commdiaginduced}
\begin{tikzcd}%[cramped, sep=scriptsize]
1 \arrow{r} & K\ab \arrow{r} \arrow[equal]{d} & B_{n_{1},\ldots,n_{k}}(M)/\Gamma_{2}(K) \arrow{r}{\widehat{h}} \arrow{d}{\widehat{\iota}_{1}} & B_{n_{1},\ldots,n_{k-1}}(M) \arrow{r} \arrow{d}{\overline{\iota}_{1}} & 1\\
1 \arrow{r} & K\ab \arrow{r} \arrow[equal]{d} & B_{n_{1},s,n_{k}}(M)/\Gamma_{2}(K) \arrow{r}{\widehat{g}} \arrow{d}{\widehat{\iota}_{2}}  & B_{n_{1},s}(M) \arrow{r} \arrow{d}{\overline{\iota}_{2}} & 1\\
1 \arrow{r} & K\ab \arrow{r} & B_{n,n_{k}}(M)/\Gamma_{2}(K) \arrow{r}{\widehat{f}}  & B_{n}(M) \arrow{r} & 1,\\
%%
%1 \arrow{r} & K\ab \arrow{r} \arrow[equal]{d} & B_{n_{1},\ldots,n_{k}}(M)/\Gamma_{2}(K) \arrow{r}{\widehat{g}} \arrow{d}{\widehat{\iota}} & B_{n_{1},\ldots,n_{k-1}}(M) \arrow{r} \arrow{d}{\overline{\iota}} & 1\\
%1 \arrow{r} & K\ab \arrow{r} & B_{n,n_{k}}(M)/\Gamma_{2}(K) \arrow{r}{\widehat{f}}  & B_{n}(M) \arrow{r} & 1,
%
%&  1 \arrow{d} & 1\arrow{d} &&\\
%& P_{1}(M\setminus\left\{x_1,\ldots,x_{n+m}\right\}) \arrow[equal]{r} \arrow{d} & P_{1}(M\setminus\left\{x_1,\ldots,x_{n+m}\right\}) \arrow{d} & &\\
%1 \arrow{r} & P_{m+1}(M\setminus\left\{x_1,\ldots,x_n\right\}) \arrow{r} \arrow{d}{p_{\ast}} & P_{n+m+1}(M) \arrow{r}{p_{\ast}} \arrow{d}{p_{\ast}} & P_{n}(M) \arrow{r} \arrow[equal]{d} & 1\\
%1 \arrow{r} & P_{m}(M\setminus\left\{x_1,\ldots,x_n\right\}) \arrow{r} \arrow{d} & P_{n+m}(M) \arrow{r}{p_{\ast}} \arrow{d}  & P_{n}(M) \arrow{r} & 1,\\
%&  1 & 1 && 
\end{tikzcd}
\end{equation}
where $\widehat{\iota}_{1}, \widehat{\iota}_{2}, \widehat{f}, \widehat{g}$ and $\widehat{h}$ are the homomorphisms induced by $\iota_{1}, \iota_{2}, f,g$ and $h$ respectively. For $i=1,2$, the injectivity of $\widehat{\iota}_{i}$ is a consequence of that of $\overline{\iota}_{i}$ and of standard diagram-chasing arguments. 
\end{proof}

%\begin{remark}\label{rem:relator}\comj{Is this remark necessary? It is just the injectivity of $\widehat{\iota}$ that was proved in \relem{induced}, so maybe we can delete it?} \comc{15abril - por mim, pode deletar}
% By Lemma~\ref{lem:induced}, if $w$ belongs to $B_{n_{1},\ldots,n_{k}}(M)/\Gamma_{2}(K)$ and $\widehat{\iota}(w)$ is trivial in the group $B_{n,n_{k}}(M)/\Gamma_{2}(K)$ then $w$ is trivial in $B_{n_{1},\ldots,n_{k}}(M)/\Gamma_{2}(K)$.
% \end{remark}

We now prove Theorem~\ref{th:split2}. The techniques are similar to those of the proof of Theorem~\ref{th:FNsplits}, and we will also use some of the results of Section~\ref{sec:FNsplits} in conjunction with Lemma~\ref{lem:induced}. 
% \comc{with Lemma~\ref{lem:induced}}
%we use the same techniques as Theorem~\ref{th:FNsplits}, 
%but we need to do some more computations because the relations in these groups are different.

\begin{proof}[Proof of Theorem~\ref{th:split2}] 
Let $M$ be the $2$-torus or the Klein bottle, let $p\colon\thinspace D_{n_1,\ldots,n_{k}}(M) \to D_{n_1,\ldots,n_{k-1}}(M)$ be the projection given by forgetting the last $n_k$ points, and consider the induced homomorphism $p_{\ast}\colon\thinspace B_{n_{1},\ldots,n_{k}}(M) \longrightarrow B_{n_{1},\ldots,n_{k-1}}(M)$. 
\begin{enumerate}[(a)]
\item Suppose that there exist $l_{1},\ldots,l_{k-1}\in\mathbb{N}$ such that $n_{k}=l_{1}n_{1}+\cdots+ l_{k-1}n_{k-1}$. Once more,  the existence of a non-vanishing vector field guarantees the existence of a cross-section on the level of configuration spaces. More precisely, let $x\in D_{n_{1},\ldots,n_{k-1}}(M)$, where $x=(x_{n_{1}},\ldots, x_{n_{k-1}})$ and $x_{n_{i}}\in D_{n_{i}}(M)$ for $i=1,\ldots,k-1$. With the notation of~(\ref{eq:sec2}), let $s\colon\thinspace D_{n_{1},\ldots,n_{k-1}}(M)\longrightarrow D_{n_{1},\ldots,n_{k-1},n_k}(M)$ be the map defined by:
%An element $x$ of {\bf \{ !!!D!!! This is not a good notation for our purpose. Perhaps use $D_{n_{1},\ldots,n_{k-1}}(M)$\}}\textbf{!!!C!!! Concordo. Acrescentei essa nota\c{c}\~{a}o na introdu\c{c}\~{a}o, e em outras partes do texto} $D_{n_{1},\ldots,n_{k-1}}(M)$ is such that $x=(x_{n_{1}},\ldots, x_{n_{k-1}})$, with $x_{n_{i}}\in D_{n_{i}}(M)$, so using 
%the same notation as~(\ref{eq:sec2}), we defined the section 
\begin{equation*}
\text{$s(x)=\left(x,s_{n_{1},l_{1}}(x_{n_{1}}),\ldots,s_{n_{k-1},l_{k-1}}(x_{n_{k-1}})\right)$ for all $x\in D_{n_{1},\ldots,n_{k-1}}(M)$.}
\end{equation*}
Then $s$ is a cross-section for $p$, and the induced homomorphism $s_{\ast}\colon\thinspace B_{n_{1},\ldots,n_{k-1}}(M)\longrightarrow B_{n_{1},\ldots,n_{k-1},n_k}(M)$ is a section for $p_{\ast}$.

%\textbf{!!C!! tive que reescrever algumas coisas pois tinha erros.} \textcolor{red}{DMAY21 Ok. We can remove  her comment}

\item Let $k\geq 2$, and let $n_1,\ldots,n_k\in \mathbb{N}$.
%\comj{JUL21: add `Let $k\geq 2$, and let $n_1,\ldots,n_k\in \mathbb{N}$.'?}{\color{red}!!C!!21JUL Ok} 
Suppose that $p_{\ast}\colon\thinspace B_{n_{1},\ldots,n_{k}}(M) \longrightarrow B_{n_{1},\ldots,n_{k-1}}(M)$ admits a section $s_{\ast}$. 
%\comj{JUL21: the following two sentences have been added to deal with a previous discussion.} {\color{red}!!C!!21JUL Ok} 
If $n_1=1$ then it suffices to take $l_1=n_k$ and $l_2=\cdots= l_{k-1}=0$. So suppose that $n_1\geq 2$. We first determine a generating set and some relations of the group $B_{n_{1},\ldots,n_{l}}(M)$, 
%\comj{I'm not sure that we do; to obtain the result, we use a generating set and some relations in a quotient of the group. So I suggest rewriting the proof as follows. Also, I removed a short remark regarding the surface relation in $B_{n_{1},\ldots,n_{k-1}}(M)$ since we can discuss the cases $l=k-1$ and $l=k$ at the same time.} {\textcolor{blue}{!!!C!!! ok}}
where $l\in \{k-1,k\}$. Let $l\geq 1$.
%$B_{n_{1},\ldots,n_{k-1}}(M)$ and $B_{n_{1},\ldots,n_{k}}(M)$. 
%We are considering $n_{1}\geq 2$, 
% \textcolor{red}{DMAY21 Shall we add two words for the benefit of the reader? namely}
%\textcolor{blue}{ We are considering $n_{1}\geq 2$, since if $n_1=1$ then $n_k=n_k.1$ and the result follows.}
 % {\bf !!!C!!!MAY22 concordo com o texto em azul}
% \textcolor{red}{DMAY23 See what you think John} . 
Using induction on $l$, applying the methods of~\cite[Proposition~1, p.139]{J} to the short exact sequence~(\ref{eq:p}) with $q=1$, and arguing as in the proof of Proposition~\ref{prop:B_{n,m}}, one may show that:
%Such a presentation may be obtained by induction on $k\geq 2$ and applying the methods of~\cite[Proposition~1, p.139]{J} to the short exact sequence~(\ref{eq:p}) with $q=1$. A set of generators of this group is given by:
\begin{align*}
&\left\{a_{i},b_{i},\,n_{1}+1\leq i \leq n_{1}+\cdots+n_{l} \right\}\cup \left\{C_{i,j},\, 1\leq i < j,\,n_{1}+1\leq j\leq n_{1}+\cdots+n_{l} \right\}\cup\\
&\left\{a,b\right\}\cup \left\{\text{$\sigma_{i}$, where $1\leq i \leq n_{1}+\cdots+n_{l}-1$, and $i\neq \sum^{r}_{t=1} n_{t}$ for $r=1,\ldots, l-1$}  \right\}
\end{align*}
%\textcolor{blue}{!!C!!! trocamos o indice $p$ por $r$ no somatorio acima}
is a generating set for $B_{n_{1},\ldots,n_{l}}(M)$. If $M=\mathbb{T}$ (resp.\ $M=\mathbb{K}$), set:
\begin{equation}\label{eq:defS1S2}
\text{$S_{1}=bab^{-1}a^{-1}$ (resp.\ $S_{1}=ba^{-1}b^{-1}a^{-1}$), and  $S_{2}=\sigma_{1}\cdots\sigma_{n_{1}-2} \sigma^{2}_{n_{1}-1}\sigma_{n_{1}-2}\cdots\sigma_{1}$.}
\end{equation}
 %Let $n=n_{1}$ and $m=n_{2}+\cdots n_{k}$, we have {\bf \{ !!!D!!! Perhaps we mean $B_{n_{1},\ldots,n_{k}}(M)\subset B_{n,m}(M)$. It is correct that we have a subgroup. But we need a proof of the claim how to obtain  a presentation of the subgroup from the presentation of the group \}} $B_{n_{1},\ldots,n_{k}}(M)\subset B_{n,m}$, so this subgroup will have a presentation similar to $B_{n,m}(M)$, given in Proposition~\ref{prop:B_{n,m}}, with the same generators with the exception of $\sigma_{n_{i}}$, $i=1,\ldots,k$. 
%The relations involving these elements are the same as those given in Proposition~\ref{prop:B_{n,m}}, with the exception of the surface relation, which we will rewrite as follows. Let $S_{1}=bab^{-1}a^{-1}$ (resp.\ $S_{1}=ba^{-1}b^{-1}a^{-1}$) if $M=\mathbb{T}$ (resp.\ $M=\mathbb{K}$), and let $S_{2}=\sigma_{1}\cdots\sigma^{2}_{n_{1}-1}\cdots\sigma_{1}$, 
 As for relation~(\ref{it:Bnm1}) of Proposition~\ref{prop:B_{n,m}}, the surface relation of $B_{n_{1},\ldots,n_{l}}(M)$ may be written as:
\begin{equation}\label{eq:sur0}
S^{-1}_{2}S_{1}=\prod^{n_{1}+\cdots+ n_{l}}_{i=n_{1}+1}C_{1,i}C^{-1}_{2,i}.
\end{equation}
%\begin{remark}
%The surface relation in $B_{n_{1},\ldots,n_{k-1}}(M)$ is $S^{-1}_{2}S_{1}=\prod^{n_{1}+\cdots+ n_{k-1}}_{i=n_{1}+1}C_{1,i}C^{-1}_{2,i}$.
%\end{remark}
In what follows, let $n=\sum_{i=1}^{k-1} n_{i}$. By~(\ref{eq:p}), $\ker{p_{\ast}}$ may be identified with $B_{n_{k}}(M\setminus \{ x_{1},\ldots, x_{n} \})$. Let $G=B_{n_{1},\ldots,n_{k}}(M)/\Gamma_{2}(\ker{p_{\ast}})$. Then $p_{\ast}$ induces a short exact sequence:
\begin{equation}\label{eq:sesgen}
1\to (\ker{p_{\ast}})\ab  \to G \stackrel{\widehat{p}_{\ast}}{\to} B_{n_{1},\ldots,n_{k-1}}(M)\to 1,
\end{equation}
where $\widehat{p}_{\ast} \colon\thinspace G  \to B_{n_{1},\ldots,n_{k-1}}(M)$ is the homomorphism induced by $p_{\ast}$. 
Note that~(\ref{eq:sesgen}) is the upper row of~(\ref{eq:commdiaginduced}), and $\widehat{p}_{\ast}= \widehat{g}$. Using the hypothesis that $s_{\ast}$ is a section for $p_{\ast}$, there exists a section $\widehat{\phi} \colon\thinspace B_{n_{1},\ldots,n_{k-1}}(M) \to G$ for $\widehat{p}_{\ast}$ induced by $s_{\ast}$.
%Let us consider the group \[G=\frac{B_{n_{1},\ldots,n_{k}}(M)}{\Gamma_{2}(B_{n_{k}}(M\setminus \left\{(n_{1}+\cdots+n_{k-1})\,\mbox{pts}\right\}))}\]
Making use of Proposition~\ref{prop:B_{n,m}ab} and the proof of Proposition~\ref{prop:kernelab}, we obtain the following information in $G$:
\begin{enumerate}[\textbullet]
\item using the Artin relations, we see that $\sigma_{i}=\sigma_{i+1}$ in $G$ for all $n+1\leq i \leq n +n_{k}-2$: we denote the $\Gamma_{2}(\ker{p_{\ast}})$-coset of $\sigma_{i}$ by $\sigma$.

\item for $n+1\leq i<j \leq n +n_{k}$, $C_{i,j}=1$ in $G$, and $\sigma$ is of order $2$ in $G$. 

\item for $n+1\leq i \leq n+n_{k}-1$, $a_{i}=a_{i+1}$ and $b_{i}=b_{i+1}$: we denote the $\Gamma_{2}(\ker{p_{\ast}})$-cosets of these elements by $x$ and $y$ respectively.

\item for $1\leq l \leq n$ and $n+1\leq j \leq n+n_{k}-1$, we have $C_{l,j}=C_{l,j+1}$: we denote the coset of these elements by $\rho_{l}$, where $\rho_{1}=1$ if $M=\mathbb{T}$ and $\rho_{1}=x^{2}$ if $M=\mathbb{K}$. To simplify further the notation in what follows, we set $\rho_{n+1}=1$.
\item $(\ker{p_{\ast}})\ab$ is isomorphic to $\mathbb{Z}_{2}\oplus \mathbb{Z}^{n+1}$, and the factors of this decomposition are generated by the elements $\sigma, x,y,\rho_{2},\ldots,\rho_{n}$. In particular, in $G$, these elements commute pairwise, and the notion of canonical form defined just after \req{phi1a} carries over to the situation of the short exact sequence~(\ref{eq:sesgen}). 
\end{enumerate}
%we have that $\sigma_{i}=\sigma_{i+1}$, for $n_{1}+\ldots +n_{k-1}+1\leq i \leq n_{1}+\ldots +n_{k}-2$, using the Artin relations, and we denote the class of $\sigma_{i}$ by $\sigma$. 
%We also have that the class of $C_{i,j}=1$, for $n_{1}+\ldots +n_{k-1}+1\leq i<j \leq n_{1}+\ldots +n_{k}$, so $\sigma$ has order $2$. 
%Moreover, $a_{i}=a_{i+1}$ and $b_{i}=b_{i+1}$, for all  $n_{1}+\ldots +n_{k-1}+1\leq i \leq n_{1}+\ldots +n_{k}-1$, and we denote the coset of these elements by $x$ and $y$ respectively. 
%Also, we have that $C_{l,j}=C_{l,j+1}$, for all $1\leq l \leq n_{1}+\ldots+n_{k-1}$ and $n_{1}+\ldots +n_{k-1}+1\leq j \leq n_{1}+\ldots +n_{k}-1$, and we denote the coset of these elements by $\rho_{l}$ ($\rho_{1}=1$ if $M=\mathbb{T}$ and $\rho_{1}=x^{2}$ if $M=\mathbb{K}$).
%The relations~(\ref{it:s2})-(\ref{it:s8}) given in Proposition~\ref{prop:B_{n,m}ab} are also valid in $G$, and we also have the following.
By Lemma~\ref{lem:induced}, if $w\in G$ is such that it becomes the trivial element when viewed as an element of $B_{n,n_{k}}(M)/\Gamma_{2}(\ker{p_{\ast}})$ then $w$ is itself trivial. In particular, if we take $m=n_{k}$, then the relations of Proposition~\ref{prop:B_{n,m}ab} that exist as expressions in $G$ are also relations in $G$. In particular, the following relations are valid in $G$: 
%\comj{for some of the formul\ae\ to make sense, I think that we should say somewhere that $\rho_{n+1}$ is to be taken to be equal to $1$. I have also joined together some of the items to shorten the list.} {\textcolor{blue}{!!!C!!! ok}}
\begin{enumerate}[(i)]
\item\label{it:sur1} $S_{2}^{-1}S_{1}=\begin{cases}
\left( \prod^{n}_{i=n_{1}+1}C_{1,i}C^{-1}_{2,i}\right) \rho^{-n_{k}}_{2} & \text{if $M=\mathbb{T}$}\\
\left(\prod^{n}_{i=n_{1}+1}C_{1,i}C^{-1}_{2,i}\right) \rho^{-n_{k}}_{2}x^{2n_{k}} & \text{if $M=\mathbb{K}$.}
\end{cases}$
\item\label{it:sur2} For $n_{1}+1\leq i \leq n$, we have:
\begin{enumerate}[(a)]
%\left\{\begin{array}{lr}\displaystyle\left( \prod^{n_{1}+\cdots+ n_{k-1}}_{i=n_{1}+1}C_{1,i}C^{-1}_{2,i}\right) \rho^{-n_{k}}_{2},& M=\mathbb{T};\\
%\displaystyle\left(\prod^{n_{1}+\cdots+ n_{k-1}}_{i=n_{1}+1}C_{1,i}C^{-1}_{2,i}\right) \rho^{-n_{k}}_{2}x^{2n_{k}},& M=\mathbb{K}.\end{array}\right.$
    
\item\label{it:sur2a} $a^{-1}_{i}ya_{i}=y\rho^{-1}_{i}\rho_{i+1}$, $b^{-1}_{i}yb_{i}=\begin{cases}
y & \text{if $M=\mathbb{T}$}\\
y\rho_{i}\rho^{-1}_{i+1} & \text{if $M=\mathbb{K}$}
\end{cases}$ and $b^{-1}_{i}xb_{i}=\begin{cases}
x\rho_{i}\rho^{-1}_{i+1} & \text{if $M=\mathbb{T}$}\\
x\rho^{-1}_{i}\rho_{i+1}& \text{if $M=\mathbb{K}$.}
\end{cases}$
    
\item\label{it:sur2b} for $1\leq q \leq n$, $a^{-1}_{i}\rho_{q}a_{i}=\rho_{q}$, and:
\begin{equation*}
b^{-1}_{i}\rho_{q}b_{i}= \begin{cases}
\rho_{q} & \text{if $M=\mathbb{T}$, or if $M=\mathbb{K}$ and $i<q$}\\
\rho^{-2}_{i}\rho^{2}_{i+1}\rho_{q} & \text{if $M=\mathbb{K}$ and $i\geq q$.}
\end{cases}
\end{equation*}  
\end{enumerate}
\item\label{it:sur3} $a$ commutes with $\sigma,x$ and $\rho_{q}$, where $2\leq q\leq n$, and $\sigma$ commutes with $b_{j}$ for $n_{1}+1\leq j\leq n$. If $M=\mathbb{T}$, then $a^{-1}ya=y\rho_{2}$, $aya^{-1}=y\rho_{2}^{-1}$, $b^{-1}xb=x \rho_{2}^{-1}$, $bxb^{-1}=x \rho_{2}$, and $b$ commutes with $y$. If $M=\mathbb{K}$ then $a^{-1}ya=yx^{-2}\rho_{2}$, $aya^{-1}=yx^{2}\rho_{2}^{-1}$, $b^{-1}xb=x^{-1} \rho_{2}$, $bxb^{-1}=x^{-1} \rho_{2}^{-1}$, $b^{-1}yb=y x^{2}\rho_{2}^{-1}$ and $byb^{-1}=y x^{-2}\rho_{2}$. 
% then $b_{j}xb_{j}^{-1}= x \rho_{i}^{-1}\rho_{i+1}$, $b_{j}yb_{j}^{-1}= y \rho_{i}\rho_{i+1}^{-1}$, {\textcolor{blue}{!!!C!!! o indice aqui esta deveria ser $i$ (ou $j$)? ou estava querendo descrever a relacao com o $b$?}} $a^{-1}ya=yx^{-2}\rho_{2}$, $aya^{-1}=yx^{2}\rho_{2}^{-1}$ and $b_{i}\rho_{2} b_{i}^{-1}= \rho_{2} \rho_{i}^{-2} \rho_{i+1}^{2}$.
\end{enumerate}
Relation~(\ref{it:sur1}) follows from~(\ref{eq:sur0}), and relations~(\ref{it:sur2})(\ref{it:sur2a}) (resp.\ relations~(\ref{it:sur2})(\ref{it:sur2b})) follow from relations~(\ref{it:puras2}),~(\ref{it:puras6}) and~(\ref{it:puras7}) (resp.\ relations~(\ref{it:puras3}) and~(\ref{it:puras8})) of Theorem~\ref{th:puras}, using the above information about $G$, and notably the fact that the $\Gamma_{2}(\ker{p_{\ast}})$-cosets of $a_{j},b_{j}$ and $C_{i,j}$ are $x,y$ and $\rho_{i}$ respectively. Relations~(\ref{it:sur3}) are consequences of Proposition~\ref{prop:B_{n,m}ab} and Remarks~\ref{rem:misc}(\ref{it:miscc}). One may check that if $M=\mathbb{K}$ and $i\geq q$ then $b_{i}\rho_{q}b_{i}^{-1}=
\rho^{-2}_{i}\rho^{2}_{i+1}\rho_{q}$.

To complete the proof of part~(\ref{it:split2b}), we follow the strategy of the proof of Theorem~\ref{th:FNsplits} by studying the images of some of the relations of $B_{n_{1}, \ldots, n_{k-1}}(M)$ under the homomorphism $\widehat{\phi}$. We may write the images of the elements $a,b$ and $\sigma_{i}$, where $1\leq i \leq n-1$ and $i\neq \sum^{r}_{t=1} n_{t}$ for $r=1,\ldots, k-2$, 
%{\textcolor{blue}{trocamos $p$ por $r$}} 
in the form of equation~(\ref{eq:exp}), where $n$ is taken to be equal to $\sum_{i=1}^{k-1} n_i$. Similarly, for $n_{1}+1\leq j \leq n$, we set:
\begin{equation*}
\widehat{\phi}(b_{j})=b_{j}\cdot x^{t_{j}}y^{p_{j}}\sigma^{s_{j,0}}\rho^{s_{j,2}}_{2}\cdots \rho^{s_{j,n}}_{n},
\end{equation*}
where $t_{j},p_{j},s_{j,2},\ldots, s_{j,n}\in \mathbb{Z}$, and $s_{j,0}$ is defined modulo $2$. 
With appropriate restrictions on $i$ and $j$, the conclusions of Lemmas~\ref{lem:exp.M1} and~\ref{lem:exp.M} are also valid here. More precisely, set $t=n_1$, $s=\sum_{i=2}^{k-1} n_i$, so $t+s=n$, and $m=n_k$ as in the statement of Lemma~\ref{lem:induced}, and let $\Gamma'= \left\{ \,\sum^{r}_{t=1} n_{t} \left\lvert \; r=1,\ldots, k-2\right.\right\}$. It follows from that lemma that if $w$ is an element of $G$ for which either $\widehat{\iota}_{1}(w)$ is a relator in  $B_{n_{1},s,n_{k}}(M)/\Gamma_{2}(K)$ or $\widehat{\iota}_{2}\circ \widehat{\iota}_{1}(w)$ is a relator in  $B_{n,n_{k}}(M)/\Gamma_{2}(K)$ then $w$ is a relator in $G$. In particular, in the current setting:
\begin{enumerate}[(a)]
\item for all $t\geq 4$ (resp.\ $t\geq 3$), the conclusion of Lemma~\ref{lem:exp.M1}(\ref{it:exp.M1a}) (resp.\ Lemma~\ref{lem:exp.M1}(\ref{it:exp.M1b})) holds for all $i,j\in \brak{1,\ldots,n-1}\setminus \Gamma'$ (resp.\ for all $1\leq i \leq t-2$).

\item the conclusion of Lemma~\ref{lem:exp.M} holds.

\item the conclusion of Lemma~\ref{lem:phiartin} remains valid when $n$ is replaced by $t=n_{1}$ and $m$ is replaced by $\sum_{i=2}^{k} n_i$. 
% \comj{and $m$ is replaced by $\sum_{i=2}^{k} n_i$?}\comc{sim, alterei o texto.}
\end{enumerate}
% the conclusions of Lemmas~\ref{lem:exp.M1} and~\ref{lem:exp.M} are also valid here, because the relators that were used in their proofs may be viewed as elements in $G$ \comj{My impression is that perhaps we need to be a bit more precise here about the values that we take in Lemma~\ref{lem:exp.M1}. I think that we have $t=n_1$, $s=\sum_{i=2}^{k-1} n_i$, so $t+s=n$, and $m=n_k$. If that is so, then in  Lemma~\ref{lem:exp.M1}(\ref{it:exp.M1a}), it looks like $i,j\in \brak{1,\ldots,n-1}\setminus \Gamma'$, where $\Gamma'=\left\{ \,\sum^{r}_{t=1} n_{t} \left\lvert \; r=1,\ldots, k-2\right.\right\}$ (a similar set $\Gamma$ is defined on the following page). Lemma~\ref{lem:exp.M1} is proved within the framework of \resec{contas}, and so perhaps it is worth extending Lemma~\ref{lem:induced} to include the homomorphism $ B_{n_{1},\ldots,n_{k}}(M)/\Gamma_{2}(K) \to B_{n_1,s,n_{k}}(M)/\Gamma_{2}(K)$? (note that the homomorphism $\widehat{\iota}$ factors through this homomorphism). This is to do with the fact that we are working with the quotient of three different groups ($B_{n_{1},\ldots,n_{k}}(M)$, $B_{n_{1},s,n_{k}}(M)$ and $B_{n,n_{k}}(M)$) by $\Gamma_2(K)$.} \comc{6maio - entendi, concordo em estender o lema 4.1 para incluir esse caso.}, and it follows from Lemma~\ref{lem:induced} that these elements are also relators in $G$. 
Taking $l=k-1$, let us study the surface relation~\reqref{sur0} of $B_{n_{1}, \ldots, n_{k-1}}(M)$ 
% \comj{? Isn't the surface relation of $B_{n_{1}, \ldots, n_{k-1}}(M)$ \reqref{sur0} with $l=k-1$?}\comc{6maio - sim, alterei} 
using~(\ref{eq:defS1S2}). 
% By~(\ref{eq:P3}),~(\ref{eq:phirprime}) \comj{Maybe~(\ref{eq:calcalphaa}) and~(\ref{eq:calcdeltaa})?}\comc{6maio - alem dessas rela\c{c}\~{o}es, acho que deve-se mencionar~(\ref{eq:relny}) e (\ref{eq:relnw})} 
By~(\ref{eq:relny}) and~(\ref{eq:relnw}), the canonical form of $\widehat{\phi}(S_{2}^{-1}S_{1})$ is given by:
\begin{equation}\label{eq:canformS2S1}
\widehat{\phi}(S^{-1}_{2}S_{1})=S^{-1}_{2}S_{1} \ldotp w'\rho^{\alpha+\delta}_{2},
\end{equation}
where $w'$ is a word in $x,y,\sigma,\rho_3,\ldots, \rho_n$. Using~(\ref{eq:defalpha}), (\ref{eq:ar4}) and Lemmas~\ref{lem:phiartin}(\ref{it:phiartin0}) and~\ref{lem:exp.M}(\ref{it:M4}), it follows that $\alpha=(n_{1}-2)k_{4}$. By~(\ref{eq:calcdeltaa}) and Lemma~\ref{lem:exp.M}(\ref{it:M3}), we have $\delta=2k_{4}$, 
% \comj{Isn't this exactly~(\ref{eq:calcdeltaa})? Do we need to mention Lemma~\ref{lem:exp.M}(\ref{it:M3}) here?} \comc{Sim, acho que pode retirar. Achei que precisaria relembrar de onde saiu a prova por causa da mudan\c{c}a dos indices.}, 
and it follows that the exponent of $\rho_{2}$ in the canonical form of $\widehat{\phi}(S^{-1}_{2}S_{1})$ given in~\reqref{canformS2S1} is equal to $n_{1} k_{4}$. % By~(\ref{eq:relny}), (\ref{eq:relnw}), (\ref{eq:calcdeltaa}), (\ref{eq:calcalphab}) and Lemma~\ref{lem:exp.M}(\ref{it:M3}), we see that the canonical form of $\widehat{\phi}(S_{2}^{-1}S_{1})$ is given by: \comj{To obtain $n_{1}$ here in the framework of Section~\ref{sec:FNsplits} (rather than $n$), I have the impression that we use something like~\reqref{commdiaginduced}, except that $\widehat{\iota}_{2}$ is a homomorphism from $B_{n_{1},s,n_{k}}(M)/\Gamma_{2}(K)$ to $B_{n_{1},s+n_{k}}(M)/\Gamma_{2}(K)$. What do you think?} {\comc{7junho -  n\~{a}o sei se precisa, o expoente do $\rho_{2}$ \'{e} $\alpha+\delta$ usando (2.27) e (2.31), e o $\alpha$ aparece em termos de $t=n_{1}$. Acho que s\'{o} precisamos mudar as numera\c{c}\~{o}es, pois o (3.23) concluiu para o $n$, pois naquele caso tinhamos que $n=t$. Pode-se falar que a conclus\~{a}o do teorema 3.3 vale para $t$ ao inves de $n$, neste caso atual, e assim usar (2.28), lemma 3.3(a) e (3.11) para concluir que $\alpha=(t-2)k_{4}$, que \'{e}o analogo ao (3.33). Adicinei sugest\~{a}o de texto em azul acima, veja se assim resolve.}}
% \begin{equation}\label{eq:canformS2S1}
% \widehat{\phi}(S^{-1}_{2}S_{1})=S^{-1}_{2}S_{1} \ldotp w'\rho^{n_{1}k_{4}}_{2},
% \end{equation}
% where $w'$ is a word in $x,y,\sigma,\rho_3,\ldots, \rho_n$. In particular the exponent of $\rho_{2}$ in this canonical form is equal to $n_{1} k_{4}$.
 % For the surface relations, the exponent of $\rho_{2}$ in the canonical form of $\phi(S^{-1}_{2}S_{1})$ is $n_{1}\cdot k_{4}$ \comj{Where does this come from? Maybe from~(\ref{eq:P3}),~(\ref{eq:phirprime}) and Lemma~\ref{lem:exp.M}(\ref{it:M3})? Also, probably it should be made clearer what $S^{-1}_{2}S_{1}$ in $\phi(S^{-1}_{2}S_{1})$ means. Up until this point, $S^{-1}_{2}S_{1}$ was an element of $G$, and here, it looks like it is being used as an element of $B_{n_{1},\ldots,n_{k-1}}(M)$.}. {\textcolor{blue}{!!!C!!! Sim, usando as relacoes (3.42), (3.43) e o lema 3.2(3). O elemento $S^{-1}_{2}S_{1}$ foi definido no inicio da demonstracao, esta acima da equacao~(\ref{it:sur0}), deveria relembrar a definicao para o leitor? }} 
 
To compute the exponent of $\rho_{2}$ in the canonical form of $\widehat{\phi}(\prod^{n}_{j=n_{1}+1}C_{1,j}C^{-1}_{2,j})$, we first study $C_{1,j}C^{-1}_{2,j}$ in $P_{n}(M)$ for $j=n_{1}+1,\ldots, n$. For such values of $j$, taking $i=1$ in relation~(\ref{it:puras2}) of Theorem~\ref{th:puras} and recalling that $a=a_{1}$, we see that $a^{-1} b_{j}a= b_{j} a_{j}C_{1,j}^{-1}C_{2,j}a_{j}^{-1}$, and thus:
\begin{equation}\label{eq:C2invC1}
C_{2,j}^{-1}C_{1,j}= a_{j}^{-1} a^{-1} b_{j}^{-1} a b_{j} a_{j},
\end{equation}
using the fact that $a$ and $a_{j}$ commute by relation~(\ref{it:puras1}) of Theorem~\ref{th:puras}. On the other hand, taking $i=j=1$ and $k=j$ in relation~(\ref{it:puras3}) of Theorem~\ref{th:puras}, we see that $a^{-1} C_{1,j}a= a_{j} C_{2,j}^{-1} C_{1,j} a_{j}^{-1} C_{2,j}$, and using the fact that $a$ commutes with $C_{2,j}$ by the same relation, we obtain $C_{1,j}C_{2,j}^{-1}= a a_{j} C_{2,j}^{-1} C_{1,j} a_{j}^{-1}a^{-1}$. Substituting~(\ref{eq:C2invC1}) in this equation, in $P_{n}(M)$ it follows that:
\begin{equation}\label{eq:C1C2inv}
C_{1,j}C_{2,j}^{-1}= b_{j}^{-1} a b_{j} a^{-1}.
\end{equation}
%\begin{equation}
%C_{1,j}C^{-1}_{2,j}=ab^{-1}_{j}a^{-1}b_{j}, {\textcolor{blue}{C_{1,j}C^{-1}_{2,j}=b^{-1}_{j}ab_{j}a^{-1}}}
%\end{equation}
%\comj{How does one obtain this equality from the given relations?} for all $j$ \comj{for all $n_{1}+1\leq j\leq n$?} using relations~(\ref{it:puras2})--(\ref{it:puras3}) of Theorem~\ref{th:puras}. 
%\textcolor{blue}{!!!C!!! a primeira igualdade \'{e} a relacao (2) com $i=1$, e depois usamos a (3) com $i=j=1$ e $k=j$, e tambem comutamos $C_{i+1,j}$ com $a_{i}$ lembrando que $a=a_{1}$, deveria deixar essa conta no texto? vou usar essa igualdade novamente no ultimo paragrafo da demonstracao
%\begin{align}
%    a^{-1}_{i}b_{j}a_{i}=b_{j}\big( a_{j}C^{-1}_{i,j}C_{i+1,j}a^{-1}_{j}\big)=b_{j}\big( a^{-1}_{i}C_{i+1,j}C^{-1}_{i,j}a_{i}\big)
%\end{align}
%A equacao deveria ser $C_{1,j}C^{-1}_{2,j}=b^{-1}_{j}ab_{j}a^{-1}$, mas nao muda muito no restante}
%\[\phi(ab^{-1}_{j}a^{-1}b_{j})=ax^{k_{1}-t_{j}}y^{k_{2}-p_{j}}b^{-1}_{j}x^{-k_{1}}y^{-k_{2}}a^{-1}b_{j}x^{t_{j}}y^{p_{j}}\omega\]
Let us compute $\widehat{\phi}(b_{j}^{-1} a b_{j} a^{-1})$. Note that $j\geq 3$ since $n_{1}\geq 2$.
\begin{enumerate}[\textbullet]
\item If $M=\mathbb{T}$, using relations~(\ref{it:sur2})(\ref{it:sur2a}),~(\ref{it:sur2})(\ref{it:sur2b})  and~(\ref{it:sur3}) above, we have:
%isso vale desde que n_{1}\geq 2
\begin{align}
\widehat{\phi}(b^{-1}_{j}ab_{j}a^{-1})&= y^{-p_{j}}x^{-t_{j}}b^{-1}_{j}ax^{k_{1}}y^{k_{2}}b_{j}x^{t_{j}-k_{1}}y^{p_{j}-k_{2}} a^{-1}\notag\\
&=b^{-1}_{j}y^{-p_{j}}x^{-t_{j}}w_{j,j+1}ax^{k_{1}}y^{k_{2}}b_{j}x^{t_{j}-k_{1}}y^{p_{j}-k_{2}} a^{-1}\notag\\
&=b^{-1}_{j}ax^{k_{1}-t_{j}}y^{k_{2}-p_{j}}\rho^{-p_{j}}_{2}w_{j,j+1}b_{j}x^{t_{j}-k_{1}}y^{p_{j}-k_{2}} a^{-1} =b^{-1}_{j}ab_{j}\rho^{-p_{j}}_{2}w'_{j,j+1}a^{-1}\notag\\
&=b^{-1}_{j}ab_{j}a^{-1}\rho^{-p_{j}}_{2}w'_{j,j+1}, \label{eq:phiT1}
\end{align}
%\begin{align}
%\phi(ab^{-1}_{j}a^{-1}b_{j})&= ax^{k_{1}-t_{j}}y^{k_{2}-p_{j}}b^{-1}_{j}y^{-k_{2}}x^{-k_{1}} a^{-1}b_{j} x^{t_{j}} y^{p_{j}}\notag\\
%%\end{align}
%%\intertext{\comj{To obtain the above line, I think that we use that $a$ and $b_{j}$ commute with all elements except possibly $x$ and $y$. So don't all the remaining terms in $\sigma,\rho_{2},\ldots,\rho_{n}$ cancel each other out? In which case, isn't $w$ above equal to $1$?}} 
%%\begin{align}
% &= ab^{-1}_{j}x^{-t_{j}}y^{-p_{j}}w_{j,j+1}a^{-1}b_{j}x^{t_{j}}y^{p_{j}}\notag\\
%&=ab^{-1}_{j}a^{-1}x^{-t_{j}}y^{-p_{j}}\rho^{p_{j}}_{2}w_{j,j+1}b_{j}x^{t_{j}}y^{p_{j}} = ab^{-1}_{j}a^{-1}b_{j}\rho^{p_{j}}_{2}w'_{j,j+1}, \label{eq:phiT1}
%\end{align}
%{\textcolor{blue}{!!!C!!! Sim, $w=1$, retirei ele na equacao, e no caso da garrafa de klein \'{e} uma palavra com $\rho_{j} , \rho_{j+1}$.}}
where %$w$ is a word in the $\rho_{k}$ for $k\geq 3$, and 
$w_{j,j+1}$ and $w'_{j,j+1}$ are words in $\rho_{j}$ and $\rho_{j+1}$.

\item If $M=\mathbb{K}$, since $j\geq 3$, for all $k\leq j$, we have $b_{j}^{-1} \rho_{k} b_{j}= v\rho_{k}$, where $v$ is a word in $\rho_{j}$ and $\rho_{j+1}$. Then by relations~(\ref{it:sur2})(\ref{it:sur2a}),~(\ref{it:sur2})(\ref{it:sur2b})  and~(\ref{it:sur3}) above, we obtain:
\begin{align}
\widehat{\phi}(b^{-1}_{j}ab_{j}a^{-1})&= y^{-p_{j}}x^{-t_{j}}b^{-1}_{j}ax^{k_{1}}y^{k_{2}}b_{j}x^{t_{j}-k_{1}}y^{p_{j}-k_{2}} a^{-1}w_{j,j+1}''\notag\\
&=b^{-1}_{j}y^{-p_{j}}x^{-t_{j}}w_{j,j+1}ax^{k_{1}}y^{k_{2}}b_{j}x^{t_{j}-k_{1}}y^{p_{j}-k_{2}} a^{-1} w_{j,j+1}''\notag\\
&=b^{-1}_{j}ax^{k_{1}-t_{j}}y^{k_{2}-p_{j}}\rho^{-p_{j}}_{2}w_{j,j+1}b_{j}x^{t_{j}-k_{1}}y^{p_{j}-k_{2}} a^{-1} w_{j,j+1}''\notag\\
&=b^{-1}_{j}ab_{j}\rho^{-p_{j}}_{2}w_{j,j+1}a^{-1}w_{j,j+1}'' = b^{-1}_{j} ab_{j}a^{-1}\rho^{-p_{j}}_{2}w'_{j,j+1}, \label{eq:phiK1}
\end{align}
%\comj{The two equations for $M=\mathbb{T}$ and $M=\mathbb{K}$ have the same label \texttt{eq:phiT1}.}
%and
%\begin{align}
%\phi(ab^{-1}_{j}a^{-1}b_{j})&= ax^{k_{1}-t_{j}}y^{k_{2}-p_{j}}b^{-1}_{j}y^{-k_{2}} x^{-k_{1}} a^{-1}b_{j} x^{t_{j}} y^{p_{j}}w %\notag\\
%%\intertext{\comj{Here I guess that $w\neq 1$, and in fact it is a word in $\rho_{j}, \rho_{j+1}$?}}
% = ab^{-1}_{j}x^{-t_{j}}y^{-p_{j}}w_{j,j+1}a^{-1}b_{j}x^{t_{j}}y^{p_{j}}w\notag\\
%&=ab^{-1}_{j}a^{-1}x^{-2p_{j}-t_{j}}y^{-p_{j}}\rho^{p_{j}}_{2}w_{j,j+1}b_{j}x^{t_{j}}y^{p_{j}}w\notag\\
%& = ab^{-1}_{j}a^{-1}b_{j}x^{-2p_{j}}\rho^{p_{j}}_{2}w'_{j,j+1}w, \label{eq:phiK1}
%\end{align}
where $w_{j,j+1}, w'_{j,j+1}$ and $w''_{j,j+1}$ are words in $\rho_{j}$ and $\rho_{j+1}$.
\end{enumerate}
Using~(\ref{eq:sur0}) with $l=k-1$ for both $M=\mathbb{T}$ and $M=\mathbb{K}$,  and~(\ref{eq:C1C2inv})--(\ref{eq:phiK1}), we see that:
\begin{align}
\widehat{\phi}(S_{2}^{-1}S_{1})&= \widehat{\phi}\left(\prod^{n}_{j=n_{1}+1} C_{1,j} C_{2,j}^{-1}\right)= \widehat{\phi}\left(\prod^{n}_{j=n_{1}+1}b^{-1}_{j}ab_{j}a^{-1}\right)= \prod^{n}_{j=n_{1}+1} b^{-1}_{j} ab_{j}a^{-1}\rho^{-p_{j}}_{2}w'_{j,j+1}\notag\\
&= \left(\prod^{n}_{j=n_{1}+1} C_{1,j} C_{2,j}^{-1} w'_{j,j+1}\right) \rho_{2}^{-\sum^{n}_{j=n_{1}+1} p_{j}},\label{eq:phisurf}
\end{align}
where to obtain the last equality, we have used also relation~(\ref{it:sur2})(\ref{it:sur2b}). 
% \comj{Some of the following rewritten slightly:} \comc{6maio - ok} 
When we put~(\ref{eq:phisurf}) in canonical form, relations~(\ref{it:sur2})(\ref{it:sur2a}),~(\ref{it:sur2})(\ref{it:sur2b})  and~(\ref{it:sur3}) imply that no new terms in $\rho_{2}$ are introduced during this process. It follows from relation~(\ref{it:sur1}),~\reqref{canformS2S1} and~(\ref{eq:phisurf}) that $n_1k_4-n_k= -\sum^{n}_{j=n_{1}+1} p_{j}$, hence: 
%\comj{In the previous version, it was written $n_{1}k_{4}=-\sum^{n}_{j=n_{1}+1} p_{j}$, but I think that it is missing the term $n_k$ that comes from relation~(\ref{it:sur1}).}\comc{6maio - ok}
\begin{equation}\label{eq:n1k4}
n_k= n_1k_4+\sum^{n}_{j=n_{1}+1} p_{j}.
\end{equation}
% so the coefficient of $\rho_{2}$ in this canonical form is equal to $-\sum^{n}_{j=n_{1}+1} p_{j}$
% As we saw above, the exponent of $\rho_{2}$ in the canonical form of $\widehat{\phi}(S_{2}^{-1}S_{1})$ is equal to $n_{1}k_{4}$. When we put the right-hand term of~(\ref{eq:phisurf}) in canonical form, the form of  relations~(\ref{it:sur2})(\ref{it:sur2a}),~(\ref{it:sur2})(\ref{it:sur2b})  and~(\ref{it:sur3}) implies that no new terms in $\rho_{2}$ are introduced during this process, so the coefficient of $\rho_{2}$ in this canonical form is equal to $-\sum^{n}_{j=n_{1}+1} p_{j}$, hence:
% \begin{equation}\label{eq:n1k4}
% n_{1}k_{4}=-\sum^{n}_{j=n_{1}+1} p_{j}.
% \end{equation}
% in both cases, we conclude that the exponent of $\rho_{2}$ in the canonical form of $\phi(\prod^{n}_{i=n_{1}+1}C_{1,i}C^{-1}_{2,i})$ is equal to $\sum^{n}_{j=n_{1}+1}-p_{j}$. 

%\comj{Some notation added, and some parts rewritten slightly:} 
To complete the proof, it remains to compute the terms $p_j$ in~\reqref{n1k4} for $j=n_{1}+1,\ldots, n$. Let $\Gamma= \left\{ \,\sum^{r}_{t=1} n_{t} \left\lvert \; r=1,\ldots, k-1\right.\right\}$. We claim that $l_{i,2}=0$ for all $1\leq i \leq n-1$ and $i\notin \Gamma$. To see this, first note that relations~(\ref{it:s2})--(\ref{it:s8}) of Proposition~\ref{prop:B_{n,m}ab} and relations~(\ref{it:full1})--(\ref{it:full7}) of Theorem~\ref{th:total} hold in our setting, with the exception of those relations involving $\sigma_{i}$ or $\sigma_{j}$, where $1\leq i,j\leq n-1$ and $\{i,j\} \cap \Gamma \neq \varnothing$. If $i=1$ then by considering the relation $b^{-1}\sigma_{1}a= \sigma_{1} a\sigma_{1} b^{-1}\sigma_{1}$ and arguing in a manner similar to that of the proof of Lemma~\ref{lem:exp.M}(\ref{it:M2}), we see that $l_{1,2}=0$.
% In a manner similar to that of the proof of Lemma~\ref{lem:exp.M}(\ref{it:M2}), we see that $l_{i,2}=0$ for all $1\leq i \leq n-1$, where $i\neq \sum^{r}_{t=1} n_{t}$ for all $r=1,\ldots, k-1$.
% {\textcolor{blue}{trocamos $p$ por $r$}}. 
Now suppose that $2\leq i \leq n-1$ and that $i\notin \Gamma$. Using relations~(\ref{it:puras2})--(\ref{it:puras3}) of Theorem~\ref{th:puras}, for $i<j\leq n$
% \comj{for $1\leq j\leq n$? Since we take $j=i+1$, $j=n$ needs to be possible in what follows.} \comc{6maio - for $i<j\leq n$} 
we have $a^{-1}_{i}b_{j}a_{i}= b_{j}a_{j}C^{-1}_{i,j} C_{i+1,j} a^{-1}_{j}= b_{j}a^{-1}_{i}C_{i+1,j}C^{-1}_{i,j}a_{i}$, and thus $C_{i,j} C^{-1}_{i+1,j}= b^{-1}_{j}a_{i}b_{j}a^{-1}_{i}$. Taking $j=i+1$ in this relation and using the fact that $C_{i,i+1}=\sigma_i^2$, it follows that $\sigma_i^2= b^{-1}_{i+1} a_{i} b_{i+1} a^{-1}_{i}$, and using the equality: % \comj{Some of the following rewritten slightly:}\comc{6maio - ok}
\begin{equation}\label{eq:bi+1bi}
b_{i+1}= \sigma_i^{-1}b_i \sigma_i^{-1}    
\end{equation}
obtained via relation~(\ref{it:k4}) of Proposition~\ref{prop:kernel}, we obtain:
%\comj{Is~(\ref{it:k3}) used?}{\textcolor{blue}
%{CD 31july - retiramos a relacao 4, pois n\~ao foi usada}}, 
\begin{align}\label{eq:b.sigma.a}
b^{-1}_{i}\sigma_{i}a_{i}=\sigma_{i}a_{i}\sigma_{i}b^{-1}_{i}\sigma_{i}. %\comj{Where does $a$ come from?}
\end{align}
%{\textcolor{blue}{!!!C!!! foi erro de digita\c{c}\~{a}o, retirei o elemento}}
Let $q_{i}$ denote the exponent of $y$ in the canonical form of $\widehat{\phi}(a_{i})$. Since the exponent of $y$ is the same on both sides of each of the relations~(\ref{it:sur2})(\ref{it:sur2a}),~(\ref{it:sur2})(\ref{it:sur2b}) and~(\ref{it:sur3}), it follows that the exponent of $y$ in the canonical form of $\widehat{\phi}(b^{-1}_{i}\sigma_{i}a_{i})$ (resp.\ of $\widehat{\phi}(\sigma_{i}a_{i}\sigma_{i}b^{-1}_{i}\sigma_{i})$) is equal to $l_{i,2}-p_{i}+q_{i}$ (resp.\ to $3l_{i,2}-p_{i}+q_{i}$). Using the fact that~(\ref{eq:b.sigma.a}) also holds when viewed as a relation in $G$, we deduce that $l_{i,2}=0$, which proves the claim. In a similar manner, if $\sum^{r}_{t=1}n_{t}< i < \sum^{r+1}_{t=1}n_{t}$, where $r=1,\ldots, k-2$, then computing the exponent of $y$ in the image by $\widehat{\phi}$ of~\reqref{bi+1bi}, 
% the canonical form of $\widehat{\phi}(\sigma^{-1}_{i}b_{i})$ (resp.\ of $\widehat{\phi}(b_{j+1}\sigma_{j})$) from 
% $\sigma^{-1}_{i}b_{i}=b_{i+1}\sigma_{i}$ obtained from~\reqref{bi+1bi}, 
and using the fact that this equality also holds in $G$, 
% In a similar manner, by relation~(\ref{it:k4}) of Proposition~\ref{prop:kernel}, we have $\sigma^{-1}_{i}b_{i}=b_{i+1}\sigma_{i}$, and computing the exponent of $y$ in the canonical form of $\widehat{\phi}(\sigma^{-1}_{i}b_{i})$ (resp.\ of $\widehat{\phi}(b_{j+1}\sigma_{j})$) and using the fact that the relation $\sigma^{-1}_{i}b_{i}=b_{i+1}\sigma_{i}$ also holds in $G$, 
we obtain $p_{i+1}= p_{i}+2l_{i,2}$, and thus $p_{i}=p_{i+1}$ since $l_{i,2}=0$. So there exists $\alpha_{n_{r+1}} \in \mathbb{Z}$ such that $p_{i}=\alpha_{n_{r+1}}$ for all $\sum^{r}_{t=1}n_{t}< i \leq \sum^{r+1}_{t=1}n_{t}$ and $r=1,\ldots, k-2$.
% \comj{Since we saw that $p_{i}=p_{i+1}$, in this inequality, I believe that $i < \sum^{r+1}_{t=1}n_{t}$ above becomes $i \leq \sum^{r+1}_{t=1}n_{t}$ here. Do you agree?}. {\textcolor{blue}{!!C!! Sim}}
% 
% Using relations~(\ref{it:sur2})(\ref{it:sur2a}),~(\ref{it:sur2})(\ref{it:sur2b})  and~(\ref{it:sur3}), notice that the exponent of $y$ does not change when we put any expression in the canonical form, so the exponent of $y$ in $\widehat{\phi}(b^{-1}_{i}\sigma_{i}a_{i})$ is $l_{i,2}-p_{i}+q_{i}$ and in $\widehat{\phi}(\sigma_{i}a_{i}a\sigma_{i}b^{-1}_{i}\sigma_{i})$ is $3l_{i,2}-p_{i}+q_{i}$ (where $q_{i}$ is the exponent of $y$ in $\widehat{\phi}(a_{i})$), therefore, $l_{i,2}=0$. 
%
%Now, using relation~(\ref{it:k4}) of Proposition~\ref{prop:kernel}, we have $\sigma^{-1}_{i}b_{i}=b_{i+1}\sigma_{i}$ for $\sum^{r}_{t=1}n_{t}< i < \sum^{r+1}_{t=1}n_{t}$, with $r=1,\ldots, k-2$. Computing the exponent of $y$ in the canonical form of $\phi(\sigma^{-1}_{i}b_{i})$ (resp.\  $\phi(b_{j+1}\sigma_{j})$) we obtain $p_{i}-l_{i,2}$ (resp.\ $p_{i+1}+l_{i,2}$), {\textcolor{blue}{then $p_{i}=p_{i+1}$, for $\sum^{r}_{t=1}n_{t}< i < \sum^{r+1}_{t=1}n_{t}$, with $r=1,\ldots, k-2$, 
% and we denote such common value of $p_{i}$ for each block as $\alpha_{n_{r+1}}$. \comj{The common value of $p_{j}$ for each block?}  \comj{We should think a bit about how to do this.} {\textcolor{blue}{ !!!C!!! veja a nova versao. E trocamos o indice $p$ por $r$ em varios lugares}} 
We deduce from~\reqref{n1k4} that $n_k=n_1k_4+n_{2}\alpha_{n_{2}}+\cdots +n_{k-1}\alpha_{n_{k-1}}$, and this completes the proof of the theorem. \qedhere
\end{enumerate}
\end{proof}

%{\bf \{!!!D!!! There is something weird here. If we take k=2, the above proof %tells that $n_2=-k_4 n_2$. So we have that $-k_4\geq 0$?\}}. 
%The difficulty in the converse of this theorem is that our algebraic method do not guarantee if the exponent is positive or negative. The special case that $k=3$, $q=1$ and $n_{3}=1$, we have the equivalence of the result for the torus, which is detailed in the next section. This particular case, make us conjecture that the converse in Theorem~\ref{th:split2} will be true. 
\subsection{The case $k=3$, $q=1$ and $n_{3}=1$}\label{sec:Bt,s,1}

In Theorem~\ref{th:split2}(\ref{it:split2a}), we obtained a geometric section on the level of configuration spaces by adding new distinct points in accordance with the relation between $n_{k}$ and $n_{1},\ldots,n_{k-1}$ using the non-vanishing vector field on $\mathbb{T}$ and $\mathbb{K}$. However, the algebraic techniques used to prove the relation of Theorem~\ref{th:split2}(\ref{it:split2b}) leave open the possibility that some of the coefficients of $n_{1},\ldots,n_{k-1}$ in that relation be negative, and it is not clear how to interpret this geometrically. In this section, we study the case where $M=\mathbb{T}$ or $M=\mathbb K$, $k=3$, $q=1$ and $n_{3}=1$, which is the situation of Theorem~\ref{th:split3}.  In this case, if $n_{1},n_{2}\geq 2$ are coprime then there exist $l_{1},l_{2}\in \mathbb{Z}$ such that $n_{3}=1=l_{1}n_{1}+l_{2}n_{2}$, and one of $l_{1}$ and $l_{2}$ must be negative. As we shall see, there does not exist a section in this case. This gives some evidence to support the conjecture that the converse of Theorem~\ref{th:split2}(\ref{it:split2a}) is true, namely that a section on the algebraic level is induced by a geometric section via the non-vanishing vector field, or in other words, the coefficients of $n_{1},\ldots,n_{k-1}$ in the statement of Theorem~\ref{th:split2}(\ref{it:split2b}) must in fact be non negative.

%, it is not clear if the geometric idea apply or not to subtracting points, therefore we do not know if the reciprocal  of Theorem~\ref{th:split2} is valid. For example, if $n_{1},n_{2}\geq 2$ are coprime, we can write $1=l_{1}n_{1}+l_{2}n_{2}$, with $l_{1},l_{2}\in \mathbb{Z}$, and we shall see that we do not have a section, as stated in Theorem~\ref{th:split3}. It gives support to conjecture that the converse of Theorem~\ref{th:split2} could be true.

%\comj{Let $M$ be the $2$-torus or the Klein bottle, let $k=3$, $q=1$ and $n_{3}=1$, let $n_{1}=t$, $n_{2}=s$, where $t,s\geq 2$,?} and let $n=t+s$.  \textcolor{red}{DMarch5 No  comments  about  the new  version}  \comj{Here is an attempt ((4.7) no longer exists).} {\textcolor{blue}{!!!C!!! estou de acordo com a nova versao}}
Let $M$ be the $2$-torus or the Klein bottle, let $k=3$, $q=1$ and $n_{3}=1$, let $n_{1}=t$, $n_{2}=s$, where $t,s\geq 2$, and let $n=t+s$. We study the projection $p_{\ast}\colon\thinspace B_{t,s,1}(M)\longrightarrow B_{t,s}(M)$. In order to prove Theorem~\ref{th:split3}, namely that $p_{\ast}$ does not admit a section, the idea is to assume on the contrary that there exists a section $\phi\colon\thinspace B_{t,s}(M) \to B_{t,s,1}(M)$, and to study the induced homomorphism of certain quotients of the two groups, for example by $\Gamma_{l}(P_{n}(M))$ and $\Gamma_{l}(P_{n+1}(M))$ respectively. If $l=2$, it turns out that this induced homomorphism admits a section, and so with our methods, we need to take a larger value of $l$. As we shall see, $l=3$ will be sufficient.
% \comj{This was suggested by Carolina to replace a paragraph whose conclusion now appears in Section~\ref{sec:contas}.} 
We will make use of the framework of Section~\ref{sec:contas} and the commutative diagram~(\ref{eq:basic}), where we take $s\neq 0$, $m=1$, $H=\Gamma_{3}(P_{n+1}(M))$, $H'=\Gamma_{3}(P_{n}(M))$ and $X=\brak{n+1}$. In order to apply the results of that section, we must first check that conditions~(\ref{eq:relnsa})--(\ref{eq:relnsd}) are satisfied. Since $m=1$,  relation~(\ref{eq:relnsa}) follows from Proposition~\ref{prop:prespmminusn}, and relation~(\ref{eq:relnsc}) holds trivially because $\sigma=1$. To check that relations~(\ref{eq:relnsb}) and~(\ref{eq:relnsd}) hold in our setting, we first give some information about the quotient groups ${B_{t,s}(M)}/{\Gamma_{3}(P_{n}(M))}$ and ${B_{t,s,1}(M)}/{\Gamma_{3}(P_{n+1}(M))}$. If $u$ and $v$ are two elements of a group, let $[u,v]=uvu^{-1}v^{-1}$ denote their commutator.

\begin{prop}\label{prop:Cij} 
Let $M$ be the $2$-torus or the Klein bottle. Then $C_{i,j}\in \Gamma_{2}(P_{n}(M))$ for all $1\leq i < j \leq n$.
\end{prop}

\begin{proof} 
Let $1\leq i < j \leq n$. By relation~(\ref{it:puras2}) of Theorem~\ref{th:puras}, we have:
\begin{equation}\label{eq:cijgamma2}
C^{-1}_{i,j}C_{i+1,j}=a^{-1}_{j}[b^{-1}_{j},a^{-1}_{i}]a_{j}\in \Gamma_{2}(P_{n}(M)).
\end{equation}
Taking $i=j-1$, we see that $C_{j-1,j}\in \Gamma_{2}(P_{n}(M))$, and it follows by reverse induction on $i$ and~(\ref{eq:cijgamma2}) that $C_{i,j}\in \Gamma_{2}(P_{n}(M))$ for all $i=1,\ldots, j-1$.
\end{proof}

\begin{remark}\label{rem:obs.Cij}
Let $1\leq i < j \leq n$ (resp.\ $1\leq i < j \leq n+1$). By Proposition~\ref{prop:Cij}, the element $C_{i,j}$ of  $B_{t,s}(M)/\Gamma_{3}(P_{t+s}(M))$ (resp.\ of $B_{t,s,1}(M)/\Gamma_{3}(P_{t+s+1}(M))$) commutes with the $\Gamma_{3}(P_{t+s}(M))$-coset (resp.\ the $\Gamma_{3}(P_{t+s+1}(M))$-coset) of every element of $P_{t+s}(M)$ (resp.\ of $P_{t+s+1}(M)$).
\end{remark}

%\begin{prop}\label{subs} Let $M$ be the torus or the Klein bottle. The following relations are valid in $B_{n}(M)$. 
%\begin{enumerate}
	%\item\label{subs:1} $\sigma^{-1}_{i}a_{j}\sigma_{i}=\left\{\begin{array}{ll}\sigma^{-2}_{i}a_{i+1},&\,\,j=i\\ a_{i}\sigma^{2}_{i},&\,\,j=i+1\\ a_{j},&\,\,\mbox{other\,\,cases};\end{array}\right.$
		%\item\label{subs:2} $\sigma^{-1}_{i}b_{j}\sigma_{i}=\left\{\begin{array}{ll}b_{i+1}\sigma^{2}_{i},&\,\,j=i\\ \sigma^{-2}_{i}b_{i},&\,\,j=i+1\\ b_{j},&\,\,\mbox{other\,\,cases}.\end{array}\right.$
%
%\item\label{subs:3} $\sigma^{2}_{i}=C_{i,i+1}$, for all $i=1,\ldots, n-1$.
%
%\item\label{subs:4} $\sigma^{-1}_{i}C_{j,k}\sigma_{i}=\left\{\begin{array}{ll} C_{j,k},& (i+1<j<k)\mbox{\,or\,}(j\leq i \leq k-1)\mbox{\,or\,}(j<k<i)\\
%C^{-1}_{k,k+1}C_{j,k+1},& (i=k)\\ 
%C_{j-1,k}C^{-1}_{j,k}C_{j+1,k}& (j=i+1)\end{array}\right.$
%\end{enumerate}
%\end{prop}

% \textcolor{red}{DMAY21 More relations should be added here or in Proposition  4.7, which are going to be used after  proposition 4.7. }
\begin{prop}\label{prop:rel.igual}
Let $M$ be either the $2$-torus or the Klein bottle. The following relations are valid in  ${B_{t,s}(M)}/{\Gamma_{3}(P_{n}(M))}$ and ${B_{t,s,1}(M)}/{\Gamma_{3}(P_{n+1}(M))}$:
\begin{enumerate}
\item\label{igual:0} $a_{i}a_{j}=a_{j}a_{i}$ for $i,j=1,t+1,\ldots,n$.
	
\item\label{igual:1} $b_{j}b_{i}b^{-1}_{j}=\begin{cases}
b_{i} & \text{if $M=\mathbb{T}$ and $i,j=1,t+1,\ldots,n$}\\
a^{-1}_{j}b_{i}a_{j} & \text{if $M=\mathbb{K}$, and either $i=1, j= t+1,\ldots, n$ or $t+1\leq i< j\leq n$.}
\end{cases}$
	
%\left\{\begin{array}{lll},&i,j=1,t+1\ldots,n, & M=\mathbb{T}\\ a^{-1}_{j}b_{i}a_{j},&i=1\,\,\mbox{or}\,\,t+1\leq i< j\leq n,& M=\mathbb{K}\end{array}\right.$

\item\label{igual:2} $a_{n}\sigma_{i}=\sigma_{i}a_{n}$ and $b_{n}\sigma_{i}=\sigma_{i}b_{n}$ for $i=1,\ldots,t-1,t+1,\ldots,n-2$.

%$\left\{\begin{array}{ll}a_{n}\sigma_{i}=\sigma_{i}a_{n},&i=1,\ldots,t-1\\b_{n}\sigma_{i}=\sigma_{i}b_{n},&i=1,\ldots,t-1\end{array}\right.$
	
\item\label{igual:3} $a_{1}\sigma_{i}=\sigma_{i}a_{1}$ and $b_{1}\sigma_{i}=\sigma_{i}b_{1}$ for $i=t+1,\ldots,n-1$.

%$\left\{\begin{array}{ll}a_{1}\sigma_{i}=\sigma_{i}a_{1},&i=t+1,\ldots,n-1\\b_{1}\sigma_{i}=\sigma_{i}b_{1},&i=t+1,\ldots,n-1\end{array}\right.$
		
\item\label{igual:4}  $\sigma_{i}\sigma_{j}=\sigma_{j}\sigma_{i}$ for $1\leq i,j\leq n-1$, where $\lvert i-j \rvert\geq 2$ and $i,j\neq t$.
		
\item\label{igual:5} $\sigma_{i}\sigma_{i+1}\sigma_{i}=\sigma_{i+1}\sigma_{i}\sigma_{i+1}$ for $i=1,\ldots, t-2,t+1,\ldots, n-2$.

\item\label{igual:6} $\sigma_{i}a_{i}\sigma^{-1}_{i}b_{i}=b_{i}\sigma_{i}a_{i}\sigma_{i}$ for $i=1,t+1$.

%\textcolor{red}{DMAY23 The next item (h) the variable $i$ in one place vary  from 1 to n-1 and in the other from 1 to $n$. I changed $n$ to $n-1$ and it is in blue} 
	
\item\label{igual:7} $\sigma^{-1}_{i}a_{i+1}=a_{i}\sigma_{i}$ and $\sigma_{i}b_{i+1}=b_{i}\sigma^{-1}_{i}$ for $i=t+1,\ldots, n-1$.

%$\left\{\begin{array}{ll}\sigma^{-1}_{i}a_{i+1}=a_{i}\sigma_{i},& i=t+1,\ldots, n-1.\\\sigma_{i}b_{i+1}=b_{i}\sigma^{-1}_{i},& i=t+1,\ldots, \textcolor{blue}{n-1}. \end{array}\right.$
	
\item\label{igual:8} for $j=t+1,\ldots, n$, $C_{1,j}C^{-1}_{2,j}=\begin{cases}
a^{-1}_{j}b^{-1}_{1}a_{j}b_{1} & \text{if $M=\mathbb{T}$}\\
(a^{-1}_{j}b^{-1}_{1}a_{j}b_{1})^{-1} & \text{if $M=\mathbb{K}$.}
\end{cases}$ This relation also holds for $j=n+1$ in ${B_{t,s,1}(M)}/{\Gamma_{3}(P_{n+1}(M))}$.
%\left\{\begin{array}{lll} a^{-1}_{j}b^{-1}_{1}a_{j}b_{1},&j=t+1,\ldots, n+1,& M=\mathbb{T}\\(a^{-1}_{j}b^{-1}_{1}a_{j}b_{1})^{-1},&j=t+1,\ldots, n+1,& M=\mathbb{K}.  \end{array}\right.$ 
\end{enumerate}
\end{prop}

\begin{proof}
With the exception of relation~(\ref{igual:1}) in the case $M=\mathbb{K}$, relation~(\ref{igual:6}) if $i=t+1$ and relation~(\ref{igual:8}), all of the relations given in the statement appear in the presentation of $B_{t,s+1}(M)$ in Proposition~\ref{prop:B_{n,m}}, where we view $B_{t,s,1}(M)$ as a subgroup of $B_{t,s+1}(M)$,
% {\textcolor{blue}{and $B_{t,s,1}(M)$ is a subgroup of $B_{t,s+1}(M)$}} \comj{to use this proposition for $B_{t,s,1}(M)$, I imagine that this group should be viewed as a subgroup of $B_{t,s+1}(M)$?} {\textcolor{blue}{!!!C!!! Sim.}}, 
and so are valid in the given quotients. 

If $M=\mathbb{K}$, by relation~(\ref{it:puras6}) of Theorem~\ref{th:puras}, we have $C_{i,j}C_{i+1,j}^{-1}=b_{j^{-1}} b_{i}^{-1} b_{j}b_{i}$, and relation~(\ref{igual:1}) may be obtained by substituting this equality in relation~(\ref{it:puras7}) of Theorem~\ref{th:puras}.
% using~(\ref{it:puras6}) and~(\ref{it:puras7}) of Theorem~\ref{th:puras} \comj{and Remark~\ref{rem:obs.Cij}?} {\textcolor{blue}{!!!C!!! nao precisa usar a remark, pois pela rel. (7) $b^{-1}_{i}a_{j}b_{i}=a_{j}b_{j}(b^{-1}_{j}b^{-1}_{i}b_{j}b_{i})^{-1}b^{-1}_{j}$ usei a rel. (6) dentro do parenteses para substituir $C_{i,j}C^{-1}_{i+1,j}$}}. 
To prove relation~(\ref{igual:6}) for $i=t+1$, by relation~(\ref{it:puras2}) of Theorem~\ref{th:puras} and Remark~\ref{rem:obs.Cij}, we see that $b_{t+2}C_{t+1,t+2}a_{t+1}= a_{t+1}b_{t+2}$. In this equality, we then 
replace $C_{t+1,t+2}$ by $\sigma^{2}_{t+1}$, and $b_{t+2}$ by $\sigma^{-1}_{t+1}b_{t+1}\sigma^{-1}_{t+1}$ using relation~(\ref{igual:7}), and this yields the given relation.
% we have $a^{-1}_{t+1}b_{t+2}=b_{t+2}a^{-1}_{t+1}C^{-1}_{t+1,t+2}$. We then obtain~(\ref{igual:6}) by replacing $C_{t+1,t+2}$ by $\sigma^{2}_{t+1}$, and $b_{t+2}$ by $\sigma^{-1}_{t+1}b_{t+1}\sigma^{-1}_{t+1}$ using relation~(\ref{igual:7}). 
Finally, to prove relation~(\ref{igual:8}) we use  relation~(\ref{it:puras7}) of Theorem~\ref{th:puras} and Remark~\ref{rem:obs.Cij}.
% first notice that by relations~(\ref{it:puras2}) and~(\ref{it:puras3}) \comj{Where is~(\ref{it:puras3}) used? Doesn't it collapse in the quotient?} {\textcolor{blue}{!!!C!!! essa relacao \'{e} valida no grupo, antes de passar o quociente, fiz essa conta em~(\ref{eq:b.sigma.a}), mas da pra fazer sem usar a rel.(3), escrevi acima}} of Theorem~\ref{th:puras} \comj{and Remark~\ref{rem:obs.Cij}?}, we have $a^{-1}_{t+1}b_{t+2}=b_{t+2}a^{-1}_{t+1}C^{-1}_{t+1,t+2}$. We then obtain~(\ref{igual:6}) by replacing $C_{t+1,t+2}$ by $\sigma^{2}_{t+1}$, and $b_{t+2}$ by $\sigma^{-1}_{t+1}b_{t+1}\sigma^{-1}_{t+1}$ using relation~(\ref{igual:7}). Finally, to prove relation~(\ref{igual:8}) we use  relation~(\ref{it:puras7}) of Theorem~\ref{th:puras} and Remark~\ref{rem:obs.Cij}.
\end{proof}

We now list the equations that we will use presently to put certain relations of the quotient group ${B_{t,s}(M)}/{\Gamma_{3}(P_{n}(M))}$ in canonical form.
 
\begin{prop}\label{prop:rel} 
We have the following relations in $B_{t,s,1}(M)/\Gamma_{3}(P_{n+1}(M))$:
\begin{enumerate}
\item\label{rel.1} $C_{l,j}$ commutes with $a_{i}$ and $b_{i}$ for all $i,l,j$ for all $1\leq l<j\leq n+1$ and $i=1,t+1,\ldots, n+1$.

%\comj{for all $1\leq l<j\leq n+1$ and $1\leq i\leq n+1$?} {\textcolor{blue}{!!!C!!! Sim, mas o indice do $i$ tem apenas $1,t+1,\ldots, n+1$}}
	
\item\label{rel.2} for $i=1, t+1,\ldots, n$, $b_{n+1}a_{i}=a_{i}b_{n+1}C^{-1}_{i,n+1}C_{i+1,n+1}$, and:
%$\left\{\begin{array}{ll}b_{n+1}a_{i}=a_{i}b_{n+1}C^{-1}_{i,n+1}C_{i+1,n+1},\quad i=1, t+1,\ldots, n;\\
\begin{equation*}
a_{n+1}b_{i}=\begin{cases}
b_{i}a_{n+1}C_{i,n+1}C^{-1}_{i+1,n+1} & \text{if $M=\mathbb{T}$}\\
b_{i}a_{n+1}(C_{i,n+1}C^{-1}_{i+1,n+1})^{-1} & \text{if $M=\mathbb{K}$.}
\end{cases}
\end{equation*}
%a_{n+1}b_{i}=\left\{\begin{array}{lll} b_{i}a_{n+1}C_{i,n+1}C^{-1}_{i+1,n+1},&i=1, t+1,\ldots, n,&M=\mathbb{T}\\ b_{i}a_{n+1}(C_{i,n+1}C^{-1}_{i+1,n+1})^{-1},&i=1, t+1,\ldots, n, &M=\mathbb{K}\end{array}\right.\end{array}\right.$
	
\item\label{rel.3} 
$b_{n+1}a_{n+1}=\begin{cases}
a_{n+1}b_{n+1}C_{1,n+1} & \text{if $M=\mathbb{T}$}\\
a^{-1}_{n+1}b_{n+1}C_{1,n+1} & \text{if $M=\mathbb{K}$.} 
%\comj{Isn't it $a^{-1}_{n+1}b_{n+1} C_{1,n+1}$?} {\textcolor{blue}{!!!C!!! Sim, arrumei}}}
\end{cases}$

%$b_{n+1}a_{n+1}=\left\{\begin{array}{ll}a_{n+1}b_{n+1}C_{1,n+1},&M=\mathbb{T}\\ a^{-1}_{n+1}b_{n+1}C^{-1}_{1,n+1},&M=\mathbb{K}\end{array}\right.$;\\
	
\item\label{rel.a} $a_{n+1}a_{i}=a_{i}a_{n+1}$ for $i=1,\ldots,n$.

 \item\label{rel.b} for $i=1,\ldots,n$, $b_{n+1}b_{i}=\begin{cases}
 b_{i}b_{n+1} & \text{if $M=\mathbb{T}$}\\
 b_{i}b_{n+1}C_{i,n+1}C^{-1}_{i+1,n+1} & \text{if $M=\mathbb{K}$.}
\end{cases}$
 
% $b_{n+1}b_{i}=\left\{\begin{array}{lll}b_{i}b_{n+1},&i=1,\ldots,n,&M=\mathbb{T}\\
% b_{i}b_{n+1}C_{i,n+1}C^{-1}_{i+1,n+1},&i=1,\ldots,n,&M=\mathbb{K}\end{array}\right.; $

 %\item $\left\{\begin{array}{ll}\sigma_{i}a_{n+1}=a_{n+1}\sigma_{i}, &i=1,\ldots,t-1,t+1,\ldots,n-1\\\sigma_{i}b_{n+1}=b_{n+1}\sigma_{i}, &i=1,\ldots,t-1,t+1,\ldots,n-1\end{array}\right.$
	
% \textcolor{red}{DMAY21 Shall we add below the variation of $i,j$ for the next relations?}\\
 %\textcolor{blue}{$i=1,\ldots...., n-1, i \neq t  \ \ \  j= t+1,\ldots, n$,}

\item\label{rel.4} %for $1\leq i\leq n-1$, where $i \neq t$, and $t+1\leq j\leq  n$:
% \begin{equation*}
% \sigma^{-1}_{i}C_{j,n+1}\sigma_{i}=\begin{cases}
% C_{j,n+1}& \text{ if $j\neq i+1$}\\ 
% C_{j-1,n+1}C^{-1}_{j,n+1}C_{j+1,n+1}&\text{ if $j= i+1$.}
% \end{cases}
% \end{equation*}
for $1\leq i\leq n-1$, where $i \neq t$, and $1\leq l<n+1$:\begin{equation*}
\sigma^{-1}_{i}C_{l,n+1}\sigma_{i}=\begin{cases}
C_{l,n+1}& \text{ if $l\neq i+1$}\\ 
C_{l-1,n+1}C^{-1}_{l,n+1}C_{l+1,n+1}&\text{ if $l= i+1$.}
\end{cases}
\end{equation*}
\end{enumerate}
\end{prop}

\begin{proof}
Relation~(\ref{rel.1}) follows from Remark~\ref{rem:obs.Cij}, relation~(\ref{rel.2}) is a consequence of  relations~(\ref{it:puras2}) and~(\ref{it:puras7}) of Theorem~\ref{th:puras} with $j=n+1$ and Remark~\ref{rem:obs.Cij}, relations~(\ref{rel.3}),~(\ref{rel.a}) and~(\ref{rel.b}) may be deduced from 
relation~(\ref{it:puras5}) with $i=n+1$, relation~(\ref{it:puras1}) with $j=n+1$, and relation~(\ref{it:puras6}) with $j=n+1$ of Theorem~\ref{th:puras} respectively, and relation~(\ref{rel.4}) is a consequence of relation~(\ref{it:k5}) of Proposition~\ref{prop:kernel}. %\comj{? Proposition~\ref{prop:kernel} does not apply for values of $i$ lying between $1$ and $t$? The form of relation~(\ref{it:Bnm7}) of Proposition~\ref{prop:B_{n,m}} looks close, but again the indices are not the same.}  {\textcolor{blue}{!!!C!!! Sim, voce tem razao. \'{E} a relacao (9) da proposicao 2.11, e usa tambem que $B_{t,s}$ \'{e} um subgrupo de $B_{t,s+1}$, e de resto os indices parecem iguais para mim pois se $1\leq l< j $ e nesse caso $j=t+s+1=n+1$, n\~{a}o sei se tem algo a mais que \~{a}o percebi.}} \comj{In Proposition~\ref{prop:B_{n,m}}, $i$ lies in the first block ($1\leq i\leq n-1$), and $j$ is in the second block ($n+1\leq j\leq n+m$), so the two blocks are disjoint. But in relation~(\ref{rel.4}) above, we have $1\leq i\leq t+s-1$ and $t+1 \leq j \leq t+s$, and these two blocks are not disjoint?}{\textcolor{blue}{!!!C!!! Tem que usar as duas proposicoes 2.11 e 2.7. A proposicao 2.11 garante a relacao se os sigmas estao no primeiro bloco, mas neste grupo tambem temos as relacoes do tipo I, que sao as relacoes dadas pela proposicao 2.7, e a relacao (6) da a relacao no caso que os sigmas estam no segundo bloco}} \comj{Can you write down the complete argument used to obtain relation~(\ref{rel.4})?}
%\comc{Minha nova sugest\~{a}o de escrita est\'{a} em azul, veja se esta ok.}

To obtain relation~(\ref{rel.4}), we view $B_{t,s,1}(M)$ as a subgroup of $B_{t,s+1}(M)$ and we make use of the presentation of $B_{t,s+1}(M)$ given in Proposition~\ref{prop:B_{n,m}}. Let $1\leq l <n+1$. If $t+1\leq i\leq t+s-1$ then relation~(\ref{rel.4}) is obtained using relation~(\ref{it:k5}) of Proposition~\ref{prop:kernel}, which is one of the relations of Type~I of Proposition~\ref{prop:B_{n,m}}, and if $1\leq i \leq t-1$, relation~(\ref{rel.4}) follows from relation~(\ref{it:Bnm7}) of Proposition~\ref{prop:B_{n,m}}.
%Using relation~(\ref{it:k5}) of Proposition~\ref{prop:kernel}, which is one of the  relations of Type~I of Proposition~\ref{prop:B_{n,m}}, we obtain relation~(\ref{rel.4}) for $t+1\leq i\leq t+s-1$ and $1\leq l < t+s+1$. Now, using relation~(\ref{it:Bnm7}) of Proposition~\ref{prop:B_{n,m}}, we obtain relation~(\ref{rel.4}) for $1\leq i \leq t-1$ and $1\leq l <t+s+1$.
\end{proof}

As we mentioned just before \repr{Cij}, relations~(\ref{eq:relnsa}) and~(\ref{eq:relnsc}) of Section~\ref{sec:contas} are satisfied in our setting. Relation~(\ref{eq:relnsb}) follows from Propositions~\ref{prop:rel.igual} and~\ref{prop:rel}, the surface relation is a consequence of Proposition~\ref{prop:B_{n,m}}, and relation~(\ref{eq:relnsd}) follows from Remark~\ref{rem:obs.Cij}. We may thus make use of the results of Section~\ref{sec:contas}.

%\comc{mudei a ordem... essa proposicao estava entre Propositions~\ref{prop:rel.igual} and~\ref{prop:rel}}

\begin{prop}\label{prop:can.form} 
%Let $z_{1},z_{2}\in B_{t,s,1}(M)/\Gamma_{3}(P_{n+1}(M))$ be such that $z_{1}=a^{p_{1}}_{n+1}b^{p_{2}}_{n+1}C^{p_{3}}_{1,n+1}\cdots C^{p_{n+2}}_{n,n+1}$ and $z_{2}=a^{q_{1}}_{n+1}b^{q_{2}}_{n+1}C^{q_{3}}_{1,n+1}\cdots C^{q_{n+2}}_{n,n+1}$, where $p_{1},\ldots,p_{n+2}, q_{1},\ldots, q_{n+2} \in \mathbb{Z}$. If $z_{1}=z_{2}$ then for all $i=1,\ldots, n+2$, $p_{i}=q_{i}$ (resp.\ $p_{i}\equiv q_{i} \bmod{2}$) if $M=\mathbb{T}$ (resp.\ $M=\mathbb{K}$).
Let $z_{1},z_{2}\in B_{t,s,1}(M)/\Gamma_{3}(P_{n+1}(M))$ be such that $z_{1}=a^{p_{1}}_{n+1}b^{p_{2}}_{n+1}C^{p_{3}}_{1,n+1}\cdots C^{p_{n+2}}_{n,n+1}$ and $z_{2}=a^{q_{1}}_{n+1}b^{q_{2}}_{n+1}C^{q_{3}}_{1,n+1}\cdots C^{q_{n+2}}_{n,n+1}$, where $p_{1},\ldots,p_{n+2}, q_{1},\ldots, q_{n+2} \in \mathbb{Z}$. Suppose that $z_{1}=z_{2}$.
\begin{enumerate}[(a)]
\item\label{it:can.forma} If $M=\mathbb{T}$, then $p_{i}=q_{i}$ for all $i=1,\ldots, n+2$.
\item\label{it:can.formb} If $M=\mathbb{K}$, then $p_{1} \equiv q_{1} \bmod{4}, p_{2}=q_{2}$, and $p_{i}\equiv q_{i} \bmod{2}$ for all $i=3,\ldots, n+2$.
\end{enumerate}
In particular, if $M=\mathbb{T}$ or $\mathbb{K}$ then $p_{i}\equiv q_{i} \bmod{2}$ for all $i=1,\ldots, n+2$.
\end{prop}

%\textcolor{red}{DMarch7 After fixing the proof I think  that the statement  " (resp.\ $p_{i}\equiv q_{i} \bmod{2}$) if $M=\mathbb{T}$ (resp.\ $M=\mathbb{K}$)." should be either: 
% a) ``(resp.\ $p_{1} \equiv q_{1} mod (4)  \ \ and \ \   p_{i}\equiv q_{i} \bmod{2}$ \ \ for\ \   all\ \  $i>1$) \ \ if $M=\mathbb{T}$ (resp.\ $M=\mathbb{K}$)."
% or  b)   `` (resp.\ $p_{1} \equiv q_{1} mod (4)$, \ \  $p_2=q_2$ \ \  and \ \    $p_{i}\equiv q_{i} \bmod{2}$ for all $i>2$) if $M=\mathbb{T}$ (resp.\ $M=\mathbb{K}$).}
%" (resp.\ $p_{i}\equiv q_{i} \bmod{2}$) if $M=\mathbb{T}$ (resp.\ $M=\mathbb{K}$)."

\begin{proof}
Let $z_{1}$ and $z_{2}$ be as defined in the statement, and suppose that $z_{1}=z_{2}$. Note that $z_{i}\in P_{n+1}(M)/\Gamma_{3}(P_{n+1}(M))$ for $i=1,2$. Let $\rho \colon\thinspace P_{n+1}(M)/\Gamma_{3}(P_{n+1}(M)) \longrightarrow \pi_{1}(M)/\Gamma_{3}(\pi_{1}(M))$ be the homomorphism % \comj{Perhaps this needs to be justified. Imagine for example that there exists a (strange) relation in $P_{n+1}(M)/\Gamma_{3}(P_{n+1}(M))$ that identifies $a_{n+1}$ and $b_{n+1}$. Then $\rho$ is not well defined. For $M=\mathbb{K}$, Daciberg suggests looking at $P_{n+1}(M)/\Gamma_{3}(P_{n+1}(M))\longrightarrow \pi_{1}(M)/\Gamma_{3}(\pi_{1}(M))$. For $M=\mathbb{T}$, there is no problem because $\Gamma_{3}(\pi_{1}(\mathbb{T}))$ is trivial. For $M=\mathbb{K}$, we have $\Gamma_{2}(\pi_{1}(M))= \langle a_{1}^{2} \rangle$ and $\Gamma_{3}(\pi_{1}(M))=\langle a_{1}^{4} \rangle$. So $\pi_{1}(M)/\Gamma_{3}(\pi_{1}(M))\cong \mathbb{Z}_{4} \rtimes \mathbb{Z}$.} {\textcolor{blue}{!!!C!!! ok, concordo.}} 
induced by the homomorphism from $P_{n+1}(M)$ to $\pi_{1}(M)$ that geometrically forgets all but the last string. If $M=\mathbb{T}$, $\Gamma_{3}(\pi_{1}(\mathbb{T}))$ is trivial because $\pi_1(\mathbb{T})\cong \mathbb{Z}^2$ is Abelian. If  $M=\mathbb{K}$, then $\pi_1(\mathbb{K})\cong \langle a_{1} \rangle \rtimes \langle b_{1} \rangle$ where both factors are infinite cyclic and the action is the non-trivial one, $\Gamma_{2}(\pi_{1}(\mathbb{K}))=\langle a_{1}^{2} \rangle$ and $\Gamma_{3}(\pi_{1}(\mathbb{K}))=\langle a_{1}^{4} \rangle$ by~\cite[page~19]{GP}. Thus $\pi_{1}(M)/\Gamma_{3}(\pi_{1}(M))$ is isomorphic to $\mathbb{Z}\times \mathbb{Z}$  (resp.\ to $\mathbb{Z}_{4} \rtimes \mathbb{Z}$) if $M=\mathbb{T}$ (resp.\ if $M=\mathbb{K}$). Since  $a^{p_{1}}_{1}b^{p_{2}}_{1}= \rho(z_{1})= \rho(z_{2})= a^{q_{1}}_{1}b^{q_{2}}_{1}$ in $\pi_{1}(M)/\Gamma_{3}(\pi_{1}(M))$, it follows that $p_{i}=q_{i}$ for $i=1,2$ (resp.\ $p_{1}\equiv q_{1} \bmod{4}$ and $p_{2}=q_{2}$) if $M=\mathbb{T}$ (resp.\ if $M=\mathbb{K}$).

% which is equal to $\mathbb{Z}\times \mathbb{Z}$ if $M=\mathbb{T}$ (resp.\ $\mathbb{Z}_{4} \rtimes \mathbb{Z}$ if $M=\mathbb{K}$), and thus $p_{i}=q_{i}$ for $i=1,2$ {\textcolor{blue}{if $M=\mathbb{T}$ (resp.\ $p_{1}\equiv q_{1} \bmod{4}$ and $p_{2}=q_{2}$ if $M=\mathbb{K}$)}}, which proves the result in these cases. 

For $i=1,\ldots,n$, consider the homomorphism $\rho_{i} \colon\thinspace P_{n+1}(M)/\Gamma_{3}(P_{n+1}(M))\longrightarrow P_{2}(M)/\Gamma_{3}(P_{2}(M))$ induced by the homomorphism from $P_{n+1}(M)$ to $P_{2}(M)$ that geometrically forgets all but the $i$th string and the last string. Then $a^{p_{1}}_{2}b^{p_{2}}_{2}C^{p_{i+2}}_{1,2}= \rho_{i}(z_{1})= \rho_{i}(z_{2})=a^{q_{1}}_{2}b^{q_{2}}_{2}C^{q_{i+2}}_{1,2}$.
\begin{enumerate}[\textbullet]
\item If $M=\mathbb{T}$, from above, we have $p_{i}=q_{i}$ for $i=1,2$, and so $C^{p_{i+2}}_{1,2}= C^{q_{i+2}}_{1,2}$. Now $P_{2}(\mathbb T) \cong \pi_{1}(\mathbb T\setminus\{x_{1}\})\times \mathbb Z^{2}$ by~\cite[Lemma 17]{BGG}, and using the fact that $\pi_{1}(\mathbb{T} \setminus \{x_{1}\})$ is the free group generated by $a_{2}$ and $b_{2}$, it follows that $\Gamma_{i}(P_2(\mathbb T)) \cong \Gamma_{i}(\pi_{1}(\mathbb{T} \setminus \{x_{1}\}))= \Gamma_{i}(\langle a_{2},b_{2}\rangle)$ for all $i\geq 2$. Further, $C_{1,2}=[b^{-1}_{2},a^{-1}_{2}]$ by taking $i=n=2$ in relation~(\ref{it:puras5}) of Theorem~\ref{th:puras}, and by~\cite[page 337, Theorem 5.12]{MKS}, the coset of this element generates the infinite cyclic group $\Gamma_{2}(\langle a_{2},b_{2}\rangle)/\Gamma_{3}(\langle a_{2},b_{2}\rangle)$. 
% In particular $C_{1,2}\neq 1$ in $P_{2}(\mathbb{T})/\Gamma_{3}(P_{2}(\mathbb{T}))$.
% therefore $C_{1,2}\neq 1$ in $P_{2}(\mathbb{T})/\Gamma_{3}(P_{2}(\mathbb{T}))$ {\textcolor{blue}{because $\pi_{1}(\mathbb{T} \setminus \{x_{1}\})$ is the free group generated by $a_{2}$ and $b_{2}$, and $\Gamma_{i}(P_2(\mathbb T))=\Gamma_{i}(\pi_{1}(\mathbb{T} \setminus \{x_{1}\}))= \Gamma_{i}(\langle a_{2},b_{2}\rangle)$ for $i\geq 2$}} \comj{So why is $C_{1,2}\notin \Gamma_{3}(P_{2}(\mathbb{T}))$? Maybe because $\pi_{1}(\mathbb{T} \setminus \{x_{1}\})$ is the free group generated by $a_{2}$ and $b_{2}$, and $\Gamma_{i}(\pi_{1}(\mathbb{T} \setminus \{x_{1}\}))= \Gamma_{i}(\langle a_{2},b_{2}\rangle)$ for $i\geq 2$?} {\textcolor{blue}{!!C!!oct02 - Sim, veja nova versao acima.}}. 
Using the fact that $P_{2}(\mathbb{T})/\Gamma_{3}(P_{2}(\mathbb{T}))$ is torsion free~\cite[Theorem 4]{BGG}, it follows that $p_{i+2}=q_{i+2}$ for all $i=1,\ldots, n$, and this proves part~(\ref{it:can.forma}).
% \comj{Don't we also need to know that $C_{1,2}\neq 1$ in $P_{2}(\mathbb{T})/\Gamma_{3}(P_{2}(\mathbb{T}))$?} \textcolor{blue}{Notice that $P_{2}(\mathbb T)=\pi_{1}(\mathbb T\setminus\{x_{1}\})\times \mathbb Z^{2}$~\cite[Lemma 17]{BGG} and $C_{1,2}=[b^{-1}_{2},a^{-1}_{2}]$ by relation~(\ref{it:puras5}) of Theorem~\ref{th:puras}, therefore $C_{1,2}\neq 1$ in $P_{2}(\mathbb{T})/\Gamma_{3}(P_{2}(\mathbb{T}))$ }, and this proves part~(\ref{it:can.forma}).

\item
%{\textcolor{blue}{If $M=\mathbb{T}$, the fact that $p_{i}=q_{i}$, for $i=1,2$ and $P_{2}(\mathbb{T})/\Gamma_{3}(P_{2}(\mathbb{T}))$ is torsion free implies that $p_{i+2}=q_{i+2}$, for all $i=1,\ldots, n$~\cite[Theorem 4]{BGG}. 
%{\textcolor{purple}{versao ANTIGA oct02 - 
Let $M=\mathbb{K}$. If $X$ is a subset of a group $G$, let $\langle\!\langle X\rangle\! \rangle_{G}$ denote the normal closure of $X$ in $G$.
% \comj{This follows a previous suggestion of Carolina, but I have made a few modifications. Contrary to what I had written before, I have the impression that we don't need to know the order of $a_2$ in $P_{2}(\mathbb{K})/\Gamma_{3}(P_{2}(\mathbb{K}))$, just that $a_2^4 \in P_{2}(\mathbb{K})/\Gamma_{3}(P_{2}(\mathbb{K}))$. See what you think.} \comc{ok, concordo}  
First recall that by~\cite[equation~(5.8)]{GP}, $P_{2}(\mathbb K) \cong \pi_{1}(\mathbb K\setminus\{x_{1}\}) \rtimes \pi_{1}(\mathbb K)$, %\comj{$s$ isn't defined here, so maybe we can just write $\pi_{1}(\mathbb{K})$ for the second factor?}\comc{ok}, 
where $\pi_{1}(\mathbb K\setminus\{x_{1}\})$ is the free group generated by $a_{2}$ and $b_{2}$. By~\cite[Theorem~5.4]{GP}, for all $m\geq 2$, we have:
\begin{equation}\label{eq:gammam}
\Gamma_{m}(P_{2}(\mathbb{K}))=\langle\!\langle a_{2}^{2^{m-1}}, x^{2^{m-i}}\,:\, x\in \Gamma_{i}(\pi_{1}(\mathbb K\setminus\{x_{1}\})), 2\leq i \leq m\rangle\! \rangle_{\pi_{1}(\mathbb K\setminus\{x_{1}\})} \rtimes \langle (a_{1}a_{2})^{2^{m-1}}\rangle.
\end{equation}
%\textcolor{blue}{24JULHO, nova vers\~ao usando as sugest\~oes abaixo: If $m=2,3$, we see from~\reqref{gammam} that $a^{2}_{2}\in\Gamma_{2}(\pi_{1}(P_2(\mathbb{K})))$ and $a^{4}_{2}\in\Gamma_{3}(\pi_{1}(P_3(\mathbb{K})))$. }
If $m=2,3$, we see from~\reqref{gammam} that $a^{2}_{2}\in\Gamma_{2}(P_2(\mathbb{K}))$ and $a^{4}_{2}\in\Gamma_{3}(P_3(
\mathbb{K}))$. From this and the first two paragraphs of this proof we conclude that $C^{p_{i+2}}_{1,2}= C^{q_{i+2}}_{1,2}$ in $P_{2}(\mathbb{K})/\Gamma_{3}(P_{2}(\mathbb{K}))$. % \comj{Using the first paragraph of the proof, doesn't this already imply that $C^{p_{i+2}}_{1,2}= C^{q_{i+2}}_{1,2}$, namely that we can put `We conclude from the first two paragraphs of this proof that $C^{p_{i+2}}_{1,2}= C^{q_{i+2}}_{1,2}$ in $P_{2}(\mathbb{K})/\Gamma_{3}(P_{2}(\mathbb{K}))$' here?}. \textcolor{red}{Dac  I think that the sentence of John propose it is ok. My only comment is if we need to say "From this ($a^{2}_{2}\in\Gamma_{2}(\pi_{1}(P_2(\mathbb{K})))$ and $a^{4}_{2}\in\Gamma_{3}(\pi_{1}(P_3(\mathbb{K})))$), we conclude from (a al John).... . May be not}
% \comj{OK for me. Probably the sentence that follows `Taking $i=n=2$ in\ldots' should be modified accordingly, and maybe some of the other $\mathbb{K}\setminus\{x_{1}\}$?} \comc{ok}
Taking $i=n=2$ in relation~(\ref{it:puras5}) of Theorem~\ref{th:puras}, we obtain $C_{1,2}= [b^{-1}_{2},a_{2}]a^{2}_{2}$, so $C_{1,2}\in \Gamma_{2}(P_{2}(\mathbb K))$, and thus $C_{1,2}^{2}\in \Gamma_{3}(P_{2}(\mathbb K))$ by~\reqref{gammam}. %\comj{Move this (see previous comment)?}% We conclude from the first two paragraphs of this proof that $C^{p_{i+2}}_{1,2}= C^{q_{i+2}}_{1,2}$ in $P_{2}(\mathbb{K})/\Gamma_{3}(P_{2}(\mathbb{K}))$.
%, and so $a^{4}_{2}\in\Gamma_{3}(P_{2}(\mathbb%{K}))$.}
%If $m=2,3$, we see from~\reqref{gammam} that $a^{2}_{2}\in\Gamma_{2}(\pi_{1}(\mathbb{K}\setminus\{x_{1}\}))$ and $a^{4}_{2}\in\Gamma_{3}(\pi_{1}(\mathbb{K}\setminus\{x_{1}\}))$, and so $a^{4}_{2}\in \Gamma_{3}(P_{2}(\mathbb{K}))$. Taking $i=n=2$ in relation~(\ref{it:puras5}) of Theorem~\ref{th:puras}, we obtain $C_{1,2}= [b^{-1}_{2},a_{2}]a^{2}_{2}$, so $C_{1,2}\in \Gamma_{2}(\pi_{1}(\mathbb{K}\setminus\{x_{1}\}))$ and $C_{1,2}^{2}\in \Gamma_{3}(\pi_{1}(\mathbb{K}\setminus\{x_{1}\}))$ by~\reqref{gammam}, and thus $C_{1,2} \in \Gamma_{2}(P_{2}(\mathbb{K}))$ and $C_{1,2}^{2} \in \Gamma_{3}(P_{2}(\mathbb{K}))$. We conclude from the first two paragraphs of this proof that $C^{p_{i+2}}_{1,2}= C^{q_{i+2}}_{1,2}$ in $P_{2}(\mathbb{K})/\Gamma_{3}(P_{2}(\mathbb{K}))$. 
So to prove the result in this case, it suffices to show that $C_{1,2} \notin \Gamma_{3}(P_{2}(\mathbb{K}))$. Suppose on the contrary that $C_{1,2} \in \Gamma_{3}(P_{2}(\mathbb{K}))$. Then using~\reqref{gammam}, we have:
% and the fact that $\Gamma_{3}(\pi_{1}(\mathbb{K}\setminus\{x_{1}\})) \subset \Gamma_{2}(\pi_{1}(\mathbb{K}\setminus\{x_{1}\}))$, in $P_{2}(\mathbb{K})$ we have \comj{I think we need to say something about the $(a_{1}a_{2})$-factor of~\reqref{gammam}?} \comc{ok, concordo com essa primeira parte. Mas ao concluir que $p=0$, temos que $C_{1,2}$ teria que ser um elemento de $\langle\!\langle a_{2}^{2^{m-1}}, x^{2^{m-i}}\,:\, x\in \Gamma_{i}(\pi_{1}(\mathbb K\setminus\{x_{1}\})), 2\leq i \leq m\rangle\! \rangle_{\pi_{1}(\mathbb K\setminus\{x_{1}\})}$, que \'{e} um subgrupo do grupo livre $<a_{2},b_{2}>$. Ent\~{a}o, a partir de agora vamos estudar se a equa\c{c}\~{a}o abaixo pode existir no grupo livre, e $a^{2}_{2}\notin\Gamma_{2}(<a_{2},b_{2}>)$} \comj{In that case, we should modify some of what is written above, after~\reqref{gammam}?}:
\begin{equation}\label{eq:exprC12}
[b^{-1}_{2},a_{2}]a^{2}_{2}=C_{1,2}=\left( \prod^{l}_{i=1}\alpha_{i}a_{2}^{4\epsilon_{i}}\alpha^{-1}_{i} \ldotp x_{i}\right) (a_{1}a_{2})^{4p},
\end{equation}
where $p\in \mathbb{Z}$, and for $1\leq i\leq l$, $\alpha_{i}\in\pi_{1}(\mathbb K\setminus\{x_{1}\})$, $x_{i}\in\Gamma_{2}(\pi_{1}(\mathbb K\setminus\{x_{1}\}))$, and $\epsilon_{i}\in \mathbb{Z}$. Taking the image of  $C_{1,2}$ under the projection from $P_{2}(\mathbb{K})$ onto $P_{1}(\mathbb{K})$ given by forgetting the second string, it follows from~\reqref{exprC12} that $a_{1}^{4p}=1$ in $P_{1}(\mathbb{K})$, and so $p=0$.  
So~\eqref{eq:exprC12} is a relation in the free group $\pi_{1}(\mathbb K\setminus\{x_{1}\})$ generated by $a_{2}$ and $b_{2}$, and projecting this equation into
the Abelianisation $\pi_{1}(\mathbb K\setminus\{x_{1}\})/\Gamma_{2}(\pi_{1}(\mathbb K\setminus\{x_{1}\}))$ which is a free Abelian group, we obtain $a^{2}_{2}=a^{4{\sum^{l}_{i=1}\epsilon_{i}}}_{2}$, which yields a contradiction.
%\comj{As was mentioned above, $a^{2}_{2}\in \Gamma_{2}(\pi_{1}(\mathbb K\setminus\{x_{1}\})$, so isn't~\reqref{exprC12} just sent to the equality $1=1$?} \comc{$a^{2}_{2}\notin \Gamma_{2}(\pi_{1}(\mathbb K\setminus\{x_{1}\})$, $a^{2}_{2}$ pertence ao $\Gamma_{2}(P_{2}(M)))$, aqui estamos olhando no grupo livre gerado por $a_{2},b_{2}$}.
%\comj{I rewrote this to take account of the above arguments:}\textcolor{red}{Dac Fine}  
We conclude that $C_{1,2} \notin \Gamma_{3}(P_{2}(\mathbb{K}))$, and since $C_{1,2}^{2} \in \Gamma_{3}(P_{2}(\mathbb{K}))$, it follows from the fact that $C^{p_{i+2}}_{1,2}= C^{q_{i+2}}_{1,2}$ in $P_{2}(\mathbb{K})/\Gamma_{3}(P_{2}(\mathbb{K}))$ that $p_{i+2}\equiv q_{i+2} \bmod{2}$,
% Now, taking $n=3$ and $i=2$ in~\cite[Theorem 5.4]{GP}, if $x\in \Gamma_{2}(\pi_{1}(\mathbb{K} \setminus\{x_{1}\}))$ then $x^{2}\in \Gamma_{3}(P_{2}(\mathbb{K}))$, and therefore $a^4_{2},C^{2}_{1,2}\in \Gamma_{3}(P_2(\mathbb{K}))$, so $a_2$ is of order $4$ and $C_{1,2}$ is of order $2$ in $P_{2}(\mathbb{K})/\Gamma_{3}(P_{2}(\mathbb{K}))$.
%}}
which proves part~(\ref{it:can.formb}) for $M=\mathbb{K}$. The last part of the statement then follows easily.\qedhere
\end{enumerate}
%Since $p_{1}\equiv q_{1} \bmod{4}$ and $p_{2}=q_{2}$, it follows that $C^{p_{i+2}}_{1,2}=C^{q_{i+2}}_{1,2}$ and we conclude that $p_{i}\equiv q_{i} \bmod{2}$ for all $i=1,\ldots, n$ as required.
%Since \textcolor{red}{$p_{1}\equiv q_{1}\ mod(4)$}  $p_{1}=q_{1}$ and $p_{2}=q_{2}$ {\textcolor{blue}{!!!C!! vai ter que ajustar aqui pois agora nao tem a igualdade na garrafa de klein, reescrevi no paragrafo de cima}} from the previous paragraph, it follows that $C^{p_{i+2}}_{i,n+1}=C^{q_{i+2}}_{i,n+1}$. If $M=\mathbb{T}$, the fact that $P_{2}(\mathbb{T})/\Gamma_{3}(P_{2}(\mathbb{T}))$ is torsion free implies that $p_{i+2}=q_{i+2}$, for all $i=1,\ldots, n$~\cite[Theorem 4]{BGG}. If $M=\mathbb{K}$, $C_{i,n+1}$ is of order $2$ in $P_{2}(\mathbb{K})/\Gamma_{3}(P_{2}(\mathbb{K}))$~\cite[Theorem 5.25]{GP} {\textcolor{blue}{!!!C!!! o Teorema 5.4 est\'{a} descrito com mais detalhes para $P_{2}$ e da para entender melhor a ordem, n\~{a}o sei se precisa destacar para o leitor que $C_{1,2}\in \Gamma_{2}(P_{2}(\mathbb{K}))$ para ver que ele tem ordem $2$ no quociente}}, and we conclude that $p_{i}\equiv q_{i} \bmod{2}$ for all $i=1,\ldots, n$ as required. 
\end{proof}

We now come back to the section %\comj{Should $\mathbb{T}$ be $M$ in what follows?} 
$\widehat{\phi}\colon\thinspace B_{t,s}(M)/\Gamma_{3}(P_{n}(M)) \to B_{t,s,1}(M)/\Gamma_{3}(P_{n+1}(M))$ for the induced homomorphism $\widehat{p}_{\ast}\colon\thinspace B_{t,s,1}(M)/\Gamma_{3}(P_{n+1}(M)) \to B_{t,s}(M)/\Gamma_{3}(P_{n}(M))$. It may be defined on the following elements of $B_{t,s}(M)/\Gamma_{3}(P_{n}(M))$ by:
%\comj{of $B_{t,s}(\mathbb{T})/\Gamma_{3}(P_{n}(\mathbb{T}))$? I am not sure that these elements generate the group, but maybe we don't need to look at a complete set of generators. In that case, maybe it would be better to write `defined on the following elements of\ldots'?} {\textcolor{blue}{!!!C!!! ok}} by:
\begin{equation}\label{eq:phi1} 
\begin{cases}
\widehat{\phi}(\sigma_{i})= \sigma_{i}\cdot a^{s_{i,1}}_{n+1} b^{s_{i,2}}_{n+1} C^{r_{i,1}}_{1,n+1}\cdots C^{r_{i,n}}_{n,n+1} & \text{for $i=1,\ldots,t-1,t+1,\ldots n-1$}\\
\widehat{\phi}(a_{i})= a_{i}\cdot a^{\alpha_{i,1}}_{n+1} b^{\alpha_{i,2}}_{n+1} C^{x_{i,1}}_{1,n+1} \cdots C^{x_{i,n}}_{n,n+1} & \text{for $i=1,t+1,t+2,\ldots,n$}\\
\widehat{\phi}(b_{i})= b_{i}\cdot a^{\beta_{i,1}}_{n+1} b^{\beta_{i,2}}_{n+1} C^{y_{i,1}}_{1,n+1} \cdots C^{y_{i,n}}_{n,n+1} & \text{for $i=1,t+1,t+2,\ldots,n$,}
\end{cases}
\end{equation}
where $s_{i,j},r_{i,j},\alpha_{i,j},\beta_{i,j},x_{i,j}, y_{i,j} \in \mathbb{Z}$ for the relevant values of $i$ and $j$. Since we are working with mixed braid groups, $\sigma_{t}$ is not an element of $B_{t,s}(M)$.
If $M=\mathbb{K}$, by Proposition~\ref{prop:can.form}, any conclusion about the coefficients will be modulo $2$. 
%\textcolor{red}{DMarch7  Not exactly due to the new statement of Proposition~\ref{prop:can.form}} \comj{It looks like modulo $2$ is OK thanks to the last part of the statement of Proposition~\ref{prop:can.form}.}   
So for both $\mathbb{T}$ and $\mathbb{K}$, the computations that follow will be carried out with coefficients in $\mathbb{Z}_{2}$, in accordance with the last part of the statement of that proposition.

\begin{lem}\label{lem:lemgen}
With the above notation, we have:
\begin{enumerate}[(a)]
\item\label{it:lg1} for $i=1,t+1,\ldots,n$ and $k=1,2$, $\alpha_{i,k}\equiv 0$ and $\beta_{i,k}\equiv 0 \bmod{2}$.
\item\label{it:lg2} for $i=1,\ldots,t-1,t+1,\ldots,n-1$ and $k=1,2$, $s_{i,k}\equiv 0 \bmod{2}$. 
% \comj{Is this information used in what follows? I couldn't find it referenced in what follows.}\comc{sim, usamos para terminar a prova do item (a) e (c), veja abaixo}
\item\label{it:lg3} for $i=1,t+1$, $r_{i,i+1}\equiv 0 \bmod{2}$.
\end{enumerate}
\end{lem}

\begin{proof}\mbox{}
We start by supposing that $i=t+1,\ldots,n$. We will study the coefficients of $a_{n+1}$ and $b_{n+1}$ in the image of relation~(\ref{igual:7})
%, $\sigma^{-1}_{i}a_{i+1}=a_{i}\sigma_{i}$ and $\sigma_{i}b_{i+1}=b_{i}\sigma^{-1}_{i}$, 
of Proposition~\ref{prop:rel.igual} by $\widehat{\phi}$. 
%\comj{I added some details for the first case:} 
Using relations~(\ref{rel.1}),~(\ref{rel.2}),~(\ref{rel.4}) of Proposition~\ref{prop:rel},~(\ref{igual:2}) of Proposition~\ref{prop:rel.igual} and Remark~\ref{rem:obs.Cij}, we have:
%{\textcolor{blue}{!!!C!!! Fiquei na duvida se os expoentes dos $C_{j,n+1}$ estao corretos na conta abaixo, embora eles n\~{a}o interferem na conta, talvez ficaria mais facil se escrevermos um $w$ sendo uma palavra nesses elementos e deixasse apenas os $a_{n+1},b_{n+1}$ destacados 
\begin{align*}
\widehat{\phi}(\sigma^{-1}_{i}a_{i+1})&= C^{-r_{i,n}}_{n,n+1} \cdots C^{-r_{i,1}}_{1,n+1} b^{-s_{i,2}}_{n+1} a^{-s_{i,1}}_{n+1} \sigma_{i}^{-1} a_{i+1}\cdot a^{\alpha_{i+1,1}}_{n+1} b^{\alpha_{i+1,2}}_{n+1} C^{x_{i+1,1}}_{1,n+1}\cdots C^{x_{i+1,n}}_{n,n+1}\\
&= \sigma_{i}^{-1} a_{i+1} b^{-s_{i,2}}_{n+1} a^{\alpha_{i+1,1}-s_{i,1}}_{n+1} b^{\alpha_{i+1,2}}_{n+1} w,\\
%\end{align*}
\intertext{and}
%where $w$ is a word in $C_{j,n+1}$, $j=1,\ldots,n$. Similarly, we have:
%\begin{align*}
\widehat{\phi}(a_{i}\sigma_{i}) &= a_{i} \sigma_{i} a^{\alpha_{i,1}}_{n+1} b^{\alpha_{i,2}}_{n+1} a^{s_{i,1}}_{n+1} b^{s_{i,2}}_{n+1} w', 
\end{align*}
where $w,w'$ are words in the $C_{j,n+1}$, $j=1,\ldots,n$.
% \begin{align*}
% \widehat{\phi}(\sigma^{-1}_{i}a_{i+1})=& C^{-r_{i,n}}_{n,n+1} \cdots C^{-r_{i,1}}_{1,n+1} b^{-s_{i,2}}_{n+1} a^{-s_{i,1}}_{n+1} \sigma_{i}^{-1} a_{i+1}\cdot a^{\alpha_{i+1,1}}_{n+1} b^{\alpha_{i+1,2}}_{n+1} C^{x_{i+1,1}}_{1,n+1}\cdots C^{x_{i+1,n}}_{n,n+1}\\
% =& \sigma_{i}^{-1} a_{i+1} b^{-s_{i,2}}_{n+1} a^{\alpha_{i+1,1}-s_{i,1}}_{n+1} b^{\alpha_{i+1,2}}_{n+1} C^{x_{i+1,1}-r_{i,1}}_{1,n+1}\cdots C_{i-1,n+1}^{x_{i+1,i-1}-r_{i,i-1}-r_{i,n+1}} C_{i,n+1}^{x_{i+1,i}+r_{i,i}}\\
% &  C_{i+1,n+1}^{x_{i+1,i+1}-r_{i,i+1}+s_{i+2}-r_{i,n+1}} C^{x_{i+1,i+2}-r_{i,i+2}-s_{i+2}}_{i+2,n+1} \cdots C^{x_{i+1,n}-r_{i,n}}_{n,n+1}.
% \end{align*}
% Similarly, we have:
% \begin{align*}
% \widehat{\phi}(a_{i}\sigma_{i}) =& a_{i} \sigma_{i} a^{\alpha_{i,1}}_{n+1} b^{\alpha_{i,2}}_{n+1} a^{s_{i,1}}_{n+1} b^{s_{i,2}}_{n+1} C^{r_{i,1}+x_{i,1}}_{1,n+1} \cdots C^{r_{i-1,n}+x_{i-1,n}+x_{i,i}}_{i-1,n+1}\\
% & C_{i,n+1}^{r_{i,i}-x_{i,i}} C^{r_{i+1,n}+x_{i+1,n}+x_{i,i}}_{i+1,n+1} \cdots C^{r_{i,n}+x_{i,n}}_{n,n+1} 
% \end{align*}
Since $\widehat{\phi}(\sigma^{-1}_{i}a_{i+1})=\widehat{\phi}(a_{i}\sigma_{i})$ and $\sigma^{-1}_{i}a_{i+1}=a_{i}\sigma_{i}$ in $B_{t,s,1}(\mathbb{T})/\Gamma_{3}(P_{n+1}(\mathbb{T}))$, to be able to compare the coefficients of $a_{n+1}$ and $b_{n+1}$, we need to put $b^{-s_{i,2}}_{n+1} a^{\alpha_{i+1,1}-s_{i,1}}_{n+1} b^{\alpha_{i+1,2}}_{n+1}$ and $a^{\alpha_{i,1}}_{n+1} b^{\alpha_{i,2}}_{n+1} a^{s_{i,1}}_{n+1} b^{s_{i,2}}_{n+1}$ in canonical form. To do so, it suffices to conjugate $a^{\alpha_{i+1,1}-s_{i,1}}_{n+1}$ by $b^{-s_{i,2}}_{n+1}$, and $a^{s_{i,1}}_{n+1}$ by $b^{\alpha_{i,2}}_{n+1}$, which we do using~(\ref{rel.3}) of Proposition~\ref{prop:rel}. If $M=\mathbb{K}$, this may alter the sign of the exponent of $a_{n+1}$, but modulo $2$, this exponent remains the same. By comparing the coefficients of $a_{n+1}$ and $b_{n+1}$, it follows from relations~(\ref{rel.1}) and~(\ref{rel.3}) of Proposition~\ref{prop:rel}  and Proposition~\ref{prop:can.form} that for $i=t+1,\ldots,n-1$ and $k=1,2$:
%\sigma_{i}\cdot a^{s_{i,1}}_{n+1} b^{s_{i,2}}_{n+1} C^{r_{i,1}}_{1,n+1}C^{r_{i,2}}_{2,n+1}\cdots C^{r_{i,n}}_{n,n+1} & \text{for $i=1,\ldots,t-1,t+1,\ldots n-1$}\\
%\widehat{\phi}(a_{i})= a_{i}\cdot a^{\alpha_{i,1}}_{n+1} b^{\alpha_{i,2}}_{n+1} C^{x_{i,1}}_{1,n+1}C^{x_{i,2}}_{2,n+1}\cdots C^{x_{i,n}}_{n,n+1} & \text{for $i=1,t+1,t+2,\ldots,n$}\\
%\widehat{\phi}(b_{i})= b_{i}\cdot a^{\beta_{i,1}}_{n+1} b^{\beta_{i,2}}_{n+1} C^{y_{i,1}}_{1,n+1}C^{y_{i,2}}_{2,n+1}\cdots C^{y_{i,n}}_{n,n+1} & \text{for $i=1,t+1,t+2,\ldots,n$,}
%\textcolor{red}{DMAY23  Revise what Carol wrote}
%%the writing of the  sentence which follows. In fact it change but not mod 2}  By Proposition~\ref{rel} that the coefficient of $a_{n+1}$ and $b_{n+1}$ do not change when we consider %the canonical form of the relation. Therefore 
%{\color{blue}!!!C!!!MAY22 By Proposition~\ref{prop:rel}, the coefficient of $b_{n+1}$ do not change when we consider the canonical form of the relation, and the coefficient of $a_{n+1}$ may change if we use relation~(\ref{rel.3}), but considering $\bmod{2}$, we have} 
\begin{align}%\label{alphabeta11}
\alpha_{i+1,k}\equiv \alpha_{i,k} \bmod{2}\label{eq:alphabeta}\\
\beta_{i+1,k}\equiv \beta_{i,k} \bmod{2}.\label{eq:alphabeta1}
\end{align}
Applying induction for $i=t+1,\ldots,n$, to prove the result, it suffices to show that
%\textcolor{blue}{ So in order to show that  $\alpha_{i,2}\equiv 0$ for $i=1, t+1,\ldots,n$   it suffices to show
$\alpha_{1,k}\equiv\alpha_{n,k}\equiv 0$ and  $\beta_{1,k}\equiv\beta_{n,k}\equiv 0$, for $k=1,2$.
%{\textcolor{blue}{$\alpha_{1,k}\equiv\alpha_{n,k}\equiv 0$ and  $\beta_{1,k}\equiv\beta_{n,k}\equiv 0$, for $k=1,2$}} \comj{And for $k=1$?} {\textcolor{blue}{!!!C!!! Sim, arrumei, e o mesmo para o $\beta$}}.
We now analyse the image of relation~(\ref{igual:0}) of Proposition~\ref{prop:rel.igual} by $\widehat{\phi}$. Since we will be comparing coefficients modulo $2$, it will be convenient not to take into account the signs of certain exponents. Using relations~(\ref{rel.1})--(\ref{rel.a}) of Proposition~\ref{prop:rel}, for $i,j=1,t+1,\ldots, n$, we have:
%{\textcolor{blue}{for $i,j=1,t+1,\ldots, n$}}: \comj{for $i,j=\ldots$?}
\begin{align}
\widehat{\phi}(a_{i}a_{j})&= a_{i}a^{\alpha_{i,1}}_{n+1}b^{\alpha_{i,2}}_{n+1}a_{j}a^{\alpha_{j,1}}_{n+1}b^{\alpha_{j,2}}_{n+1}w
=a_{i}a_{j}a^{\alpha_{i,1}}_{n+1}b^{\alpha_{i,2}}_{n+1}C^{\alpha_{i,2}}_{j,n+1}C^{\alpha_{i,2}}_{j+1,n+1}
a^{\alpha_{j,1}}_{n+1}b^{\alpha_{j,2}}_{n+1}w\notag\\
&=a_{i}a_{j}a^{\alpha_{i,1}+\alpha_{j,1}}_{n+1}b^{\alpha_{j,2}+\alpha_{i,2}}_{n+1}C^{\alpha_{i,2}\alpha_{j,1}}_{1,n+1}C^{\alpha_{i,2}}_{j,n+1}
C^{\alpha_{i,2}}_{j+1,n+1}w,\label{eq:relaiaj}
\end{align}
%\textcolor{red}{DMAY21 Seems  that  there  is  a  problema  here. Should be: $C^{-\alpha_{i,2}}_{j,n+1}C^{\alpha_{i,2}}_{j+1,n+1}$ 
%and  $C^{\alpha_{i,2}\alpha_{j,1}}_{1,n+1}C^{-\alpha_{i,2}}_{j,n+1}C^{\alpha_{i,2}}_{j+1,n}$}
%{\bf\color{blue}!!C!!MAY22 sim, voce esta correto, mas como a conclus\~{a}o vai ser mod 2, eu escrevi os expoentes sempre positivos para facilitar}
%\textcolor{red}{DMAY23  Concordo e vou remover o sinal pois estamos  mod (2)}
where $w=\prod^{n}_{k=1}C^{x_{i,k}+x_{j,k}}_{k,n+1}$. In a similar manner, we obtain: 
\begin{equation}\label{eq:relajai}
\widehat{\phi}(a_{j}a_{i})=a_{j}a_{i}a^{\alpha_{i,1}+\alpha_{j,1}}_{n+1}b^{\alpha_{j,2}+\alpha_{i,2}}_{n+1}C^{\alpha_{j,2}\alpha_{i,1}}_{1,n+1}C^{\alpha_{j,2}}_{i,n+1}C^{\alpha_{j,2}}_{i+1,n+1}w.
\end{equation}
First let $i=1$ and $j=t+1$ in equations~(\ref{eq:relaiaj}) and~(\ref{eq:relajai}). Since $\widehat{\phi}(a_{i}a_{j})=\widehat{\phi}(a_{j}a_{i})$ and $a_{i}a_{j}= a_{j}a_{i}$ in $B_{t,s,1}(\mathbb{T})/\Gamma_{3}(P_{n+1}(\mathbb{T}))$ by relation~(\ref{rel.a}) of Proposition~\ref{prop:rel}, we obtain 
$C^{\alpha_{1,2}\alpha_{t+1,1}}_{1,n+1} C^{\alpha_{1,2}}_{t+1,n+1} C^{\alpha_{1,2}}_{t+2,n+1} =C^{\alpha_{t+1,2}\alpha_{1,1}}_{1,n+1} C^{\alpha_{t+1,2}}_{1,n+1} C^{\alpha_{t+1,2}}_{2,n+1}$. Comparing the coefficients of $C_{t+2,n+1}$ and using Proposition~\ref{prop:can.form}, we conclude that:
\begin{equation}\label{eq:alpha_{i,2}}
\alpha_{1,2}\equiv 0 \bmod{2}.
\end{equation} 
Now take $i=t+1$ and $j=n$ in equations~(\ref{eq:relaiaj}) and~(\ref{eq:relajai}). In a similar way, we obtain   
$C^{\alpha_{t+1,2}\alpha_{n,1}}_{1,n+1}C^{\alpha_{t+1,2}}_{n,n+1}=C^{\alpha_{n,2}\alpha_{t+1,1}}_{1,n+1}C^{\alpha_{n,2}}_{t+1,n+1}C^{\alpha_{n,2}}_{t+2,n+1}$. 
If $s>2$ (resp.\ $s=2$) then comparing the coefficients of $C_{n,n+1}$ (resp.\ of $C_{n-1,n+1}$) and using Proposition~\ref{prop:can.form}, we see that:
%then follows by looking the exponents  of $C_{n,n+1}$ that  
\begin{equation}\label{eq:alpha_{i,2}b}
\text{$\alpha_{t+1,2}\equiv 0 \bmod{2}$ (resp.\ $\alpha_{n,2}\equiv 0 \bmod{2}$)}.
\end{equation}  
%If $s=2$ then the equality above becomes 
%$C^{\alpha_{n-1,2}\alpha_{n,1}}_{1,n+1}C^{\alpha_{n-1,2}}_{n,n+1}=C^{\alpha_{n,2}\alpha_{n-1,1}}_{1,n+1}C^{\alpha_{n,2}}_{n-1,n+1}C^{\alpha_{n,2}}_{n,n+1}$.
%By looking at the exponents of $C_{n-1,n+1}$ we obtain that 
%%and of  $C_{n,n+1}$ we obtain that 
%\begin{equation}\label{alpha_{i,2}}
%   \alpha_{n,2}\equiv 0 \bmod{2}.
%\end{equation}
We deduce from~(\ref{eq:alphabeta}),~(\ref{eq:alpha_{i,2}}) and~(\ref{eq:alpha_{i,2}b}) that for $i=1,t+1, \ldots,n$:
\begin{equation}\label{eq:alpha_{i,2}c}
\alpha_{i,2}\equiv 0 \bmod{2}.
\end{equation}
%\begin{equation}\label{alpha_{i,|2}}
%    \alpha_{n,2}\equiv \alpha_{n-1,2} \bmod{2}.
%\end{equation} So  
%\begin{equation}\label{alpha_{i,2}}
%   \alpha_{t+1,2}\equiv 0 \bmod{2}.
%\end{equation}}
%\begin{equation}\label{alpha_{i,22}}
%    \alpha_{n-1,2}=\alpha_{t+1,2}\equiv 0 \bmod{2}.
%\end{equation}.

%{\bf\color{blue}!!C!!MAY22 o que voce fez acima esta certo, mas acho que ainda precisa de alguma frase como a de baixo para os casos que n\~{a}o s\~{a}o os extremos, para poder %concluir para todo $i$}\textcolor{red}{DMAY23 A new sentence  was added above}
%\textcolor{blue}{   If we combine these two cases with eq. ~\ref{alphabeta}       the result follows.}

%\textcolor{red}{DMAY21 The next two lines  should be removed, if make sense what I wrote above}

%Comparing the coefficients of $C_{i,n+1}$, follows that 
%\begin{equation}\label{alpha_{i,2}}
 %   \alpha_{i,2}\equiv 0 \bmod{2},\,i=1,t+1,\ldots,n.
%\end{equation}

We now consider relation~(\ref{igual:1}) of Proposition~\ref{prop:rel.igual}. If $M=\mathbb{T}$, then arguing as above, for $i=1,t+1,\ldots,n$, we obtain:
\begin{equation}\label{eq:beta_i,1.toro}
\beta_{i,1} \equiv 0 \bmod{2}.
\end{equation}
If $M=\mathbb{K}$, for either $i=1, j=t+1,\ldots,n$ or $t+1\leq i<j\leq n$,
%{\textcolor{blue}{either $i=1, j=t+1,\ldots,n$ or $t+1\leq i<j\leq n$}} $i<j$ \comj{Other restrictions on $i,j$? We need some to apply~(\ref{eq:alpha_{i,2}c}) later and to obtain the conclusion.}, 
using relations~(\ref{rel.1})--(\ref{rel.3}) and~(\ref{rel.b}) of Proposition~\ref{prop:rel}, we have:
\begin{align}
\widehat{\phi}(b_{j}b_{i}b^{-1}_{j})&=b_{j}a^{\beta_{j,1}}_{n+1}b^{\beta_{j,2}}_{n+1}b_{i}a^{\beta_{i,1}-\beta_{j,1}}_{n+1}b^{\beta_{i,2}-\beta_{j,2}}_{n+1}C^{\beta_{j,1}(\beta_{i,2}-\beta_{j,2})}_{1,n+1}b^{-1}_{j}w\notag\\
&=b_{j}b_{i}a^{\beta_{j,1}}_{n+1}b^{\beta_{j,2}}_{n+1}C^{\beta_{j,1}+\beta_{j,2}}_{i,n+1}C^{\beta_{j,1}+\beta_{j,2}}_{i+1,n+1}a^{\beta_{i,1}-\beta_{j,1}}_{n+1}b^{\beta_{i,2}-\beta_{j,2}}_{n+1}C^{\beta_{j,1}(\beta_{i,2}-\beta_{j,2})}_{1,n+1}b^{-1}_{j}w\notag\\
& =b_{j}b_{i}a^{\beta_{i,1}}_{n+1}b^{\beta_{i,2}}_{n+1}C^{\beta_{j,2}(\beta_{i,1}-\beta_{j,1})}_{1,n+1}C^{\beta_{j,1}(\beta_{i,2}-\beta_{j,2})}_{1,n+1} C^{\beta_{j,1}+\beta_{j,2}}_{i,n+1} C^{\beta_{j,1}+\beta_{j,2}}_{i+1,n+1}b^{-1}_{j}w \notag\\
&=b_{j}b_{i}b^{-1}_{j}a^{\beta_{i,1}}_{n+1}b^{\beta_{i,2}}_{n+1}C^{\beta_{j,2}\beta_{i,1}+\beta_{i,2}\beta_{j,1}}_{1,n+1}C^{\beta_{j,1}+\beta_{j,2}}_{i,n+1}C^{\beta_{j,1}+\beta_{j,2}}_{i+1,n+1}C^{\beta_{i,1}+\beta_{i,2}}_{j,n+1}C^{\beta_{i,1}+\beta_{i,2}}_{j+1,n+1}w,\label{eq:bjbiK1}
\end{align}
where $w=\prod^{n}_{k=1}C^{y_{i,k}}_{k,n+1}$. Also, applying~(\ref{rel.1})--(\ref{rel.a}) of Proposition~\ref{prop:rel} and~(\ref{eq:alpha_{i,2}c}), we see that:
\begin{align}
\widehat{\phi}(a^{-1}_{j}b_{i}a_{j})&=a^{-\alpha_{j,1}}_{n+1}a^{-1}_{j}b_{i}a^{\beta_{i,1}}_{n+1}b^{\beta_{i,2}}_{n+1}a_{j}a^{\alpha_{j,1}}_{n+1}w=a^{-1}_{j}b_{i}a^{\beta_{i,1}-\alpha_{j,1}}_{n+1}C^{\alpha_{j,1}}_{i,n+1}C^{\alpha_{j,1}}_{i+1,n+1}b^{\beta_{i,2}}_{n+1}a_{j}a^{\alpha_{j,1}}_{n+1}w\notag\\
&=a^{-1}_{j}b_{i}a_{j}a^{\beta_{i,1}}_{n+1}b^{\beta_{i,2}}_{n+1}C^{\beta_{i,2}\alpha_{j,1}}_{1,n+1}C^{\beta_{i,2}}_{j,n+1}C^{\beta_{i,2}}_{j+1,n+1}C^{\alpha_{j,1}}_{i,n+1}C^{\alpha_{j,1}}_{i+1,n+1}w. \label{eq:bjbiK2}
\end{align}
Making use of relation~(\ref{igual:1}) of Proposition~\ref{prop:rel.igual} in $B_{t,s,1}(\mathbb{T})/\Gamma_{3}(P_{n+1}(\mathbb{T}))$, and comparing the coefficients of $C_{j,n+1}$ for the given values of $j$, 
%\comj{for which values of $j$?} {\textcolor{blue}{!!!C!!! apenas o $C_{j,n+1}$ expoente na primeira equacao \'{e} $\beta_{i,1}+\beta_{i,2}$ e na segunda \'{e} $\beta_{i,2}$, e o valor de $i$ pode ser $1,t+1,\ldots, n-1$ e $j\neq 1$}}
in~(\ref{eq:bjbiK1}) and~(\ref{eq:bjbiK2}) using Proposition~\ref{prop:can.form}, it follows that $\beta_{i,1}\equiv 0 \bmod{2}$ for $i=1,t+1,\ldots,n-1$, and applying~(\ref{eq:alphabeta}), we also deduce the result for $i=n$. Therefore, for $i=1,t+1,\ldots,n$, we obtain:
\begin{equation}\label{eq:beta_{i,1}}
\beta_{i,1}\equiv 0 \bmod{2}.
\end{equation}
It follows from~(\ref{eq:alpha_{i,2}c}),~(\ref{eq:beta_i,1.toro}) and~(\ref{eq:beta_{i,1}}) that $\alpha_{i,2} \equiv \beta_{i,1}\equiv 0 \bmod{2}$ for $i=1,t+1,\ldots,n$, which proves half of part~(\ref{it:lg1}) of the statement. Before showing that the other congruences of part~(\ref{it:lg1}) hold, we first prove part~(\ref{it:lg2}). To do so, we start by studying the image of relations~(\ref{igual:3}) of Proposition~\ref{prop:rel.igual} by $\widehat{\phi}$. Using~(\ref{eq:alpha_{i,2}}), relations~(\ref{rel.1})--(\ref{rel.a}) and~(\ref{rel.4}) of Proposition~\ref{prop:rel}, and relation~(\ref{igual:2}) of Proposition~\ref{prop:rel.igual}, for $i=t+1,\ldots,n-1$, we have:
\begin{align}
\widehat{\phi}(a_{1}\sigma_{i})&=a_{1}a^{\alpha_{1,1}}_{n+1}C^{x_{1,1}}_{1,n+1}\cdots C^{x_{1,n}}_{n,n+1}\sigma_{i}a^{s_{i,1}}_{n+1}b^{s_{i,2}}_{n+1}C^{r_{i,1}}_{1,n+1}\cdots C^{r_{i,n}}_{n,n+1}\notag\\
% &=a_{1}\sigma_{i}a^{\alpha_{1,1}+s_{i,1}}_{n+1}b^{s_{i,2}}_{n+1}C^{x_{1,1}+r_{i,1}}_{1,n+1}\cdots C^{x_{1,i}+r_{i,i}+r_{i,i+1}}_{i,n+1}C^{x_{1,i+1}+r_{i,i+1}}_{i+1,n+1}C^{x_{1,i+2}+r_{i,i+2}+r_{i,i+1}}_{i+2,n+1}\cdots C^{x_{1,n}+r_{i,n}}_{n,n+1}\notag
%\intertext{\comj{shouldn't the above line be as follows? (shifting $\sigma_{i}$ to the left only introduces terms of the form $x_{\cdot,\cdot}$).} {\textcolor{blue}{!!!C!!! sim, voce tem razao}}}
&=a_{1}\sigma_{i}a^{\alpha_{1,1}+s_{i,1}}_{n+1}b^{s_{i,2}}_{n+1} w C_{i,n+1}^{x_{1,i+1}} C_{i+2,n+1}^{x_{1,i+1}},\label{eq:sia1a}
%C^{x_{1,1}+r_{i,1}}_{1,n+1}\cdots C_{i,n+1}^{x_{1,i}+r_{i,i}+x_{1,i+1}} C_{i+1,n+1}^{x_{1,i+1}+r_{i,i+1}} C_{i+2,n+1}^{x_{1,i+2}+r_{i,i+2}+x_{1,i+1}}\cdots C^{x_{1,n}+r_{i,n}}_{n,n+1},\label{eq:sia1a}
\intertext{where $w=\prod_{k=1}^{n} C_{k,n+1}^{x_{1,k}+r_{i,k}}$ and}
\widehat{\phi}(\sigma_{i}a_{1})&=\sigma_{i}a^{s_{i,1}}_{n+1}b^{s_{i,2}}_{n+1}C^{r_{i,1}}_{1,n+1}\cdots C^{r_{i,n}}_{n,n+1}a_{1}a^{\alpha_{1,1}}_{n+1}C^{x_{1,1}}_{1,n+1}\cdots C^{x_{1,n}}_{n,n+1}\notag\\
&=\sigma_{i}a_{1}a^{s_{i,1}}_{n+1}b^{s_{i,2}}_{n+1}C^{s_{i,2}}_{1,n+1}C^{s_{i,2}}_{2,n+1} a^{\alpha_{1,1}}_{n+1}C^{r_{i,1}+x_{1,1}}_{1,n+1}\cdots C^{r_{i,n}+x_{1,n}}_{n,n+1}\notag\\
&=\sigma_{i}a_{1}a^{s_{i,1}+\alpha_{1,1}}_{n+1}b^{s_{i,2}}_{n+1} w C_{1,n+1}^{s_{i,2}\alpha_{1,1}+s_{i,2}}C_{2,n+1}^{s_{i,2}}. \label{eq:sia1b}
%&=\sigma_{i}a_{1}a^{s_{i,1}+\alpha_{1,1}}_{n+1}b^{s_{i,2}}_{n+1}C^{s_{i,2}\alpha_{1,1}+s_{i,2}+x_{1,1}+r_{i,1}}_{1,n+1}C^{s_{i,2}+x_{1,2}+r_{i,2}}_{2,n+1}C^{x_{1,3}+r_{i,3}}_{3,n+1}\cdots C^{x_{1,n}+r_{i,n}}_{n,n+1}\label{eq:sia1b}
\end{align}
Since $i\geq3$, using the fact that $a_{1}\sigma_{i}=\sigma_{i}a_{1}$ in ${B_{t,s,1}(M)}/{\Gamma_{3}(P_{n+1}(M))}$ by relation~(\ref{igual:3}) of Proposition~\ref{prop:rel.igual}, and comparing the coefficients of $C_{2,n+1}$
%and $C_{i,n+1}$ 
in~(\ref{eq:sia1a}) and~(\ref{eq:sia1b}) and making use of Proposition~\ref{prop:can.form}, for $i=t+1,\ldots,n-1$, it follows that:
\begin{equation}\label{eq:si2}
s_{i,2} \equiv 0 \bmod{2}. 
%r_{i,i+1}&\equiv 0 \bmod{2}. \comj{This depends on the previous comment above.} 
\end{equation}
%{\textcolor{blue}{!!!C!!! essa igualdade (4.27) n\~{a}o \'{e} valida. tem que retirar do texto. Mas isso n\~{a}o compromete o resto do texto, essa equacao n\~{a}o foi usado no resto das demonstracoes, a parte importante \'{e} $s_{i,2}\equiv 0$.}}
In a similar manner, analysing the image by $\widehat{\phi}$ of the relation $b_{1}\sigma_{i}=b_{1}\sigma_{i}$ for $i=t+1,\ldots,n-1$, using~(\ref{eq:si2}), and comparing the coefficients of $C_{2,n+1}$, we see that $s_{i,1}\equiv 0 \bmod{2}$. 

Now suppose that $i=1,\ldots, t-1$. Analysing the image by $\widehat{\phi}$ of the relation $a_{n}\sigma_{i}=\sigma_{i}a_{n}$ (resp.\ $b_{n}\sigma_{i}=\sigma_{i}b_{n}$), we have:
%Also, by relations % $a_{t+1}\sigma_{i}=\sigma_{i}a_{t+1}$ (resp.\ $b_{t+1}\sigma_{i}=\sigma_{i}b_{t+1}$) and
% $a_{n}\sigma_{i}=\sigma_{i}a_{n}$ (resp.\ $b_{n}\sigma_{i}=\sigma_{i}b_{n}$) for $i=1,\ldots, t-1$, we have: 
% \comj{As mentioned below, the computations for $a_{t+1}\sigma_{i}=\sigma_{i}a_{t+1}$ and  $b_{t+1}\sigma_{i}=\sigma_{i}b_{t+1}$ (if we use them) should probably be given.} \textcolor{blue}{!!!C!!! nao precisa usar a relacao $a_{t+1}\sigma_{i}=\sigma_{i}a_{t+1}$, usamos apenas a relacao que estao descrita abaixo} :
\begin{align}
\widehat{\phi}(a_{n}\sigma_{i})&=a_{n}a^{\alpha_{n,1}}_{n+1}C^{x_{n,1}}_{1,n+1}\cdots C^{x_{n,n}}_{n,n+1}\sigma_{i}a^{s_{i,1}}_{n+1}b^{s_{i,2}}_{n+1}C^{r_{i,1}}_{1,n+1}\cdots C^{r_{i,n}}_{n,n+1}\notag\\
&=a_{n}\sigma_{i}a^{\alpha_{n,1}+s_{i,1}}_{n+1}b^{s_{i,2}}_{n+1} w C_{i,n+1}^{x_{n,i+1}} C_{i+2,n+1}^{x_{n,i+1}},\label{eq:siana}
\intertext{where $w=\prod_{k=1}^{n} C_{k,n+1}^{x_{n,k}+r_{i,k}}$ and}
\widehat{\phi}(\sigma_{i}a_{n})&=\sigma_{i}a^{s_{i,1}}_{n+1}b^{s_{i,2}}_{n+1}C^{r_{i,1}}_{1,n+1}\cdots C^{r_{i,n}}_{n,n+1}a_{n}a^{\alpha_{n,1}}_{n+1}C^{x_{n,1}}_{1,n+1}\cdots C^{x_{n,n}}_{n,n+1}\notag\\
&=\sigma_{i}a_{n}a^{s_{i,1}}_{n+1}b^{s_{i,2}}_{n+1}C^{s_{i,2}}_{n,n+1}a^{\alpha_{n,1}}_{n+1}C^{r_{i,1}+x_{n,1}}_{1,n+1}\cdots C^{r_{i,n}+x_{n,n}}_{n,n+1}\notag\\
&=\sigma_{i}a_{n}a^{s_{i,1}+\alpha_{n,1}}_{n+1}b^{s_{i,2}}_{n+1} w C_{1,n+1}^{s_{i,2}\alpha_{n,1}}C_{n,n+1}^{s_{i,2}}. \label{eq:sianb}
\end{align}
Since $i\leq t-1$ and $t+1<n$, we have $i+2<n$, and comparing the coefficient of $C_{n,n+1}$ in~(\ref{eq:siana}) and~(\ref{eq:sianb}), we conclude that $s_{i,2} \equiv 0 \bmod{2}$. In a similar manner, we obtain $s_{i,1}\equiv 0 \bmod{2}$ using the relation $b_{n}\sigma_{i}=\sigma_{i}b_{n}$. In summary,  for $i=1,\ldots, t-1$, we have: 
% \comj{I take it that $s_{i,1}\equiv 0 \bmod{2}$ comes from the relation $b_{n}\sigma_{i}=\sigma_{i}b_{n}$?} {\textcolor{blue}{!!!C!!! Sim}}
%we conclude by comparing the coefficient of $C_{t+1,n+1}$ that for $i=1,\ldots, t-1, t+1,\ldots,n-1$:
\begin{equation}\label{eq:si1i2}
s_{i,1}\equiv s_{i,2} \equiv 0 \bmod{2},
\end{equation}
which proves part~(\ref{it:lg2}) of the statement. %\comj{For $a_{t+1}\sigma_{i}=\sigma_{i}a_{t+1}$, I think we have $\widehat{\phi}(a_{t+1}\sigma_{i})=a_{t+1} \sigma_{i} a_{n+1}^{\alpha_{t+1,1}+s_{i,1}} b_{n+1}^{s_{i,2}} w C_{i,n+1}^{x_{t+1},k} C_{i+2,n+1}^{x_{t+1},k}$ and $\widehat{\phi}(\sigma_{i}a_{t+1})=\sigma_{i}a_{t+1} a_{n+1}^{\alpha_{t+1,1}+s_{i,1}} b_{n+1}^{s_{i,2}} w C_{1,n+1}^{s_{i,2} \alpha_{t+1,1}} C_{t+1,n+1}^{s_{i,2}} C_{t+2,n+1}^{s_{i,2}}$, where $w=\prod_{k=1}^{n} C_{k,n+1}^{x_{t+1,k}+r_{i,k}}$. If so, then comparing the coefficients of $C_{t+1,n+1}$ for $1\leq i \leq t-2$, we obtain $s_{i,2}\equiv 0 \bmod{2}$, but if $i=t-1$ then $t+1=i+2$, and thus $s_{i,2}\equiv x_{t+1,k} \bmod{2}$. Do we know that $x_{t+1,k} \equiv 0\bmod{2}$? Or maybe we obtain the result by looking at the coefficients of $C_{t+2,n+1}$? If necessary, check also for $b_{t+1}\sigma_{i}=\sigma_{i}b_{t+1}$. In any case, I suggest writing out the expressions for $\widehat{\phi}(\cdot)$ in canonical form, without necessarily giving all the computations.}

%{\textcolor{blue}{!!!C!!! Ok, no caso do $i=t-1$ daria certo olhar o expoente do $C_{t+2,n+1}$, mas para faciliar a escrita, usei a relacao $a_{n}\sigma_{i}=\sigma_{i}a_{n}$, $i=1,\ldots,t-1$, pois assim basta olhar sempre o mesmo elemento, o elemento $C_{n,n+1}$, a conta ta acima, e \'{e} bem analoga ao caso com o $a_{1}$}} 
%\comj{I am not sure whether this means that the above computations are sufficient or not?} \textcolor{blue}{!!!C!!! sim, \'{e} o suficiente, pois com a equacao (4.27) e (4.30) concluimos o resultado para todo $i\neq t$. Se usasse a relacao $a_{t+1}\sigma_{i}=\sigma_{i}a_{t+1}$ como estava descrito anteriormente, daria para chegar na mesma conclusao, mas teria mais passos na justificativa, entao troquei a desmonstracao e usei a relacao $a_{n}\sigma_{i}=\sigma_{i}a_{n}$ pois a desmonstracao fica mais curta.}

We now return to the proof of the outstanding cases of part~(\ref{it:lg1}), as well as that of part~(\ref{it:lg3}). We first study relation~(\ref{igual:6}) of Proposition~\ref{prop:rel.igual}. 
% \comj{We can deal with the two cases at once up to the point where we compare coefficients (the arguments are a little different).} {\textcolor{blue}{!!!C!!! as contas para chegar ate (4.31) e (4.32) sao identicas, mas vou especificar melhor a comparacao dos expoentes}} 
Let $i\in \{1,t+1\}$. Setting $w=\displaystyle \prod_{\substack{k=1\\ k\neq i+1}}^{n}C^{x_{i,k}+y_{i,k}}_{k,n+1}$, using~(\ref{rel.1})--(\ref{rel.a}) and~(\ref{rel.4}) of Proposition~\ref{prop:rel}, relation~(\ref{igual:2}) of Proposition~\ref{prop:rel.igual} in ${B_{t,s,1}(M)}/{\Gamma_{3}(P_{n+1}(M))}$, and equations~(\ref{eq:alpha_{i,2}c}),~(\ref{eq:beta_i,1.toro}),~(\ref{eq:beta_{i,1}}) and~(\ref{eq:si1i2}),
%\textcolor{green}{Seria aqui um bom lugar para colocar ; "and item (b)"?}
%\comj{As you mentioned in the e-mail, the equations are sufficient.} 
we have: 
% \comc{a partir daqui, estamos usando o item (b), pois na imagem de $\sigma_{i}$ por $\hat{\phi}$ n\~{a}o aparece $a^{s_{i,1}}_{n+1}b^{s_{i,2}}_{n+1}$, deveria citar o item?} \comj{Maybe yes.}
%\begin{align}
%\phi(\sigma_{1}a_{1}\sigma^{-1}_{1}b_{1})&=\sigma_{1}a_{1}a^{\alpha_{1,1}}_{n+1}C^{x_{1,2}}_{2,n+1}\sigma^{-1}_{1}b_{1}b^{\beta_{1,2}}_{n+1}C^{y_{1,2}}_{2,n+1}w\notag\\
%&=\sigma_{1}a_{1}\sigma^{-1}_{1}a^{\alpha_{1,1}}_{n+1}C^{x_{1,2}}_{1,n+1}C^{x_{1,2}}_{2,n+1}C^{x_{1,2}}_{3,n+1}b_{1}b^{\beta_{1,2}}_{n+1}C^{y_{1,2}}_{2,n+1}w\notag\\
%&=\sigma_{1}a_{1}\sigma^{-1}_{1}b_{1}a^{\alpha_{1,1}}_{n+1}b^{\beta_{1,2}}_{n+1}C^{\alpha_{1,1}+x_{1,2}+y_{1,2}}_{1,n+1}C^{\alpha_{1,1}+x_{1,2}}_{2,n+1}C^{x_{1,2}}_{3,n+1}w\notag\\
%\end{align}
%{\color{blue}!!C!!MAY22 na conta acima o $y_{1,2}$ que apareceu estava errado, ele deveria ser o expoente do $C_{2,n+1}$ e nao do $C_{1,n+1}$, pois na ultima linha o que foi feito %foi passar o $a_{n+1}$ pelo $b_{1}$, com a relac\~{a}o  $a_{n+1}b_{1}=b_{1}a_{n+1}C_{1,n+1}C^{-1}_{2,n+1}$ a conta certa esta abaixo:
\begin{align*}
\widehat{\phi}(\sigma_{i}a_{i}\sigma_{i}^{-1}b_{i})&=\sigma_{i}a_{i}a_{n+1}^{\alpha_{i,1}} C_{i+1,n+1}^{x_{i,i+1}} \sigma^{-1}_{i} b_{i} b_{n+1}^{\beta_{i,2}} C_{i+1,n+1}^{y_{i,i+1}}w\notag\\
&=\sigma_{i}a_{i} \sigma^{-1}_{i} a_{n+1}^{\alpha_{i,1}} C_{i,n+1}^{x_{i,i+1}} C_{i+1,n+1}^{x_{i,i+1}} C_{i+2,n+1}^{x_{i,i+1}} b_{i} b_{n+1}^{\beta_{i,2}} C_{i+1,n+1}^{y_{i,i+1}}w\notag\\
&=\sigma_{i}a_{i} \sigma^{-1}_{i} b_{i}  a_{n+1}^{\alpha_{i,1}} b_{n+1}^{\beta_{i,2}} C_{i,n+1}^{x_{i,i+1}+\alpha_{i,1}} C_{i+1,n+1}^{x_{i,i+1}+\alpha_{i,1}+y_{i,i+1}} C_{i+2,n+1}^{x_{i,i+1}}  w%\label{eq:sasb}
\intertext{and}
\widehat{\phi}(b_{i}\sigma_{i}a_{i}\sigma_{i})&=b_{i}b^{\beta_{i,2}}_{n+1}C_{i+1,n+1}^{y_{i,i+1}}\sigma_{i} C_{i+1,n+1}^{r_{i,i+1}} a_{i} a_{n+1}^{\alpha_{i,1}} C_{i+1,n+1}^{x_{i,i+1}} \sigma_{i} C_{i+1,n+1}^{r_{i,i+1}}w\notag\\
&=b_{i} \sigma_{i} b_{n+1}^{\beta_{i,2}} C_{i,n+1}^{y_{i,i+1}} C_{i+1,n+1}^{y_{i,i+1}+r_{i,i+1}} C_{i+2,n+1}^{y_{i,i+1}} 
a_{i} a_{n+1}^{\alpha_{i,1}} C_{i+1,n+1}^{x_{i,i+1}} \sigma_{i} C_{i+1,n+1}^{r_{i,i+1}}w\notag\\
&=b_{i} \sigma_{i} a_{i} b_{n+1}^{\beta_{i,2}} C_{i,n+1}^{y_{i,i+1}+\beta_{i,2}} C_{i+1,n+1}^{y_{i,i+1}+r_{i,i+1}+\beta_{i,2}} C_{i+2,n+1}^{y_{i,i+1}} a_{n+1}^{\alpha_{i,1}} C_{i+1,n+1}^{x_{i,i+1}} \sigma_{i} C_{i+1,n+1}^{r_{i,i+1}}w\notag\\
&=b_{i} \sigma_{i} a_{i} a_{n+1}^{\alpha_{i,1}} b_{n+1}^{\beta_{i,2}} C_{1,n+1}^{\alpha_{i,1} \beta_{i,2}} C_{i,n+1}^{y_{i,i+1}+\beta_{i,2}} C_{i+1,n+1}^{y_{i,i+1}+r_{i,i+1}+\beta_{i,2}+x_{i,i+1}} C_{i+2,n+1}^{y_{i,i+1}} \sigma_{i} C_{i+1,n+1}^{r_{i,i+1}}w\notag\\
&=b_{i} \sigma_{i} a_{i} \sigma_{i} a_{n+1}^{\alpha_{i,1}} b_{n+1}^{\beta_{i,2}} C_{1,n+1}^{\alpha_{i,1} \beta_{i,2}} 
C_{i,n+1}^{r_{i,i+1}+x_{i,i+1}} 
C_{i+1,n+1}^{y_{i,i+1}+\beta_{i,2}+x_{i,i+1}} 
C_{i+2,n+1}^{r_{i,i+1}+\beta_{i,2}+x_{i,i+1}} w.%\label{eq:bsas}
\end{align*}
Using the equalities $\sigma_{i}a_{i}\sigma_{i}^{-1}b_{i}=b_{i}\sigma_{i}a_{i}\sigma_{i}$ in ${B_{t,s,1}(M)}/{\Gamma_{3}(P_{n+1}(M))}$ by relation~(\ref{igual:6}) of Proposition~\ref{prop:rel.igual} and $\widehat{\phi}(\sigma_{i}a_{i}\sigma_{i}^{-1}b_{i})= \widehat{\phi}(b_{i}\sigma_{i}a_{i}\sigma_{i})$, and applying Proposition~\ref{prop:can.form}, we see that:
\begin{equation}\label{eq:ccmod2}
C_{i,n+1}^{x_{i,i+1}+\alpha_{i,1}} C_{i+1,n+1}^{x_{i,i+1}+\alpha_{i,1}+y_{i,i+1}} C_{i+2,n+1}^{x_{i,i+1}}= C_{1,n+1}^{\alpha_{i,1} \beta_{i,2}} 
C_{i,n+1}^{r_{i,i+1}+x_{i,i+1}} C_{i+1,n+1}^{y_{i,i+1}+\beta_{i,2}+x_{i,i+1}} C_{i+2,n+1}^{r_{i,i+1}+\beta_{i,2}+x_{i,i+1}}.
\end{equation}
Comparing certain coefficients of~\reqref{ccmod2} modulo $2$, we obtain the following congruences:
\begin{enumerate}[(i)]
\item\label{it:cong2i} for the coefficient of $C_{i+1,n+1}$, $\alpha_{i,1} \equiv \beta_{i,2} \bmod{2}$.
\item\label{it:cong2ii} if $i=1$ (resp.\ $i=t+1$), for the coefficient of $C_{3,n+1}$ (resp.\ of $C_{t+1,n+1}$), we have $r_{i,i+1} \equiv \beta_{i,2} \bmod{2}$, where we use~(\ref{it:cong2i}) in the case $i=t+1$.
\item for the coefficient of $C_{1,n+1}$, $\alpha_{i,1}^{2} \equiv 0 \bmod{2}$ using~(\ref{it:cong2i}) and~(\ref{it:cong2ii}).
\end{enumerate}
% see that , and comparing the coefficients of $C_{3,n+1}$, if $i=1$ (resp.\ $C_{t+1,n+1}$, if $i=t+1$), we obtain $r_{i,i+1} \equiv \beta_{i,2}$ and then comparing the coefficients of $C_{1,n+1}$, with $i=1,t+1$, we obtain $\alpha_{i,1}^{2} \equiv 0$, and thus 
It follows that $r_{i,i+1} \equiv \alpha_{i,1}\equiv \beta_{i,2}\equiv 0 \bmod{2}$ for $i\in \{ 1,t+1\}$. We thus obtain part~(\ref{it:lg3}), and applying~(\ref{eq:alphabeta}),~(\ref{eq:alphabeta1}) and induction on $t+1\leq i\leq n$, the remaining congruences of part~(\ref{it:lg1}) follow. 
\end{proof}

Part~(\ref{it:rij1}) of the following lemma is the analogue in our setting of \relem{phiartin}(\ref{it:phiartin0}).

\begin{lem}\label{lem:rij}
Let $M=\mathbb T$ or $\mathbb K$. Using the notation of~(\ref{eq:phi1}), we have: 
\begin{enumerate}[(a)] 
\item\label{it:rij1} if $t\geq 4$ then $r_{j,k}\equiv 0\bmod{2}$ for all $1\leq j\leq t-1$ and $k=2,\ldots, j-1,j+3,\ldots, t$. 
% \comj{I combined the previous parts~(a) and~(b) so that the statement is similar to \relem{phiartin}(\ref{it:phiartin0}).} \comc{ok}

%\item\label{it:rij1} $r_{j,i}\equiv 0\bmod{2}$ if $2\leq i < j \leq t-1$.
%\item\label{it:rij2} $r_{i,j}\equiv 0 \bmod{2}$ if $i\geq 1$ and $i+3\leq j \leq t$.
\item\label{it:rij3} if $t\geq 3$ then $r_{i+1,i+2}\equiv r_{i,i+1}\equiv r_{1,2}\equiv 0 \bmod{2}$ if $1\leq i \leq t-2$. 
\end{enumerate}
\end{lem}

\begin{proof}
%{\color{green!50!black}The proof of~(\ref{it:rij1}) and~(\ref{it:rij2}) $\bmod{2}$ \comj{How?}.}\comc{7junho - sugestao abaixo}
Let $M=\mathbb T$ or $\mathbb K$. 
\begin{enumerate}[(a)]
\item Let $t\geq 4$. Let $1\leq i,j\leq t-1$ be such that $\abs{i-j}\geq 2$. Applying relation~(\ref{rel.3}) of \repr{rel} to~(\ref{eq:commba1}) in the case $m=1$, we have:
\begin{equation*}
C^{\gamma}_{1,n+1}=C_{i,n+1}^{r_{j,i+1}} C_{i+1,n+1}^{-2r_{j,i+1}} C_{i+2,n+1}^{r_{j,i+1}}
C_{j,n+1}^{-r_{i,j+1}} C_{j+1,n+1}^{2r_{i,j+1}} C_{j+2,n+1}^{-r_{i,j+1}},
\end{equation*}
which is in canonical form (possibly up to permutation of some of the factors). Comparing the coefficients of $C_{i+1,n+1}$ (resp.\ $C_{j+1,n+1}$) if $i<j$ (resp.\ $i>j$) and using \repr{can.form}, we see that $r_{j,i+1}\equiv 0\bmod{2}$ (resp.\ $r_{i,j+1}\equiv 0\bmod{2}$). So for all $1\leq j\leq t-1$, $r_{j,k}\equiv 0\bmod{2}$ for all $k=2,\ldots,j-1$ (resp.\ for all $k=j+3,\ldots,t$) as required.

%\textcolor{blue}{The proof of~(\ref{it:rij1}) and~(\ref{it:rij2}) uses Lemma~\ref{lem:exp.M1}, with $m=1$. If $t\geq 4$ , $|i-j|\geq2$, applying relation~(\ref{rel.3}) of \repr{rel}, and~\req{commba1},we have
%\begin{equation*}
%C^{\gamma}_{1,n+1}=C_{i,n+1}^{r_{j,i+1}} C_{i+1,n+1}^{-2r_{j,i+1}} C_{i+2,n+1}^{r_{j,i+1}}
%C_{j,n+1}^{-r_{i,j+1}} C_{j+1,n+1}^{2r_{i,j+1}} C_{j+2,n+1}^{-r_{i,j+1}}
%\end{equation*}
%for some $\gamma\in\mathbb Z$, which is the canonical form, and comparing the coefficients of $C_{i,n+1}$ and using Proposition~\ref{prop:can.form}, taking this relation modulo $2$, we obtain~(\ref{it:rij1}). Also, comparing the coefficients of $C_{j+2,n+1}$ (resp. $C_{i+2,n+1}$) if $i<j$ (resp. $i>j$), we obtain~(\ref{it:rij2}).}

To prove part~(\ref{it:rij3}), let $t\geq 3$, and let $1\leq i \leq t-2$. Then comparing the coefficients of $C_{i+1,n+1}$ in ~\req{commba2} and using relation~(\ref{rel.3}) of \repr{rel} and the lift of relation~(\ref{it:full2}) of Theorem~\ref{th:total}, we see that $\rho \equiv 0 \bmod{2}$ by \repr{can.form}
% \comj{Since we are looking at the coefficients of $C_{i+1,n+1}$, we should apply \repr{can.form}, and so don't we have $\rho \equiv 0 \bmod{2}$ (and not $\rho=0$?)}\comc{ok, alterei}, and taking this relation modulo $2$, 
and it follows that $r_{i+1,i+2}\equiv r_{i,i+1} \bmod{2}$. Part~(\ref{it:rij3}) is then a consequence of Lemma~\ref{lem:lemgen}(\ref{it:lg3}) and induction on $i$.\qedhere
\end{enumerate}
% If $1\leq i \leq t-2$, to prove~(\ref{it:rij3}) we use equations~(\ref{eq:v1canon}) and~(\ref{eq:v2canon}).
% if $M=\mathbb{T}$ (resp.\ $M=\mathbb{K}$). Since the canonical form of $v_k$, $k=1,2$, given in~(\ref{eq:v1canon}) and~(\ref{eq:v2canon}) does not contain $a_{n+1}$ or $b_{n+1}$, we have $v_1=v_2$. Comparing the coefficients of $C_{i+1,n+1}$ and using the lift of  relation~(\ref{it:full2}) of Theorem~\ref{th:total}, we deduce that:
%\begin{equation}
%r_{i+1,i+2}\equiv r_{i,i+1} \bmod{2}\label{eq:ar2a}\\
%\end{equation}
%and the resul follows from Lemma~\ref{lem:lemgen}(\ref{it:lg3}).}
\end{proof}

We end this paper by proving Theorem~\ref{th:split3}.

\begin{proof}[Proof of Theorem~\ref{th:split3}] 
Let $M$ be the $2$-torus or the Klein bottle, and let $t,s\geq 2$. Suppose on the contrary that the projection $p_{\ast}\colon\thinspace B_{t,s,1}(M)\longrightarrow B_{t,s}(M)$ admits a section. As we showed earlier in this section, we may make use of the framework of Section~\ref{sec:contas}, and so we use the notation defined there. Taking $m=1$ in~\reqref{final} and applying Proposition~\ref{prop:can.form}, we see that:
\begin{equation}\label{eq:coeffcaseb}
\gamma+1\equiv\alpha + \delta \bmod{2}
\end{equation}
To obtain a contradiction to~\reqref{coeffcaseb}, it suffices to prove that $\alpha,\gamma$ and $\delta$ are even. Applying Lemma~\ref{lem:rij}(\ref{it:rij3}) to~\reqref{defalpha}, we see that $\alpha$ is even. By~\reqref{defdelta}, we have $\delta\equiv \beta_{1,2}+\alpha_{1,1} \bmod{2}$, and from~\reqref{defgammai} and~\reqref{defgamma}, $\gamma= -\sum_{i=t+1}^{n} \alpha_{i,1}$ (resp.\ $\gamma= \sum_{i=t+1}^{n} (\alpha_{i,2}-\alpha_{i,1})$) if $M=\mathbb{T}$ (resp.\ if $M=\mathbb{K}$). It follows from \relem{lemgen}(\ref{it:lg1}) that $\delta$ and $\gamma$ are even too. This contradicts~\reqref{coeffcaseb}, which proves Theorem~\ref{th:split3}. 
\end{proof}

\end{document}